\documentclass[11pt,fleqn,a4paper]{article}

\usepackage{amsmath,amssymb,amsthm}
\usepackage[mathscr]{eucal}

\usepackage{hyperref}
\hypersetup{colorlinks, linkcolor=blue, citecolor=blue, urlcolor=blue}

\newcommand{\p}{{\partial}}
\newcommand{\lsemioplus}{\mathbin{\mbox{$\lefteqn{\hspace{.77ex}\rule{.4pt}{1.2ex}}{\in}$}}}

\newtheorem{theorem}{Theorem}
\newtheorem{lemma}[theorem]{Lemma}
\newtheorem{corollary}[theorem]{Corollary}
\newtheorem{proposition}[theorem]{Proposition}
\newtheorem*{proposition*}{Proposition}
{\theoremstyle{definition}
\newtheorem{definition}[theorem]{Definition}

\newtheorem{remark}[theorem]{Remark}

\newtheorem*{notation}{Notation and convention}
}

\newcommand{\todo}[1][\null]{\ensuremath{\clubsuit}}

\newcommand{\noprint}[1]{}

\begin{document}
\par\noindent {\LARGE\bf
Admissible transformations and Lie symmetries\\ of linear systems of second-order ordinary differential\\ equations
\par}

\vspace{4mm}\par\noindent
{\large Vyacheslav M.\ Boyko$^{ab}$, Oleksandra V.\ Lokaziuk$^{ac}$ and Roman O.\ Popovych$^{ade}$
\par\vspace{2mm}\par}

\par\noindent{\it\small
$^a$\,Institute of Mathematics of NAS of Ukraine, 3 Tereshchenkivska Str., 01024 Kyiv, Ukraine
}

\par\smallskip\par\noindent{\it\small
$^b$\,Department of Mathematics, Kyiv Academic University, 36 Vernads'koho Blvd, 03142 Kyiv, Ukraine
\par}

\par\smallskip\par\noindent{\it\small
$^c$\,Borys Grinchenko Kyiv Metropolitan University, 18/2 Bulvarno-Kudryavska Str., 04053 Kyiv, Ukraine
\par}

\par\smallskip\par\noindent{\it\small
$^d$\,Mathematical Institute, Silesian University in Opava, Na Rybn\'\i{}\v{c}ku 1, 746 01 Opava, Czech Republic
\par}

\par\smallskip\par\noindent {\it\small
$^e$\,Fakult\"at f\"ur Mathematik, Universit\"at Wien, Oskar-Morgenstern-Platz 1, 1090 Wien, Austria
}

{\vspace{2mm}\par\noindent {\it
\textup{E-mail:} boyko@imath.kiev.ua, o.lokaziuk@kubg.edu.ua, rop@imath.kiev.ua
}\par}

\vspace{7mm}\par\noindent\hspace*{10mm}\parbox{140mm}{\small
We revisit the results on admissible transformations
between normal linear systems of second-order ordinary differential equations
with an arbitrary number of dependent variables
under several appropriate gauges of the arbitrary elements parameterizing these systems.
For each class from the constructed chain of nested gauged classes of such systems, we single out its singular subclass,
which appears to consist of systems being similar to the elementary (free particle) system
whereas the regular subclass is the complement of the singular one.
This allows us to exhaustively describe the equivalence groupoids of the above classes
as well as of their singular and regular subclasses.
Applying various algebraic techniques, we establish principal properties of Lie symmetries of the systems under consideration
and outline ways for completely classifying these symmetries.
In particular, we compute the sharp lower and upper bounds for the dimensions of the maximal Lie invariance algebras possessed by systems
from each of the above classes and subclasses.
We also show how equivalence transformations and Lie symmetries can be used
for reduction of order of such systems and their integration.
As an illustrative example of using the theory developed,
we solve the complete group classification problems for all these classes
in the case of two dependent variables.\looseness=-1
}\par\vspace{7mm}

\noprint{

}

\noindent
{\small\emph{Keywords}:
group classification of differential equations,
linear systems of second-order ordinary differential equations,
algebraic method of group classification,
Lie symmetry,
equivalence groupoid,
normalized class of differential equations,
equivalence group,
equivalence algebra

}

\section{Introduction}\label{sec:Introduction}

Transformational properties and Lie symmetries of ordinary differential equations are classical objects of study
\cite{berk2002A,blum2002A,blum1989A,ibra1992a,ibra1999A,ibra2010A,lie1893A,maho2007a,olve1993A,olve1995A,schw2000a,schw2008A,wilc1906A}
but this is not the case for normal systems of ordinary differential equations, not to mention general systems of such equations.
The most studied are the systems of ordinary differential equations of the same order~$r$.
First reasonable upper bounds of the dimensions of maximal Lie invariance algebras of these systems
were derived in~\cite{gonz1983a,gonz1985a}.
The least upper bounds were initially found exclusively for the particular cases
$r=2$ (see
\cite[pp.~68--69, Theorem~44]{mark1960A} for the preliminary consideration
as well as \cite[Sections~4 and~5]{gonz1983a} and \cite{fels1993A,fels1995a} for further enhancements)
and~$r=3$~\cite{fels1993A}.
These bounds are equal to $n^2+4n+3$ and $n^2+3n+3$, and, up to the general point equivalence,
they are attained only on the elementary systems
${\rm d}^2\boldsymbol x/{\rm d}t^2=\boldsymbol0$ and ${\rm d}^3\boldsymbol x/{\rm d}t^3=\boldsymbol0$, respectively~\cite{fels1993A}.
Here and in what follows $n$ is the number of equations in systems, $n\geqslant2$,
and $\boldsymbol x(t)=\big(x^1(t),\dots,x^n(t)\big){}^{\mathsf T}$
is the unknown vector-valued function of the independent variable~$t$.
See also \cite{gonz1988d} for the linear systems with $r=2$,
\cite{amin2006a,amin2010a,
merk2006a
} for the general systems with $r=2$
and~\cite{medv2011a} for $r=3$.
So-called fundamental invariants for systems with $r=2$ were first computed in~\cite{fels1993A,fels1995a}.
Essentially later, such invariants for $r=3$ and $r\geqslant4$ were obtained in~\cite{medv2011a} and~\cite{doub2014b},
respectively.
Results of the latter paper imply
that the least upper bound of the dimensions of maximal Lie invariance algebras of systems with fixed $r\geqslant4$
is equal to $n^2+rn+3$, and, up to the general point equivalence,
it is attained only on the elementary system ${\rm d}^r\boldsymbol x/{\rm d}t^r=\boldsymbol0$.
For single ordinary differential equations, i.e., $n=1$,
the analogous least upper bounds were computed by Sophus Lie himself for arbitrary~$r$ \cite[S.~294--301]{lie1893A}
although this Lie's result was not well known and was repeatedly re-obtained;
see the discussion in the introduction of~\cite{boyk2015a}.

As a summary of the above results, the following theorem holds true.

\begin{theorem}\label{thm:SODEsMaxDim}
For an arbitrary $n\in\mathbb N$,
the least upper bound of the dimensions of maximal Lie invariance algebras
of normal systems of $n$ ordinary differential equations of the same order~$r$ is equal to
$n^2+4n+3$ if $r=2$ and $n^2+rn+3$ if $r\geqslant3$.
This bound is attained only by systems that are similar with respect to point transformations to
the elementary system ${\rm d}^r\boldsymbol x/{\rm d}t^r=\boldsymbol0$.
\end{theorem}

We further discuss the linear system of the above kind with $r\geqslant3$
in Section~\ref{sec:OnGeneralizationToArbitraryEqOrder}.
Note that the study of equivalence problem for classes of linear system of ordinary differential equations with $n>1$,
the computation of equivalence groups of such classes
and the construction of differential invariants of subgroups of these groups
was initiated by Wilczynski~\cite{wilc1906A}.
Problems and results in the above framework are known to have a geometric interpretation.
In particular, the group classification of linear normal systems of $n$ ordinary differential equations of order~$r$
is equivalent to the classification of curves in the Grassmann variety~${\rm Gr}_n(\mathbb F^{rn})$
over the complex or real field~$\mathbb F$, see, e.g.,~\cite{seas1993a}.

Normal systems of $n$ ordinary differential equations of the same order~$r=2$ over the field~$\mathbb F$,
\begin{gather}\label{SODEs}
\boldsymbol x_{tt}=\boldsymbol F(t,\boldsymbol x,\boldsymbol x_t),
\end{gather}
are the most studied in other aspects as well;
here
$\boldsymbol x_t ={\rm d}\boldsymbol x/{\rm d}t$,
$\boldsymbol x_{tt}={\rm d}^2\boldsymbol x/{\rm d}t^2$,
and $\boldsymbol F$ is an arbitrary (sufficiently smooth) vector-valued function of $(t,\boldsymbol x,\boldsymbol x_t)$.
They are the systems of ordinary differential equations that naturally arise in a number of contexts
in various fields of mathematics and its applications,
which includes equations for geodesics on manifolds in differential geometry, the calculus of variations,
Newtonian equations of motion in classical mechanics, e.g., in the classical $n$-body problem, to mention a few.
Some Lie-symmetry and transformational properties of systems of the form~\eqref{SODEs}
were described due to the relation of such systems to certain (Riemannian, parabolic, Cartan, etc.)
geometries when studying these geometries,
see, e.g., \cite{case2013a,krug2017a} and references therein.
Thus, these results imply the following theorem via applying the method of filtered deformation~\cite{krug2021z}.

\begin{theorem}\label{thm:SODEsSubMaxDim}
\looseness=-1
The submaximum dimension of the maximal Lie invariance algebras of systems of the form~\eqref{SODEs} with $n\geqslant2$
is equal to $n^2+5$,
and it is attained only by systems of the form~\eqref{SODEs}
that are similar with respect to  point transformations in the space $\mathbb F_t\times\mathbb F^n_{\boldsymbol x}$
to the (nonlinearizable) system
\begin{gather}\label{eq:SODECanonicalSystemOfSubmaxDimOfMIA}
x^1_{tt}=(x^2_t)^3,\quad x^2_{tt}=0,\quad \dots,\quad x^n_{tt}=0.
\end{gather}
\end{theorem}

The system~\eqref{eq:SODECanonicalSystemOfSubmaxDimOfMIA} was presented in~\cite[Example~1]{case2013a} for $n=2$ 
and in~\cite[Proposition~5.3.2]{krug2017a} for general $n\geqslant2$, 
see also~\cite[Section~5.2]{krug2022a}.
A basis of the maximal Lie invariance algebra of the system~\eqref{eq:SODECanonicalSystemOfSubmaxDimOfMIA} consists of the vector fields
\begin{gather*}
\p_{x^a},\quad t\p_{x^a}+\tfrac32\delta_{a2}(x^2)^2\p_{x^1},\quad x^b\p_{x^c},\,_{b\ne1,\,c\ne2},\quad 3x^1\p_{x^1}+x^2\p_{x^2},
\\
\p_t,\quad t\p_t-x^1\p_{x^1},\quad t^2\p_t+tx^d\p_{x^d}+\tfrac12(x^2)^3\p_{x^1}.
\end{gather*}
Here and in what follows $a,b,c,d=1,\dots,n$, the summation convention is used for repeated indices,
and $\delta_{a2}$ denotes the Kronecker delta.
Although the system~\eqref{eq:SODECanonicalSystemOfSubmaxDimOfMIA} is not linearizable,
integrating it reduces to integrating (elementary) linear systems. 

Moreover, a similar theorem holds for the next possible less dimension $n^2+4$~\cite{krug2021z}.

\begin{theorem}\label{thm:SODEsSubSubMaxDim}\looseness=-1
Up to the point equivalence, there exists a unique system of the form~\eqref{SODEs} with $n\geqslant2$,
whose maximal Lie invariance algebra is $(n^2+4)$-dimensional.
This is, e.g., the linear system
\begin{gather}\label{eq:CanonicalSystemOfMaxDimOfMIAinL1}
x^1_{tt}=x^2,\quad x^2_{tt}=0,\quad \dots,\quad x^n_{tt}=0.
\end{gather}
\end{theorem}

The maximal Lie invariance algebra of the system~\eqref{eq:CanonicalSystemOfMaxDimOfMIAinL1} is spanned by the vector fields
\begin{gather*}
\p_t,\ \ t\p_t+2x^1\p_{x^1},\ \
x^b\p_{x^c},\,_{b\ne1,\,c\ne2},\ \ x^1\p_{x^1}+x^2\p_{x^2},\ \
\p_{x^a}+\tfrac12\delta_{a2}t^2\p_{x^1},\ \ t\p_{x^a}+\tfrac16\delta_{a2}t^3\p_{x^1}.
\end{gather*}
The system~\eqref{eq:CanonicalSystemOfMaxDimOfMIAinL1}
also realizes the submaximum dimension of the maximal Lie invariance algebras of linear systems of the form~\eqref{SODEs},
see Theorem~\ref{thm:MaxDimOfMIAinL1}.
In view of Theorem~\ref{thm:SODEsSubSubMaxDim},
it is not surprising that Egorov's system $\boldsymbol x_{tt}=2x^1x^1_tx^2_t\boldsymbol x_t$ \cite[Eq.~(5.11)]{krug2017a}
whose maximal Lie invariance algebra is $(n^2+4)$-dimensional as well \cite[Eq.~(5.12)]{krug2017a}
is linearized to the system~\eqref{eq:CanonicalSystemOfMaxDimOfMIAinL1} by the point transformation
$\tilde t=x^1$, $\tilde x^1=t+\frac12(x^1)^2x^2$, $\tilde x^a=x^a$, $a=2,\dots,n$.

Unfortunately, there are no general results on admissible transformations between systems of the form~\eqref{SODEs}
and on their Lie symmetries for an arbitrary~$n\geqslant2$,
beyond the ones discussed above.
Moreover, most of the particular results known for such systems are still related to their linearization.
Thus, the necessary and sufficient conditions for a system of the form~\eqref{SODEs} with $n=2$
to be linearizable by a point transformation were found in~\cite{bagd2010a}.
Classes of solvable and/or integrable and/or linearizable systems of the form~\eqref{SODEs},
which can be treated as dynamical systems for classical many-body problems in dimensions one, two and three,
were constructed in a series of monographes and papers by Calogero, see \cite{calo2001A,calo2018A} and references therein.
This is why even linear systems from the class~\eqref{SODEs} are attractive and still nontrivial  objects
to be considered from the point of view of their Lie symmetries and related transformational properties.
Group classification of such systems leads to examples
of Lie algebras of vector fields that are admitted by systems of the form~\eqref{SODEs} as their maximal Lie invariance algebras
and which dimensions of these algebras are possible beyond the known maximal and submaximum dimension values.
It is also important for solving the problem on linearization and integration of systems from the class~\eqref{SODEs}.
\looseness=-1

In present paper, we comprehensively study, within the framework of group analysis of differential equations,
the class of linear systems of the form~\eqref{SODEs}, i.e.,
the class~$\bar{\mathcal L}$ of normal linear systems of $n$ second-order ordinary differential equations,
\begin{gather}\label{SLODE}
\boldsymbol x_{tt}=A(t)\boldsymbol x_t+B(t)\boldsymbol x+\boldsymbol f(t),
\end{gather}
with $n$ unknown functions~$x^1$, \dots, $x^n$, $\boldsymbol x(t)=\big(x^1(t),\dots,x^n(t)\big){}^{\mathsf T}$,
where $n\geqslant2$.
The basic field~$\mathbb F$ is complex or real, $\mathbb F=\mathbb C$ or $\mathbb F=\mathbb R$.
The tuple $\theta=(A,B,\boldsymbol f)$ of arbitrary elements of the class~$\bar{\mathcal L}$
consists of arbitrary (sufficiently smooth) $n\times n$ matrix-valued functions~$A$ and~$B$ of~$t$
and an arbitrary (sufficiently smooth) vector-valued function~$\boldsymbol f$ of~$t$.
We consider various objects related to the class~$\bar{\mathcal L}$ and describe their properties.
This includes convenient gauges of the arbitrary-element tuple~$\theta$,
a hierarchy of subclasses of~$\bar{\mathcal L}$,
equivalence groupoids, equivalence groups and equivalence algebras of these subclasses
and their normalization properties,
maximal and essential Lie invariance algebras of systems from these subclasses,
the structure of such algebras and estimates of their dimensions
as well as the application of equivalence transformations and Lie symmetries
to reducing the order of systems of the form~\eqref{SLODE} and to their integration.

Systems from the class~$\bar{\mathcal L}$ with an arbitrary value $n\geqslant2$ have also not been well studied in the literature,
except certain very particular cases.
Lie symmetries of systems from the class~$\bar{\mathcal L}$ with $A=0$ and constant~$B$
were considered in the series of papers~\cite{camp2011a,camp2012a,gorr1988a,mele2011a,wafo2010a}
for lower values of~$n$ (at most six).
The group classification problem for such systems in the case of arbitrary~$n\geqslant2$ was thoroughly studied in~\cite{boyk2013a},
and the systems from the class~$\bar{\mathcal L}$ with constant commuting matrices~$A$ and~$B$ are reduced to such systems.
The consideration of the entire class~$\bar{\mathcal L}$ for the particular value $n=2$ within the framework of Lie symmetries
was initiated in~\cite{wafo2000b} and later continued in~\cite{mele2014a,mkhi2015a,moyo2013a};
see Section~\ref{sec:SLODEn2} for the exhaustive solution of the group classification problem for this class.
The case $n=3$ was in part considered in~\cite{suks2015a}.

The structure of the paper is as follows.

Section~\ref{sec:SLODEEquivGroupoidsAndEquivGroups} is devoted to
the construction of the equivalence groupoids and the equivalence groups of the class~$\bar{\mathcal L}$
and its nested subclasses obtained by gauging of the arbitrary-element tuple $\theta=(A,B,\boldsymbol f)$
by a known wide family of admissible transformations from the action groupoid of the equivalence group of~$\bar{\mathcal L}$.%
\footnote{%
See~\cite{opan2022a} for the notions of gauging arbitrary elements of a class of differential equations
and of mappings between such classes via families of point transformations.
}
We successively impose the gauges $\boldsymbol f=\boldsymbol0$, $A=0$ and $\mathop{\rm tr}B=0$
and reparameterize the singled out subclasses after each of the first two gauges,
excluding~$\boldsymbol f$ and then~$A$ from the tuple of arbitrary elements for the corresponding classes.
This leads to the class chain $\bar{\mathcal L}\hookleftarrow\mathcal L\hookleftarrow\mathcal L'\supset\mathcal L''$.
In particular, the class~$\mathcal L$ is the ``homogeneous'' counterpart of~$\bar{\mathcal L}$, i.e.,
is constituted by the homogeneous systems of the form~\eqref{SLODE}, where $\boldsymbol f=\boldsymbol 0$.
We prove that the above classes are semi-normalized in the usual sense,
the orbit of the elementary system $\boldsymbol x_{tt}=\boldsymbol0$ in each of them
under action of the corresponding equivalence group is the singular part of this class,
and the complement to the orbit, which is assumed as the regular class part,
has better normalization properties than this entire class.

The equivalence algebras of all the classes and their singular and regular parts
are computed in Section~\ref{sec:SLODEEquivAlgebras} as the infinitesimal counterparts
of the respective equivalence groups.

In Section~\ref{sec:SLODEPreliminaryAnalysisOfLieSyms},
we show that the group classifications of the class~$\bar{\mathcal L}$ and the gauged subclasses
split into the group classifications of their singular and regular parts.
The group classification problems for the singular parts are trivial
since their solutions are given by the elementary system $\boldsymbol x_{tt}=\boldsymbol0$
and its maximal Lie invariance algebra.
We derive the systems of determining equations for Lie symmetries of the systems from the regular class parts
and show that the group classification problems for these parts reduce to
the classifications of essential Lie-symmetry extensions within them.
We also compute the kernel point-symmetry groups for all the classes arising in this consideration.

Properties of possible essential Lie-symmetry extensions within the regular class parts
are studied in Section~\ref{sec:DescriptionOfEssLieSymExtensions}.
Using the obtained properties, we suggest two ways for classifying such extensions
within the framework of the algebraic approach
depending on their structure.

When studying Lie-symmetry extensions with nonzero projections on the $t$-line,
we need to recognize similar systems among those
possessing Lie-symmetry vector fields with constant $t$-components.
Necessary and sufficient conditions for such similarity are presented in Section~\ref{sec:SLODESummaryOfGroupClassification}.

The acquired knowledge on properties of essential Lie-symmetry extensions within the regular class parts
allows us to find the maximum dimension of the maximal Lie invariance algebras of systems from these parts,
which is submaximum within the entire class~$\bar{\mathcal L}$ and its gauged subclasses.
This is done in Section~\ref{sec:PropertiesOfIAs}.
Therein, we more thoroughly describe the structure of the essential Lie invariance algebras
of systems from the regular class parts
and characterize the maximal Lie invariance algebras of systems from the singular class parts.

In Section~\ref{sec:SLODEOrderReductionAndIntegrationUsingLieSyms},
we prove several assertions on
order reduction for normal linear systems of second-order ordinary differential equations and their integration
using their known Lie symmetries or equivalence transformations between them.

The techniques developed in Section~\ref{sec:DescriptionOfEssLieSymExtensions}
are applied in Section~\ref{sec:SLODEn2} to exhaustively solving
the problem of group classification of normal linear systems of second-order ordinary differential equations
for the lowest particular value $n=2$.

Possible generalizations of papers' results to normal linear systems of ordinary differential equations of greater order $r\geqslant3$
are discussed in Section~\ref{sec:OnGeneralizationToArbitraryEqOrder}.

In Section~\ref{sec:Conclusion}, we briefly sum up the main results of the present paper
and discuss open problems and possible directions for future research.

Specific properties of equivalence transformations in a general class of differential equations
whose auxiliary system for the arbitrary class elements contains algebraic equations
are studied in the final Section~\ref{sec:AlgebraicAuxiliaryEqsAndClassReparameterization}.

\begin{notation}
Throughout the paper, by default
the indices $a$, $b$, $c$ and~$d$ run from $1$ to $n$, i.e., $a,b,c,d=1,\dots,n$,
and we use the summation convention for repeated indices.
Functions with subscripts denote derivatives with respect to the corresponding variables.
Whenever it is essential to indicate the (open) domain of~$\mathbb F$ that is run by~$t$ in a particular situation,
we denote this domain by~$\mathcal I$.
We treat the differential equations and point transformations between them within the local approach.
As is customary in group analysis of differential equations,
parameter or arbitrary functions are assumed to be either analytical or, if $\mathbb F=\mathbb R$, smooth.
Nevertheless, at many points of the consideration,
the required degree of smoothness can be lowered to finite continuous differentiability.

$E$ is the $n\times n$ identity matrix.
For a square matrix~$M$, the matrices~$M_{\rm s}$ and~$M_{\rm n}$
denote the semisimple and nilpotent parts of~$M$ in its Jordan--Chevalley decomposition, $M=M_{\rm s}+M_{\rm n}$.
Fixing a basis in~$\mathbb F^n$, we identify linear operators on~$\mathbb F^n$ with their matrices in this basis.
By $[M_1,M_2]$ we denote the commutator of square matrices~$M_1$ and~$M_2$ of the same size, $[M_1,M_2]:=M_1M_2-M_2M_1$.
In the context of equations with matrix-valued functions,
it is convenient to interpret (constant) matrices as constant matrix-valued functions.
For example, the equations like $[M,V]=0$ and $V=M$ for a matrix-valued function~$V$ of~$t$ and a constant matrix~$M$
means that $[M,V(t)]=0$ and $V(t)=M$ for any $t\in\mathcal I$, respectively.

Given a class~$\mathcal K$ of systems of differential equations with the arbitrary-element tuple~$\theta$,
we use the same notation~$\theta$ for an arbitrary fixed value of this arbitrary-element tuple as well.
By~$\mathcal G^\sim_{\mathcal K}$, $G^\sim_{\mathcal K}$, $\mathfrak g^\sim_{\mathcal K}$,
$G^{{\rm s\vphantom{g}}\sim}_{\mathcal K}$, $\mathfrak g^{{\rm s\vphantom{g}}\sim}_{\mathcal K}$
and $\mathcal G^{G^\sim_{\mathcal K}}$
we denote the equivalence (pseudo)groupoid, the usual equivalence (pseudo)group, the usual equivalence (pseudo)algebra,
the significant usual equivalence (pseudo)group and the significant usual equivalence (pseudo)algebra of this class
and the action groupoid of~$G^\sim_{\mathcal K}$, respectively.
We omit the prefix ``pseudo'' for the above objects.
Since we consider only usual equivalence groups and usual equivalence algebras,
we omit the attribute ``usual'' for these objects as well.
For a fixed value of~$\theta$,
by~$K_\theta$, $G_\theta$, $\mathfrak g_\theta$ and $\mathcal G_\theta$
we denote
the system in~$\mathcal K$ associated with this value of~$\theta$,
the complete point-symmetry group of this system,
its maximal Lie invariance algebra and
the corresponding vertex group in~$\mathcal G^\sim_{\mathcal K}$.
See~\cite{vane2020b} and references therein for definitions of structures
related to point or contact transformations within classes of differential equations.

The notation $\mathcal K'\hookleftarrow\mathcal K$,
where $\mathcal K'$ is also a class of systems of differential equations
whose arbitrary-element tuple~$\theta'$ is a subtuple of~$\theta$,
means that the class~$\mathcal K'$ can be embedded in the class~$\mathcal K$ as a subclass.
In other words, the class~$\mathcal K'$ can be obtained from the class~$\mathcal K$
by imposing additional algebraic equations on~$\theta$,
substituting the expressions for the leading components of~$\theta$ in view of these equations into the systems from~$\mathcal K$
and reparameterizing the class~$\mathcal K$ to the class~$\mathcal K'$ via excluding the leading components from~$\theta$,
which results in the tuple~$\theta'$;
see Section~\ref{sec:AlgebraicAuxiliaryEqsAndClassReparameterization} for terminology and more details.
We use the term ``algebraic equation'' in the sense opposite to ``being truly differential''.

Irrespectively of the class under consideration,
by~$\varpi$ we denote the natural projection from the space
run by the joint tuple of the independent and dependent variables and the arbitrary elements of the class
onto the space run by the tuple of the independent and dependent variables.
Analogously, by $\pi$ we denote the natural projection from the space with the coordinates $(t,\boldsymbol x)$
onto the space with the single coordinate~$t$, $\pi\colon\mathbb F_t\times\mathbb F^n_{\boldsymbol x}\to\mathbb F_t$.
We use the subscript (resp.\ the superscript) ``$*$'' for denoting the pushforward (resp.\ the pullback)
by a~transformation.

Given a Lie algebra~$\mathfrak g$ and its subset~$\mathsf S$,
${\rm C}_{\mathfrak g}(\mathsf S)$ and ${\rm N}_{\mathfrak g}(\mathsf S)$ denote
the centralizer and the normalizer of~$\mathsf S$ in~$\mathfrak g$,
${\rm C}_{\mathfrak g}(\mathsf S):=\{u\in\mathfrak g\mid[u,v]=0\ \forall v\in \mathsf S\}$ and
${\rm N}_{\mathfrak g}(\mathsf S):=\{u\in\mathfrak g\mid[u,v]\in \mathsf S\ \forall v\in \mathsf S\}$.

For readers' convenience, we also collect here the notation of specific classes considered in the paper:
\begin{itemize}\itemsep=0ex
\item[$\bar{\mathcal L}$:]
the entire class of normal linear systems of $n$ second-order ordinary differential equations,
which are of the general form~\eqref{SLODE},
\item[$\mathcal L$:]
the ``homogeneous'' counterpart of $\bar{\mathcal L}$
consisting of the systems of the form~\eqref{SLODE} with $\boldsymbol f=\boldsymbol 0$,
\item[$\mathcal L'$:]
the class of the systems of the form~\eqref{basis_SLODE},
which is an analogue of the rational form for single linear ordinary differential equations,
\item[$\mathcal L''$:]
the subclass of~$\mathcal L'$ constituted by the systems of the form~\eqref{basis_SLODE}
with traceless values of the arbitrary-element matrix~$V$,
which is an analogue of the Laguerre--Forsyth form for single linear ordinary differential equations,
cf.\ \cite[p.~22]{wilc1901a} and \cite[p.~266]{seas1993a}.
\end{itemize}
The singular subclasses of the above classes,
which consist of the class elements similar to the elementary system $\boldsymbol x_{tt}=\boldsymbol0$,
are denoted by the notation of the corresponding class with subscript~0,
$\bar{\mathcal L}_0$, $\mathcal L_0$, $\mathcal L'_0$ and $\mathcal L''_0$.
The regular subclasses are the complements of singular ones and are denoted by
$\bar{\mathcal L}_1$, $\mathcal L_1$, $\mathcal L'_1$ and $\mathcal L''_1$, respectively.
\end{notation}

\section{Equivalence groupoids and equivalence groups}\label{sec:SLODEEquivGroupoidsAndEquivGroups}

It is obvious that the equivalence group~$G^\sim_{\bar{\mathcal L}}$ of the class~$\bar{\mathcal L}$
contains the subgroup constituted by the transformations
\begin{subequations}\label{eq:EquivTransbarLtoL'}
\begin{gather}\label{eq:EquivTransbarLtoL'A}
\tilde t=t, \quad
\tilde{\boldsymbol x}=H(t)\boldsymbol x+\boldsymbol h(t),
\\\label{eq:EquivTransbarLtoL'B}
\tilde A=(HA+2H_t)H^{-1},\ \
\tilde B=(HB-\tilde AH_t+H_{tt})H^{-1},\ \
\,\,\tilde{\!\!\boldsymbol f}=H\!\boldsymbol f+\boldsymbol h_{tt}-\tilde A\boldsymbol h_t-\tilde B\boldsymbol h,
\end{gather}
\end{subequations}
where $H$ is an arbitrary invertible $n\times n$ matrix-valued function of $t$
and $\boldsymbol h$ is an arbitrary vector-valued function of $t$.
Each fixed system~$\bar L_\theta$ from the class~$\bar{\mathcal L}$
is mapped by a transformation of the form~\eqref{eq:EquivTransbarLtoL'},
where $H$ satisfies the matrix equation $H_t+\frac12HA=0$
and $\boldsymbol h$ is a particular solution of~$\bar L_\theta$,
to the system~$\bar L_{\tilde\theta}$ from the same class,
where
\[\tilde A=0,\quad\tilde B=H\left(B-\frac12A_t+\frac14A^2\right)H^{-1},\quad\boldsymbol f=\boldsymbol 0.\]
Thus, the study of the class~$\bar{\mathcal L}$ reduces to the study of its subclass~$\mathcal L'$
singled out by the constraints $A=0$ and $\boldsymbol f=\boldsymbol0$.
We trivially reparameterize the class~$\mathcal L'$,
excluding $A$ and~$\boldsymbol f$ from the tuple of arbitrary elements for this class and re-denoting~$B$ by~$V$,
i.e., the systems from~$\mathcal L'$ take the form
\begin{gather}\label{basis_SLODE}
\boldsymbol x_{tt}=V(t)\boldsymbol x,
\end{gather}
where $V$ is an arbitrary $n\times n$ matrix-valued function of $t$, $V=V(t)=\big(V^{ab}(t)\big)$.
Note that according to the general notation introduced in Section~\ref{sec:Introduction},
$L'_V$ denotes the system from the class~$\mathcal L'$
corresponding to a fixed value of the matrix-valued parameter function~$V$.
The family of the equivalence transformations of the form~\eqref{eq:EquivTransbarLtoL'},
where the matrix- and vector-valued parameter functions~$H$ and~$\boldsymbol h$ are the solutions of
the matrix equation $H_t+\frac12HA=0$ and the system~$\bar L_\theta$
with the initial conditions $H(t_0)=E$ and $\boldsymbol h(t_0)=\boldsymbol h_t(t_0)=\boldsymbol0$
for a fixed point $t_0$ from the domain~$\mathcal I$ run by~$t$, respectively,
induces a mapping of the class~$\bar{\mathcal L}$ onto its subclass~$\mathcal L'$.

Consider the subclass~$\mathcal L'_0$ of~$\mathcal L'$
consisting of the systems~$L'_V$ in~$\mathcal L'$
with $V(t)=v(t)E$ for any $t\in\mathcal I$, where $v$ runs through the set of functions of~$t$,
$\{V(t)\mid t\in\mathcal I\}\subseteq\langle E\rangle$.
For such~$V$, we will say that $V$ is proportional to~$E$ with time-dependent proportionality factor.
In other words, the subclass~$\mathcal L'_0$ is singled out from the class~$\mathcal L'$
by the algebraic equations
\begin{gather}\label{eq:ClassL'AdditionalAuxiliarySystemForL'0}
V^{ab}=0,\ \ a\ne b,\quad V^{11}=\dots=V^{nn}.
\end{gather}
We can reparameterize it, assuming the common value~$v$ of the diagonal entries of~$V$
as the only arbitrary element of this subclass.
Let $\mathcal L'_1:=\mathcal L'\setminus\mathcal L'_0$,
and thus $\mathcal L'=\mathcal L'_0\sqcup\mathcal L'_1$.

The structure of the groupoid~$\mathcal G^\sim_{\mathcal L'}$  is described by the following theorem.

\begin{theorem}\label{thm:EquivGroupoidL'}
(i) The equivalence group~$G^\sim_{\mathcal L'}$ of the class~$\mathcal L'$ consists of the transformations of the form%
\footnote{%
For the square root of~$T_t$ in the expression for~$\tilde x$,
the absolute value of~$T_t$ should be substituted instead of~$T_t$ in the real case
or a branch of square root should be fixed in the complex case.\label{fnt:OnSquareRootOfTt}
}%
\begin{subequations}\label{eq:EquivGroupL'}
\begin{gather}\label{eq:EquivGroupL'A}
\tilde t=T(t), \quad \tilde{\boldsymbol x}=T_t^{\,1/2}C\boldsymbol x,
\\\label{eq:EquivGroupL'B}
\tilde V=\dfrac1{T_t^{\,2}}CVC^{-1}+\dfrac{2T_tT_{ttt}-3T_{tt}^{\,\,2}}{4T_t^{\,4}}E,
\end{gather}
\end{subequations}
where $T=T(t)$ is an arbitrary function of~$t$ with $T_t\ne0$
and $C$ is an arbitrary constant invertible $n\times n$ matrix.
The equivalence group of the subclass~$\mathcal L'_1$ and
the canonical significant equivalence group of the subclass~$\mathcal L'_0$
coincide with the group~$G^\sim_{\mathcal L'}$.

(ii) The partition of the class~$\mathcal L'$ into its subclasses~$\mathcal L'_0$ and~$\mathcal L'_1$
induces the partition of the groupoid~$\mathcal G^\sim_{\mathcal L'}$
into its subgroupoids~$\mathcal G^\sim_{\mathcal L'_0}$ and $\mathcal G^\sim_{\mathcal L'_1}$,
$\mathcal G^\sim_{\mathcal L'}=\mathcal G^\sim_{\mathcal L'_0}\sqcup\mathcal G^\sim_{\mathcal L'_1}$.

(iii) The subclass~$\mathcal L'_1$ is uniformly semi-normalized with respect to linear superposition of solutions.

(iv) The subclass~$\mathcal L'_0$ is the $G^\sim_{\mathcal L'}$-orbit of the elementary system~$L'_0$,
and this subclass and the entire class~$\mathcal L'$ are semi-normalized in the usual sense.

\end{theorem}

\begin{proof}
We first compute the equivalence group~$G^\sim_{\mathcal L'}$ of the class~$\mathcal L'$ using the direct method.
This class is formally defined as the set of systems of differential equations
with the single independent variable~$t$ and the $n$ dependent variables~$x^a$
that are of general form~\eqref{basis_SLODE}.
The arbitrary-element matrix~$V$ runs through the solution set of the auxiliary system of the matrix equations
\begin{gather}\label{eq:ClassL'AuxiliarySystem}
V_{x^a}=V_{x^a_t}=V_{x^a_{tt}}=0.
\end{gather}
The auxiliary system is equivalent to the condition that the entries of~$V$ depend at most on~$t$.
In view of this fact, the space with coordinates $(t,\boldsymbol x,V)$
can be taken as the underlying space for the group~$G^\sim_{\mathcal L'}$,
cf.\ \cite[footnote~1]{kuru2020a} and \cite[footnote~1]{opan2017a}.
In other words, we can assume that the group~$G^\sim_{\mathcal L'}$
consists of the point transformations in the space with coordinates $(t,\boldsymbol x,V)$
that are projectable to the space with coordinates $(t,\boldsymbol x)$
and whose proper prolongations preserve the joint system~\eqref{basis_SLODE} and~\eqref{eq:ClassL'AuxiliarySystem}.
The general form of such transformations is
\[
\mathscr T\colon\quad
\tilde t=T(t,\boldsymbol x),\quad
\tilde{\boldsymbol x}=\boldsymbol X(t,\boldsymbol x),\quad
\tilde V=\mathscr V(t,\boldsymbol x,V),
\]
where
$\boldsymbol X=(X^1,\dots,X^n)^{\mathsf T}$ and $\mathscr V=(\mathscr V^{ab})$.
For each transformation~$\mathscr T$,
its components~$T$ and~$X^a$ are functions on a domain $\Omega\subseteq\mathbb F_t\times\mathbb F^n_{\boldsymbol x}$
with $J:=\big|\p(T,\boldsymbol X)/\p(t,\boldsymbol x)\big|\ne0$ in each point of~$\Omega$,
its components~$\mathscr V^{ab}$ are functions on a domain $\Omega\times\Theta$,
where $\Theta$ is an open subset of~$\mathbb F^{n^2}_V$,
and the Jacobian of the entries of~$\mathscr V$ with respect to the entries of~$V$ is nonzero in each point of~$\Omega\times\Theta$.
Via the pullback with respect to the natural projection~$\varpi$
from $\mathbb F_t\times\mathbb F^n_{\boldsymbol x}\times\mathbb F^{n^2}_V$ onto $\mathbb F_t\times\mathbb F^n_{\boldsymbol x}$,
we can assume that the components~$T$ and~$X^a$ are defined on~$\Omega\times\Theta$ as well.
The prolongation~$\mathscr T_{(2,\boldsymbol x)}$ of~$\mathscr T$ to~${\boldsymbol x}_t$ and~${\boldsymbol x}_{tt}$
is computed in view of the chain rule,
\begin{gather}\label{eq:ProlongatioToXtAndXtt}
\tilde{\boldsymbol x}_{\tilde t}=\boldsymbol X^t:=\frac{\mathrm D_t\boldsymbol X}{\mathrm D_tT},\quad
\tilde{\boldsymbol x}_{\tilde t\tilde t}=\boldsymbol X^{tt}:=\frac1{\mathrm D_tT}\mathrm D_t\left(\frac{\mathrm D_t\boldsymbol X}{\mathrm D_tT}\right),
\end{gather}
where \smash{$\mathrm D_t=\p_t+x^a_t\p_{x^a}+x^a_{tt}\p_{x^a_t}+\cdots$} is the operator of total derivative with respect to~$t$
that is associated with the jet space ${\rm J}^\infty(\mathbb F_t,\mathbb F^n_{\boldsymbol x})$.
Thus, the transformation~\smash{$\mathscr T_{(2,\boldsymbol x)}$} acts on the space
${\rm J}^2(\mathbb F_t,\mathbb F^n_{\boldsymbol x})\times\mathbb F^{n^2}_V$,
where the second-order jet space ${\rm J}^2(\mathbb F_t,\mathbb F^n_{\boldsymbol x})$
can be identified with $\mathbb F_t\times\mathbb F^n_{\boldsymbol x}\times\mathbb F^n_{\boldsymbol x_t}\times\mathbb F^n_{\boldsymbol x_{tt}}$.
The further prolongation~$(\mathscr T_{(2,\boldsymbol x)})_{(1,V)}$ of~$\mathscr T_{(2,\boldsymbol x)}$
to the first derivatives of~$V$ with respect to $(t,\boldsymbol x,\boldsymbol x_t,\boldsymbol x_{tt})$ is defined by the matrix equations
\begin{subequations}\label{eq:L'PrologationOfEquivTransToDerivativesOfV}
\begin{gather}\label{eq:L'PrologationOfEquivTransToDerivativesOfVA}
T_t\tilde V_{\tilde t}+X^a_t\tilde V_{\tilde x^a}+X^{t,a}_t\tilde V_{\tilde x^a_{\tilde t}}+X^{tt,a}_t\tilde V_{\tilde x^a_{\tilde t\tilde t}}
=\mathscr V_t+V^{ab}_t\mathscr V_{V^{ab}},
\\ \label{eq:L'PrologationOfEquivTransToDerivativesOfVB}
T_{x^c}\tilde V_{\tilde t}+X^a_{x^c}\tilde V_{\tilde x^a}+X^{t,a}_{x^c}\tilde V_{\tilde x^a_{\tilde t}}+X^{tt,a}_{x^c}\tilde V_{\tilde x^a_{\tilde t\tilde t}}
=\mathscr V_{x^c}+V^{ab}_{x^c}\mathscr V_{V^{ab}},
\\ \label{eq:L'PrologationOfEquivTransToDerivativesOfVC}
T_{x^c_t}\tilde V_{\tilde t}+X^a_{x^c_t}\tilde V_{\tilde x^a}+X^{t,a}_{x^c_t}\tilde V_{\tilde x^a_{\tilde t}}+X^{tt,a}_{x^c_t}\tilde V_{\tilde x^a_{\tilde t\tilde t}}
=\mathscr V_{x^c_t}+V^{ab}_{x^c_t}\mathscr V_{V^{ab}},
\\ \label{eq:L'PrologationOfEquivTransToDerivativesOfVD}
T_{x^c_{tt}}\tilde V_{\tilde t}+X^a_{x^c_{tt}}\tilde V_{\tilde x^a}+X^{t,a}_{x^c_{tt}}\tilde V_{\tilde x^a_{\tilde t}}+X^{tt,a}_{x^c_{tt}}\tilde V_{\tilde x^a_{\tilde t\tilde t}}
=\mathscr V_{x^c_{tt}}+V^{ab}_{x^c_{tt}}\mathscr V_{V^{ab}}.
\end{gather}
\end{subequations}
In view of preserving the joint system~\eqref{basis_SLODE} and~\eqref{eq:ClassL'AuxiliarySystem} by~$\mathscr T$,
we can simultaneously set $V_{x^a}=V_{x^a_t}=V_{x^a_{tt}}=0$ and
$\tilde V_{\tilde x^a}=\tilde V_{\tilde x^a_{\tilde t}}=\tilde V_{\tilde x^a_{\tilde t\tilde t}}=0$ in the system~\eqref{eq:L'PrologationOfEquivTransToDerivativesOfV}.
Then the equations~\eqref{eq:L'PrologationOfEquivTransToDerivativesOfVC} and~\eqref{eq:L'PrologationOfEquivTransToDerivativesOfVD}
are satisfied identically since the components of~$\mathscr T$ do not depend on~${\boldsymbol x}_t$ and~${\boldsymbol x}_{tt}$.
The equations~\eqref{eq:L'PrologationOfEquivTransToDerivativesOfVB} reduce to the equations $T_{x^c}\tilde V_{\tilde t}=\mathscr V_{x^c}$,
which can subsequently be split with respect to the unconstrained derivatives~$\tilde V^{ab}_{\tilde t}$ to $T_{x^c}=0$ and $\mathscr V_{x^c}=0$
since there are not the arbitrary-element matrix~$V$ and its derivatives in these equations.
Therefore, $T=T(t)$ with~$T_t\ne0$ in view of~$J\ne0$.

To complete the system of determining equations for the components of~$\mathscr T$,
we successively substitute the expressions~$\boldsymbol X^{tt}$, $T$, $\boldsymbol X$ and $V\boldsymbol x$
for~$\tilde{\boldsymbol x}_{\tilde t\tilde t}$, $\tilde t$, $\tilde{\boldsymbol x}$ and $\boldsymbol x_{tt}$,
respectively, into the system $\tilde{\boldsymbol x}_{\tilde t\tilde t}=\tilde V(\tilde t)\tilde{\boldsymbol x}$
and take into account that $T=T(t)$.
This leads to the system
\begin{gather}\label{eq:EquivGroupL'DetEqs0}
T_t(X^a_{tt}+2X^a_{tb}x^b_t+X^a_{bc}x^b_tx^c_t+X^a_bV^{bc}x^c)-T_{tt}(X^a_t+X^a_dx^d_t)-T_t^{\,3}\tilde V^{ab}X^b=0,
\end{gather}
where $X^a_b:=\p X^a/\p x^b$, etc.
Since the first derivative~$\boldsymbol x_t$ is not constrained,
we can split the system~\eqref{eq:EquivGroupL'DetEqs0} with respect to its components.
Collecting the coefficients of the terms that are quadratic in these components
leads to the equations~$X^a_{bc}T_t=0$, i.e., $X^a_{bc}=0$.
This means that $\boldsymbol X=H(t)\boldsymbol x+\boldsymbol h(t)$ for
an invertible $n\times n$ matrix-valued function $H=(H^{ab})$ of~$t$
and a vector-valued function~$\boldsymbol h=(h^1,\dots,h^n)^{\mathsf T}$ of~$t$.
For each~$b$, we also collect coefficients of~$x^b_t$ in~\eqref{eq:EquivGroupL'DetEqs0}
under the obtained constraints and derive the system $2X^a_{tb}T_t-X^a_bT_{tt}=0$
or, equivalently, \smash{$\big(H^{ab}T_t^{-1/2}\big)_t=0$},
which integrates to \smash{$H=T_t^{1/2}C$}, where $C$ is a constant invertible $n\times n$ matrix;
see also footnote~\ref{fnt:OnSquareRootOfTt} for the proper interpretation of~\smash{$T_t^{1/2}$}.
The substitution of the obtained expressions for~$T$ and $\boldsymbol X$ into the rest of~\eqref{eq:EquivGroupL'DetEqs0}
consisting of the terms without the first derivative of~$\boldsymbol x$ leads to the equation
\begin{gather*}
T_t(H_{tt}\boldsymbol x+\boldsymbol h_{tt}+HV\boldsymbol x)-T_{tt}(H_t\boldsymbol x+\boldsymbol h)-T_t^{\,3}\tilde{V}(H\boldsymbol x+\boldsymbol h)=\boldsymbol0.
\end{gather*}
The subsequent splitting with respect to~$\boldsymbol x$ gives
the $V$-component of equivalence transformations,
\begin{gather}\label{eq:L'RelationBetweenVAndTildeV}
\tilde V=\frac{1}{T_t\,^2}\left(H_{tt}+HV-\frac{T_{tt}}{T_t}H_t\right)H^{-1}
=\frac1{T_t^{\,2}}CVC^{-1}+\frac{2T_tT_{ttt}-3T_{tt}^{\,\,2}}{4T_t^{\,4}}E,
\end{gather}
and the system
\begin{gather}\label{eq:L'RelationOfLinSuperposition}
T_t\boldsymbol h_{tt}-T_{tt}\boldsymbol h-T_t^{\,3}\tilde V\boldsymbol h=\boldsymbol0.
\end{gather}
Splitting the system~\eqref{eq:L'RelationOfLinSuperposition}
with respect to the transformed arbitrary elements~$\tilde V^{ab}$, we obtain $\boldsymbol h=\boldsymbol0$.
Therefore, any element of the group~$G^\sim_{\mathcal L'}$ is of the form~\eqref{eq:EquivGroupL'},
and any transformation of this form belongs to~$G^\sim_{\mathcal L'}$.

It is obvious that the action of~$G^\sim_{\mathcal L'}$ on the class~$\mathcal L'$
preserves its subclasses~$\mathcal L'_0$ and~$\mathcal L'_1$.
Moreover, in course of computing the equivalence group of the subclass~$\mathcal L'_1$,
the system~\eqref{eq:EquivGroupL'DetEqs0} is split in the same way as in the course of computing~$G^\sim_{\mathcal L'}$,
and we obtain the same form~\eqref{eq:EquivGroupL'} for elements of this group as for those of~$G^\sim_{\mathcal L'}$.
Hence the equivalence group of the subclass~$\mathcal L'_1$ coincides with~$G^\sim_{\mathcal L'}$.

For the subclass~$\mathcal L'_0$, the auxiliary system~\eqref{eq:ClassL'AuxiliarySystem}
is extended by the algebraic equations~\eqref{eq:ClassL'AdditionalAuxiliarySystemForL'0}.
We can repeat the above procedure for this subclass,
splitting with respect to $\tilde v_{\tilde t}$ and~$\tilde v$ instead of~\smash{$\tilde V^{ab}_{\tilde t}$} and~$\tilde V^{ab}$,
where~$\tilde v$ denotes the common value of $\tilde V^{11}$, \dots, $\tilde V^{nn}$.
The only difference in results is that now the matrix relation~\eqref{eq:L'RelationBetweenVAndTildeV}
holds for~$V$ satisfying~\eqref{eq:ClassL'AdditionalAuxiliarySystemForL'0}.
The condition of preserving the additional auxiliary system~\eqref{eq:ClassL'AdditionalAuxiliarySystemForL'0}
by~$\mathscr T$ is independent on that for the joint system~\eqref{basis_SLODE} and~\eqref{eq:ClassL'AuxiliarySystem}.
It implies that
$\mathscr V^{ab}=0$, $a\ne b$, $\mathscr V^{11}=\dots=\mathscr V^{nn}$ whenever
$V^{ab}=0$, $a\ne b$, $V^{11}=\dots=V^{nn}$.
The last condition is consistent with~\eqref{eq:L'RelationBetweenVAndTildeV}.
The above consideration means that the canonical significant equivalence group of the subclass~$\mathcal L'_0$
coincides with the group~$G^\sim_{\mathcal L'}$ as well, which completes proving item~(i) of the theorem.

To compute the equivalence groupoids of the class~$\mathcal L'$, we apply the direct method as well.
Suppose that the triple $(V,\Phi,\tilde V)$ is an admissible transformation of the class~$\mathcal L'$.
Here $\Phi$ is a~point transformation in the space with coordinates~$(t,\boldsymbol x)$,
\[
\Phi\colon\ \tilde t=T(t,\boldsymbol x),\ \tilde{\boldsymbol x}=\boldsymbol X(t,\boldsymbol x),
\quad\mbox{where}\quad
\boldsymbol X=(X^1,\dots,X^n)^{\mathsf T},\quad
J:=\left|\frac{\p(T,\boldsymbol X)}{\p(t,\boldsymbol x)}\right|\ne0
\]
on the domain~$\Omega\subseteq\mathbb F^{n+1}$ of~$\Phi$,
that maps the system~$L'_V$ to the system~\smash{$L'_{\tilde V}$}.
The second prolongation of~$\Phi$ is computed according to~\eqref{eq:ProlongatioToXtAndXtt}.
To derive the system of determining equations for the components of~$\Phi$ and
the expression of~$\tilde V$ in terms of $V$, $T$ and~$\boldsymbol X$,
we substitute the above expression for~$\tilde{\boldsymbol x}_{\tilde t\tilde t}$
and then $T$, $\boldsymbol X$ and $V\boldsymbol x$ for $\tilde t$, $\tilde{\boldsymbol x}$ and $\boldsymbol x_{tt}$,
respectively, into the system~\smash{$L'_{\tilde V}$}.
This leads to the system
\begin{gather}\label{eq:EquivGroupoidL'DetEqs0}
\begin{split}
&(X^a_{tt}+2X^a_{tb}x^b_t+X^a_{bc}x^b_tx^c_t+X^a_bV^{bc}x^c)(T_t+T_dx^d_t)\\
&\qquad{} -(T_{tt}+2T_{tb}x^b_t+T_{bc}x^b_tx^c_t+T_bV^{bc}x^c)(X^a_t+X^a_dx^d_t)-\tilde V^{ab}X^b(T_t+T_cx^c_t)^3=0,
\end{split}
\end{gather}
where $T_a:=\p T/\p x^a$, etc., which is the counterpart of the system~\eqref{eq:EquivGroupL'DetEqs0}.
It is convenient to handle the system~\eqref{eq:EquivGroupoidL'DetEqs0} in a different way.
For each value of~$a$, the left-hand side of~\eqref{eq:EquivGroupoidL'DetEqs0}
and all of its parts are polynomials in the components of~$\boldsymbol x_t$
with coefficients depending on $(t,\boldsymbol x)$.
The polynomial $P_T:=T_t+T_dx^d_t$ divides the first and the third summands in~\eqref{eq:EquivGroupoidL'DetEqs0}.
Hence it divides the second summand as well.
If the degree of the polynomial~$P_T$ (with respect to the components of~$\boldsymbol x_t$) equals one,
then it divides a factor in the second summand,
and this should be the first factor since otherwise $J=0$.
The divisibility of the first factor by~$P_T$ is trivial if $P_T$ does not depend on~$\boldsymbol x_t$.
As a result, there exist functions~$\Lambda^0$ and~$\Lambda^a$ of $(t,\boldsymbol x)$ such that
\begin{subequations}\label{eq:EquivGroupoidL'DetEqs1}
\begin{gather}\label{eq:EquivGroupoidL'DetEqs1a}
T_{tt}+2T_{tb}x^b_t+T_{bc}x^b_tx^c_t+T_bV^{bc}x^c=(\Lambda^0\!+\!\Lambda^cx^c_t)(T_t\!+\!T_dx^d_t),
\\\label{eq:EquivGroupoidL'DetEqs1b}
X^a_{tt}+2X^a_{tb}x^b_t+X^a_{bc}x^b_tx^c_t+X^a_bV^{bc}x^c=(\Lambda^0\!+\!\Lambda^cx^c_t)(X^a_t\!+\!X^a_dx^d_t)+\tilde V^{ab}X^b(T_t\!+\!T_cx^c_t)^2.
\end{gather}
\end{subequations}
Introducing the notation $x^0:=t$, $X^0:=T$ and $\tilde V^{0b}:=0$,
we represent the equations~\eqref{eq:EquivGroupoidL'DetEqs1a} and~\eqref{eq:EquivGroupoidL'DetEqs1b}
in a unified form,
\begin{gather}\label{eq:EquivGroupoidL'DetEqs1'}
X^\kappa_{\lambda\mu}x^\lambda_tx^\mu_t+X^\kappa_bV^{bc}x^c(x^0_t)^2=\Lambda^\mu x^\mu_tX^\kappa_\nu x^\nu_t+\tilde V^{\kappa b}X^b(X^0_\mu x^\mu_t)^2.
\end{gather}
Here and up to the end of the proof, the indices~$\kappa$, $\lambda$, $\mu$ and~$\nu$ run from~0 to~$n$.
The splitting of the system~\eqref{eq:EquivGroupoidL'DetEqs1} with respect to $(x^b_t)_{b=1,\dots,n}$
is equivalent to
the splitting of the system~\eqref{eq:EquivGroupoidL'DetEqs1'} with respect to $(x^\lambda_t)_{\lambda=0,\dots,n}$,
which leads to the complete system of determining equations for admissible transformations in the class~$\mathcal L'$,
\begin{gather}\label{eq:EquivGroupoidL'DetEqs1''}
X^\kappa_{\lambda\mu}
=\frac12(\Lambda^\mu X^\kappa_\lambda+\Lambda^\lambda X^\kappa_\mu)
+\tilde V^{\kappa b}X^bX^0_\lambda X^0_\mu
-X^\kappa_bV^{bc}x^c\delta_{0\lambda}\delta_{0\mu},
\end{gather}
where $\delta_{0\lambda}$ is the Kronecker delta.

In view of the identity $\p_\nu X^\kappa_{\lambda\mu}=\p_\mu X^\kappa_{\lambda\nu}$, 
for each fixed value of $(\kappa,\lambda,\mu)$,
we differentiate the corresponding equation in~\eqref{eq:EquivGroupoidL'DetEqs1''} with respect to~$x^\nu$ with $\nu\ne\mu$
and subtract the obtained differential consequence of~\eqref{eq:EquivGroupoidL'DetEqs1''}
from the similar consequence with permuted~$\mu$ and~$\nu$.
This gives the equations
\begin{gather}\label{eq:EquivGroupoidL'DetEqs1'''}
\begin{split}
&\frac12X^\kappa_\lambda K^{\mu\nu}+\frac12X^\kappa_\mu M^{\lambda\nu}-\frac12X^\kappa_\nu M^{\lambda\mu}
 +\tilde V^{\kappa b}(X^b_\nu X^0_\mu-X^b_\mu X^0_\nu)X^0_\lambda\\
&\qquad{}-X^\kappa_bV^{bc}(\delta_{c\nu}\delta_{0\mu}-\delta_{c\mu}\delta_{0\nu})\delta_{0\lambda}=0,\quad \mu\ne\nu,
\end{split}
\end{gather}
where
$K^{\mu\nu}:=\Lambda^\mu_\nu-\Lambda^\nu_\mu$ (hence $K^{\mu\nu}=-K^{\nu\mu}$) and
$M^{\lambda\nu}:=\Lambda^\lambda_\nu-\frac12\Lambda^\nu\Lambda^\lambda+\Lambda^bV^{bc}x^c\delta_{0\nu}\delta_{0\lambda}$.
We take each pair of equations of the form~\eqref{eq:EquivGroupoidL'DetEqs1'''}
with the same values of~$\kappa$, $\mu$ and~$\nu$, different values of $\lambda$ and~$\lambda'$,
multiply these equations by $2X^0_{\lambda'}$ and $2X^0_\lambda$, respectively,
and subtract the second obtained equations from the first one.
As a result, we derive
\begin{gather}\label{eq:EquivGroupoidL'DetEqs1iv}
\begin{split}
&K^{\mu\nu}X^0_{\lambda'}\bar{\boldsymbol X}_\lambda
-K^{\mu\nu}X^0_\lambda\bar{\boldsymbol X}_{\lambda'}
+(M^{\lambda\nu}X^0_{\lambda'}-M^{\lambda'\nu}X^0_{\lambda})\bar{\boldsymbol X}_\mu
-(M^{\lambda\mu}X^0_{\lambda'}-M^{\lambda'\mu}X^0_{\lambda})\bar{\boldsymbol X}_\nu
 \\
&\qquad{}-2V^{bc}(\delta_{c\nu}\delta_{0\mu}-\delta_{c\mu}\delta_{0\nu})
(\delta_{0\lambda}X^0_{\lambda'}-\delta_{0\lambda'}X^0_{\lambda})\bar{\boldsymbol X}_b=\boldsymbol0, \quad\mu\ne\nu, \quad \lambda\ne\lambda',
\noprint{
&X^\kappa_\lambda K^{\mu\nu}X^0_{\lambda'}-X^\kappa_{\lambda'} K^{\mu\nu}X^0_\lambda
+X^\kappa_\mu(M^{\lambda\nu}X^0_{\lambda'}-M^{\lambda'\nu}X^0_{\lambda})
-X^\kappa_\nu(M^{\lambda\mu}X^0_{\lambda'}-M^{\lambda'\mu}X^0_{\lambda})
 \\
&\qquad{}-2X^\kappa_bV^{bc}(\delta_{c\nu}\delta_{0\mu}-\delta_{c\mu}\delta_{0\nu})
(\delta_{0\lambda}X^0_{\lambda'}-\delta_{0\lambda'}X^0_{\lambda})=0, \quad\mu\ne\nu, \quad \lambda\ne\lambda'.
}
\end{split}
\end{gather}
where $\bar{\boldsymbol X}=(X^0,\dots,X^n)^{\mathsf T}$.
The right-hand side of the system~\eqref{eq:EquivGroupoidL'DetEqs1iv} is skew-symmetric in $(\mu,\nu)$ and in $(\lambda,\lambda')$,
and thus its independent equations are exhausted, e.g., by those with $\mu<\nu$ and $\lambda<\lambda'$.
The function tuples $\bar{\boldsymbol X}_\mu$, $\mu=0,\dots,n$, are linearly independent at each point~$(t,\boldsymbol x)\in\Omega$
since $J=\det(X^\kappa_\lambda)\ne0$ on~$\Omega$.
This means that we can split the system~\eqref{eq:EquivGroupoidL'DetEqs1iv},
collecting the coefficients of~$\bar{\boldsymbol X}_\lambda$ for each fixed~$\lambda$
and equating them to zero.

Suppose that there exists $b$ with $X^0_b\ne0$.

If $n\geqslant3$, then for each value of the tuple of $(\lambda,\lambda',\mu,\nu)$
with $0=\lambda\ne\lambda'$ and $0<\mu<\nu$ or
with $\lambda\ne\lambda'$, $\lambda\lambda'\ne0$, $0=\mu<\nu$ and $\nu\ne\lambda$,
after collecting the coefficients of~$\bar{\boldsymbol X}_\lambda$, we obtain $K^{\mu\nu}X^0_{\lambda'}=0$.
In the last system, the subscript~$\lambda'$ runs from~1 to~$n$,
and thus $K^{\mu\nu}=0$ for any $\mu,\nu=0,\dots,n$.

In the case $n=2$, we collect the coefficients of~$\bar{\boldsymbol X}_0$ in the equations from the system~\eqref{eq:EquivGroupoidL'DetEqs1iv}
with $(\lambda,\lambda',\mu,\nu)=(0,a,1,2)$, $a\in\{1,2\}$, and derive the equations $K^{12}X^0_a=0$, which jointly imply $K^{12}=0$.
Then collecting the coefficients of~$\bar{\boldsymbol X}_1$ and~$\bar{\boldsymbol X}_2$
in the same equations from the system~\eqref{eq:EquivGroupoidL'DetEqs1iv}
leads to the equations $M^{0a}X^0_b-M^{ba}X^0_0=0$.
Considering the coefficients of~$\bar{\boldsymbol X}_2$ and~$\bar{\boldsymbol X}_0$
in the equations~\eqref{eq:EquivGroupoidL'DetEqs1iv}
with $(\lambda,\lambda',\mu,\nu)=(1,2,0,1)$ and $(\lambda,\lambda',\mu,\nu)=(0,2,0,1)$, respectively,
we obtain $K^{01}X^0_1=0$ and $K^{01}X^0_2+M^{01}X^0_2-M^{21}X^0_0=0$.
Hence $K^{01}X^0_2=0$ as well, and together with $K^{01}X^0_1=0$ this gives $K^{01}=0$.
In a similar way, we derive $K^{02}=0$.
As~a result, in total we again have that $K^{\mu\nu}=0$ for any $\mu,\nu=0,\dots,n$.

Then, taking into account the last condition, we consider an equation from the system~\eqref{eq:EquivGroupoidL'DetEqs1iv}
with $\nu=\lambda'=0$ and, therefore, $\mu\lambda\ne0$.
Collecting the coefficients of~$\bar{\boldsymbol X}_b$ with $b\ne\mu$ and~$\bar{\boldsymbol X}_\mu$,
we respectively obtain the equations $V^{b\mu}X^0_\lambda=0$ with $b\ne\mu$ and
$M^{\lambda0}X^0_0-M^{00}X^0_\lambda-2V^{\mu\mu}X^0_{\lambda}=0$.
Subtracting the last equation from the similar equation, where $\mu$ is replaced by $\mu'\notin\{0,\mu\}$,
gives $(V^{\mu\mu}-V^{\mu'\mu'})X^0_{\lambda}=0$.
Since here the subscript~$\lambda$ runs from~1 to~$n$, the derived equations imply $V^{ab}=0$ and $V^{aa}=V^{bb}$ for $a\ne b$.
(Above there is no summation with respect to the repeated indices~$\mu$, $\mu'$, $a$ and~$b$.)

Summing up, we prove that the matrix-valued function~$V$
is proportional to the identity matrix~$E$ with time-dependent proportionality factor
if $X^0_b\ne0$ for some~$b$.

Therefore, for the other values of the matrix-valued parameter function~$V$,
we have \mbox{$X^0_b=0$} for any~$b$, i.e., the $t$-component~$T$ of~$\Phi$ depends only on~$t$,
$T=T(t)$, and $T_t\ne0$ since \mbox{$J\ne0$}.
Then the equations~\eqref{eq:EquivGroupoidL'DetEqs1a} imply $\Lambda^c=0$,
and thus collecting coefficients of quadratic in~$\boldsymbol x_t$ terms
in the equations~\eqref{eq:EquivGroupoidL'DetEqs1b} gives $X^a_{bc}=0$ for any~$a$, $b$ and~$c$.
The further computation can be based on the system~\eqref{eq:EquivGroupoidL'DetEqs0}.
It is analogous to the above computation of the group~$G^\sim_{\mathcal L'}$ and results in
the~relation~\eqref{eq:L'RelationBetweenVAndTildeV} between the source and target arbitrary-element tuples~$V$ and~$\tilde V$
as well as in the system~\eqref{eq:L'RelationOfLinSuperposition}
meaning that the composition $\boldsymbol h\circ\tilde T$ of $\boldsymbol h$ with the inverse~$\tilde T$ of~$T$
is a~solution of the system~\smash{$L'_{\tilde V}$}.
Hence the matrix-valued functions~$V$ and~$\tilde V$ are simultaneously not proportional to the identity matrix~$E$ with time-dependent proportionality factors.
In other words, the action of the groupoid~$\mathcal G^\sim_{\mathcal L'}$ preserves
the subclass~$\mathcal L'_1$ and hence the subclass~$\mathcal L'_0$ as well since it is the complement of~$\mathcal L'_1$ in~$\mathcal L'$,
which proves item~(ii).
Moreover, any admissible transformation in~$\mathcal L'_1$ is the composition
of two admissible transformations, one generated by an equivalence transformation from~$L'_V$ to~$L'_{\tilde V}$
and one related to the linear superposition of solutions of~$L'_{\tilde V}$, and thus item~(iii) holds true.
The fact that the subclass~$\mathcal L'_0$ is the $G^\sim_{\mathcal L'}$-orbit of~$L'_0$
is obvious since any system from~$\mathcal L'_0$ is $G^\sim_{\mathcal L'}$-equivalent to the elementary (\emph{free particle}) system~$L'_0$.
Consequently, the subclasses~$\mathcal L'_0$ and~$\mathcal L'_1$ are semi-normalized in the usual sense
and admit the same equivalence group, which implies the semi-normalization of the entire class~$\mathcal L'$
and completes the proof of item~(iv).
\end{proof}

\begin{remark}
The equivalence group of the reparameterized subclass~$\mathcal L'_0$,
where $v=v(t)$ with $V=vE$, is assumed as the single arbitrary element instead of the arbitrary-element matrix~$V$,
consists of the point transformations in the space with the coordinates $(t,\boldsymbol x,v)$,
whose $(t,\boldsymbol x)$- and $v$-components are respectively given by~\eqref{eq:EquivGroupL'A} and
\[
\tilde v=\dfrac1{T_t^{\,2}}v+\dfrac{2T_tT_{ttt}-3T_{tt}^{\,\,2}}{4T_t^{\,4}}.
\]
\end{remark}

\begin{remark}
Denote by
\smash{$\mathcal G^{G^\sim_{\mathcal L'}}\!|_{\mathcal L'_0}$},
\smash{$\mathcal G^{G^\sim_{\mathcal L'}}\!|_{\mathcal L'_1}$},
\smash{$\mathcal G^{\rm lin}_{\mathcal L'_1}$}
and $\mathcal G_0$
the restriction of the action groupoid~\smash{$\mathcal G^{G^\sim_{\mathcal L'}}$}
of the equivalence group~$G^\sim_{\mathcal L'}$ to the subclasses~$\mathcal L'_0$ and~$\mathcal L'_1$,
the groupoid related to the linear superposition of solutions within the subclass~$\mathcal L'_0$
and the vertex group of the system~$L'_0$, respectively,
\[
\smash{\mathcal G^{G^\sim_{\mathcal L'}}\!|_{\mathcal L'_0}=\mathcal G^{G^\sim_{\mathcal L'}}\cap\mathcal G^\sim_{\mathcal L'_0}},\quad
\smash{\mathcal G^{G^\sim_{\mathcal L'}}\!|_{\mathcal L'_1}=\mathcal G^{G^\sim_{\mathcal L'}}\cap\mathcal G^\sim_{\mathcal L'_1}},\quad
\mathcal G^{\rm lin}_{\mathcal L'_1}\subset\mathcal G^{\rm f}_{\mathcal L'_1},
\]
where $\mathcal G^{\rm f}_{\mathcal L'_1}$ is the fundamental groupoid of~$\mathcal L'_1$.
Item~(iv) of Theorem~\ref{thm:EquivGroupoidL'} is reformulated in terms of admissible transformations
in the following way:
for any \smash{$\mathcal T\in\mathcal G^\sim_{\mathcal L'_0}$}
there exist \smash{$\mathcal T_1,\mathcal T_3\in\mathcal G^{G^\sim_{\mathcal L'}}\!|_{\mathcal L'_0}$}
and $\mathcal T_2\in\mathcal G_0$ such that $\mathcal T=\mathcal T_1\star\mathcal T_2\star\mathcal T_3$.
Here ``$\star$'' denotes the operation of composing of admissible transformations~\cite{vane2020b}.
Since \smash{$\mathcal G^{\rm lin}_{\mathcal L'_1}$} is a normal subgroupoid of \smash{$\mathcal G^\sim_{\mathcal L'_1}$},
item~(iii) of Theorem~\ref{thm:EquivGroupoidL'} means
that \smash{$\mathcal G^\sim_{\mathcal L'_1}$} is a semidirect product of
\smash{$\mathcal G^{G^\sim_{\mathcal L'}}\!|_{\mathcal L'_1}$} and \smash{$\mathcal G^{\rm lin}_{\mathcal L'_1}$}.
\end{remark}

From Theorem~\ref{thm:EquivGroupoidL'}, we can derive
the descriptions of the equivalence groupoids of the superclasses~$\bar{\mathcal L}$ and~$\mathcal L$.
The auxiliary system for the arbitrary elements of the class~$\bar{\mathcal L}$ is
\[
A_{x^a}=A_{x^a_t}=A_{x^a_{tt}}=0,\quad
B_{x^a}=B_{x^a_t}=B_{x^a_{tt}}=0,\quad
\boldsymbol f_{x^a}=\boldsymbol f_{x^a_t}=\boldsymbol f_{x^a_{tt}}=0,
\]
and its counterpart for the class~$\mathcal L$ is its subsystem consisting of the equations to~$A$ and~$B$ but not to~$\boldsymbol f$.
Denote by $\bar{\mathcal L}_0$ and~$\mathcal L_0$
the $G^\sim_{\bar{\mathcal L}\,}$- and $G^\sim_{\mathcal L}$-orbits of
the elementary system~$\boldsymbol x_{tt}=\boldsymbol 0$, respectively,
and let
$\bar{\mathcal L}_1:=\bar{\mathcal L}\setminus\bar{\mathcal L}_0$ and
$\mathcal L_1:=\mathcal L\setminus\mathcal L_0$.
The consideration in the beginning of this section implies that
the subclasses~$\bar{\mathcal L}_0$ and~$\mathcal L_0$ are respectively singled out
from their superclasses~$\bar{\mathcal L}$ and~$\mathcal L$ by the constraint
that the matrix-valued function $B-\frac12A_t+\frac14A^2$ is proportional to the identity matrix~$E$ with time-dependent proportionality factor.
In other words, the extension of the auxiliary systems for~$\bar{\mathcal L}_0$ and~$\mathcal L_0$
is given by the equations
$\big(B-\frac12A_t+\frac14A^2\big)^{ab}=0$, $a\ne b$, $\big(B-\frac12A_t+\frac14A^2\big)^{11}=\dots=\big(B-\frac12A_t+\frac14A^2\big)^{nn}$.

\begin{theorem}\label{thm:EquivGroupoidbarL}
(i) The equivalence group~$G^\sim_{\bar{\mathcal L}}$ of the class~$\bar{\mathcal L}$ consists of the transformations of the form
\begin{subequations}\label{eq:EquivGroupbarL}
\begin{gather}\label{eq:EquivGroupbarLA}
\tilde t=T(t), \quad
\tilde{\boldsymbol x}=H(t)\boldsymbol x+\boldsymbol h(t),
\\ \label{eq:EquivGroupbarLB}
\tilde A=T_t^{-2}(T_tHA+2T_tH_t-T_{tt}H)H^{-1}, \\
\label{eq:EquivGroupbarLC}
\tilde B=T_t^{-3}(T_tHB-T_t^{\,2}\tilde AH_t+T_tH_{tt}-T_{tt}H_t)H^{-1},\\
\label{eq:EquivGroupbarLD}
\,\,\tilde{\!\!\boldsymbol f}=T_t^{-3}(T_tH\!\boldsymbol f+T_t\boldsymbol h_{tt}-T_{tt}\boldsymbol h_t-T_t^{\,2}\tilde A\boldsymbol h_t-T_t^{\,3}\tilde B\boldsymbol h),
\end{gather}
\end{subequations}
where $T=T(t)$ is an arbitrary function of~$t$ with $T_t\ne0$,
$H$ is an arbitrary invertible $n\times n$ matrix-valued function of~$t$
and $\boldsymbol h$ is an arbitrary vector-valued function of $t$.
The equivalence groups of the subclasses~$\bar{\mathcal L}_1$ and~$\bar{\mathcal L}_0$
coincide with~$G^\sim_{\bar{\mathcal L}}$.

(ii) The partition of the class~$\bar{\mathcal L}$ into its subclasses~$\bar{\mathcal L}_0$ and~$\bar{\mathcal L}_1$
induces the partition of the groupoid~$\mathcal G^\sim_{\bar{\mathcal L}}$
into its subgroupoids~$\mathcal G^\sim_{\bar{\mathcal L}_0}$ and $\mathcal G^\sim_{\bar{\mathcal L}_1}$,
$\mathcal G^\sim_{\bar{\mathcal L}}=\mathcal G^\sim_{\bar{\mathcal L}_0}\sqcup\mathcal G^\sim_{\bar{\mathcal L}_1}$.

(iii) The subclass~$\bar{\mathcal L}_1$ is normalized in the usual sense.

(iv) The subclass~$\bar{\mathcal L}_0$ and the entire class~$\bar{\mathcal L}$ are semi-normalized in the usual sense.
\end{theorem}

\begin{proof}
It can be checked by direct computation that any transformation of the form~\eqref{eq:EquivGroupbarL}
is an equivalence transformation of the class~$\bar{\mathcal L}$
and its subclasses~$\bar{\mathcal L}_0$ and~$\bar{\mathcal L}_1$.
Let us prove simultaneously with the other claims of the theorem
that this class has no other equivalence transformations.

Consider any two similar systems~$\bar L_\theta$ and~\smash{$\bar L_{\tilde\theta}$} from the class~$\bar{\mathcal L}$,
and let the systems $L'_V$ and~\smash{$L'_{\tilde V}$} from the class~$\mathcal L'$
be respectively the images of~$\bar L_\theta$ and~\smash{$\bar L_{\tilde\theta}$}
under the mapping of~$\bar{\mathcal L}$ to~$\mathcal L'$
that is introduced in the first paragraph of this section.
If $\bar L_\theta\in\bar{\mathcal L}_0$, then $L'_V\in\mathcal L'_0$, \smash{$L'_{\tilde V}\in\mathcal L'_0$},
and thus \smash{$\bar L_{\tilde\theta}\in\bar{\mathcal L}_0$}.
Therefore, the subclass~$\bar{\mathcal L}_0$ is preserved by the action of~$\mathcal G^\sim_{\bar{\mathcal L}}$ on~$\bar{\mathcal L}$.
Then the subclass~$\bar{\mathcal L}_1$ is also preserved as the complement of~$\bar{\mathcal L}_0$ in~$\bar{\mathcal L}$,
which gives item~(ii).

Let $\Phi$ and $\tilde\Phi$ be the point transformations of the form~\eqref{eq:EquivTransbarLtoL'A}
that are associated with these systems under the above mapping,
$\Phi_*\bar L_\theta=L'_V$ and \smash{$\tilde\Phi_*\bar L_{\tilde\theta}=L'_{\tilde V}$},
and let $\Psi$ be a point transformation mapping~$\bar L_\theta$ to~\smash{$\bar L_{\tilde\theta}$}.
Then the transformation $\hat\Psi:=\tilde\Phi\circ\Psi\circ\Phi^{-1}$ maps~$L'_V$ to~\smash{$L'_{\tilde V}$}.

Consider separately the cases $\bar L_\theta\in\bar{\mathcal L}_1$ and $\bar L_\theta\in\bar{\mathcal L}_0$.

For $\bar L_\theta\in\bar{\mathcal L}_1$, we have that $L'_V\in\mathcal L'_1$.
In view of item~(iii) of Theorem~\ref{thm:EquivGroupoidL'}, the transformation~$\hat\Psi$ is the composition
of a point transformation of the form~\eqref{eq:EquivGroupL'A}
and a point-symmetry transformation of linear superposition of solutions for the system~\smash{$L'_{\tilde V}$},
whereas the relation between $V$~and~$\tilde V$ is given by~\eqref{eq:EquivGroupL'B}.
Hence the transformation $\Psi=\tilde\Phi^{-1}\circ\hat\Psi\circ\Phi$ is of the form~\eqref{eq:EquivGroupbarLA},
and the arbitrary-element tuples~$\theta$ and~$\tilde\theta$ are related
according to~\eqref{eq:EquivGroupbarLB} and~\eqref{eq:EquivGroupbarLC}.
In other words, any admissible transformation in the subclass~$\bar{\mathcal L}_1$
is induced an equivalence transformation of the form~\eqref{eq:EquivGroupbarL}.
Therefore, the subclass~$\bar{\mathcal L}_1$ is normalized in the usual sense.
Since there is no non-identity insignificant or gauge equivalence transformations in this subclass,
its equivalence group is exhausted by the transformations
in the space with coordinates $(t,\boldsymbol x,A,B,\boldsymbol f)$
whose components are defined by~\eqref{eq:EquivGroupbarL}.
\looseness=-1

For $\bar L_\theta\in\bar{\mathcal L}_0$, we have that $L'_V\in\mathcal L'_0$,
and thus item~(iv) of Theorem~\ref{thm:EquivGroupoidL'} is relevant here.
The transformation~$\hat\Psi$ is the composition
of a transformation of the form~\eqref{eq:EquivGroupL'A}
and a point-symmetry transformation of the system~\smash{$L'_{\tilde V}$},
whereas the relation between~$V$ and~$\tilde V$ is still given by~\eqref{eq:EquivGroupL'B}.
This implies that the arbitrary-element tuples~$\theta$ and~$\tilde\theta$ are again related
according to~\eqref{eq:EquivGroupbarLB} and~\eqref{eq:EquivGroupbarLC},
and the transformation $\Psi=\tilde\Phi^{-1}\circ\hat\Psi\circ\Phi$ is the composition
of a transformation of the form~\eqref{eq:EquivGroupbarLA}
and a point-symmetry transformation of the system~\smash{$\bar L_{\tilde\theta}$}.
Therefore, any admissible transformation in the subclass~$\bar{\mathcal L}_0$
is the composition of an admissible transformation induced an equivalence transformation of the form~\eqref{eq:EquivGroupbarL}
and an element of the vertex group for the corresponding target arbitrary element.
The only insignificant or gauge equivalence transformation in the subclass~$\bar{\mathcal L}_0$ is the identity transformation.
This means that the subclass~$\bar{\mathcal L}_0$ is semi-normalized in the usual sense,
and its equivalence group is exhausted by the same transformations as that of the subclass~$\bar{\mathcal L}_1$.

Since the equivalence groups of the subclasses~$\bar{\mathcal L}_0$ and~$\bar{\mathcal L}_1$ coincide,
and these subclasses are preserved under the action of~$\mathcal G^\sim_{\bar{\mathcal L}}$ on~$\bar{\mathcal L}$,
the equivalence group of the entire class~$\bar{\mathcal L}$ is the same
and, moreover, this class is semi-normalized in the usual sense.
\end{proof}

\begin{corollary}\label{cor:EquivGroupoidL}
(i) The equivalence group~$G^\sim_{\mathcal L}$ of the class~$\mathcal L$
is the natural projection, on the space with coordinates $(t,\boldsymbol x,A,B)$,
of the subgroup of~$G^\sim_{\bar{\mathcal L}}$ singled out by the constraint~$\boldsymbol h=\boldsymbol0$.
The equivalence groups of the subclasses~$\mathcal L_1$ and~$\mathcal L_0$
coincide with~$G^\sim_{\mathcal L}$.

(ii) The partition of the class~$\mathcal L$ into its subclasses~$\mathcal L_0$ and~$\mathcal L_1$
induces the partition of the groupoid~$\mathcal G^\sim_{\mathcal L}$
into its subgroupoids~$\mathcal G^\sim_{\mathcal L_0}$ and $\mathcal G^\sim_{\mathcal L_1}$,
$\mathcal G^\sim_{\mathcal L}=\mathcal G^\sim_{\mathcal L_0}\sqcup\mathcal G^\sim_{\mathcal L_1}$.

(iii) The subclass~$\mathcal L_1$ is uniformly semi-normalized with respect to linear superposition of solutions.

(iv) The subclass~$\mathcal L_0$ and the entire class~$\mathcal L$ are semi-normalized in the usual sense.
\end{corollary}

\begin{proof}
The optimal way for proving this corollary is to modify the proof of Theorem~\ref{thm:EquivGroupoidbarL},
replacing the classes~$\bar{\mathcal L}$, $\bar{\mathcal L}_0$ and~$\bar{\mathcal L}_1$
by their ``homogeneous'' counterparts~$\mathcal L$, $\mathcal L_0$ and~$\mathcal L_1$, respectively.
We can also assume that
$\Phi$ and $\tilde\Phi$ are the point transformations of the form~\eqref{eq:EquivTransbarLtoL'A} with $\boldsymbol h=\boldsymbol0$.

We have $\boldsymbol f=\boldsymbol0$ and $\,\,\tilde{\!\!\boldsymbol f}=\boldsymbol0$
for any admissible transformation between homogeneous systems in the class~$\bar{\mathcal L}$,
and thus the composition $\boldsymbol h\circ\tilde T$ of the corresponding value of the parameter function~$\boldsymbol h$
with the inverse~$\tilde T$ of~$T$ is a solution of the target system.
The only common solution of the systems in the class~$\mathcal L$
(resp.\ in the subclass~$\mathcal L_0$ or~$\mathcal L_1$) is the zero solution.
Hence $\boldsymbol h=\boldsymbol0$ for elements
of the equivalence groups~$G^\sim_{\mathcal L}$, $G^\sim_{\mathcal L_0}$ and~$G^\sim_{\mathcal L_1}$.
Moreover, the normalization of~$\bar{\mathcal L}_1$ in the usual sense is converted into
the uniform semi-normalization of~$\mathcal L_1$ with respect to linear superposition of solutions.
\end{proof}

Due to appearing the parameter function~$T$ of~$t$ among the parameters of the equivalence group~$G^\sim_{\mathcal L'}$,
we can further gauge the arbitrary elements of the class~$\mathcal L'$.
For any source value of the arbitrary-element matrix~$V$,
an element of~$G^\sim_{\mathcal L'}$ of the form~\eqref{eq:EquivGroupL'} with $C=E$ and $T$ satisfying the equation
\[
n\frac{2T_tT_{ttt}-3T_{tt}^{\,\,2}}{4T_t^{\,2}}+\mathop{\rm tr}V=0, \quad\mbox{or}\quad \{T,t\}=-\frac2n\mathop{\rm tr}V,
\]
maps the system~$L'_V$ to a system~\smash{$L'_{\tilde V}$} with $\mathop{\rm tr}\tilde V=0$.
(Here $\{T,t\}:=T_{ttt}/T_t-\frac32T_{tt}^{\,\,2}/T_t^{\,2}$ is the Schwarzian derivative of~$T$ with respect to~$t$.)
In other words, the class~$\mathcal L'$ is mapped by a wide family of admissible transformations parameterized by~$\mathop{\rm tr}V$
to its subclass~$\mathcal L''$ constituted by the systems of the form~\eqref{basis_SLODE}
with traceless values of the arbitrary-element matrix~$V$.
This particular form is called the \emph{Laguerre--Forsyth canonical form}~\cite[p.~266]{seas1993a}.
Note that the subclass $\mathcal L''_0=\mathcal L''\cap\mathcal L'_0$
consists of the single elementary system~$L'_0$: $\boldsymbol x_{tt}=\boldsymbol0$,
and thus $\mathcal L''_1=\mathcal L''\cap\mathcal L'_1=\mathcal L'\setminus\{L'_0\}$.
We do not reparameterize the class~$\mathcal L''$ in order to avoid loosing the symmetry
between entries of the arbitrary-element matrix~$V$,
and since the constraint $\mathop{\rm tr}V=0$ singling out the subclass~$\mathcal L''$ from the class~$\mathcal L'$
is algebraic, this leads to the appearance insignificant equivalence transformations within the class~$\mathcal L''$,
see Section~\ref{sec:AlgebraicAuxiliaryEqsAndClassReparameterization} for definitions and explanations.
This is why we present the canonical significant equivalence group of the class~$\mathcal L''$
instead of the entire equivalence group;
see \cite[p.~22]{wilc1901a} and \cite[p.~116]{wilc1906A}
for the first informal computation of this group in the case $n=2$
and \cite[p.~266]{seas1993a},
where the natural projection of the equivalence group
of the class~$\mathcal L''_{r,n}$ (in the notation of Section~\ref{sec:OnGeneralizationToArbitraryEqOrder})
to the space with coordinates $(t,\boldsymbol x)$
is given for arbitrary $(r,n)$ using the geometric terminology.

\begin{corollary}\label{cor:EquivGroupoidL''}
(i) The canonical significant equivalence group%
\footnote{%
Following~\cite{kova2023b,kova2023a}, we can consider
the transformations~\eqref{eq:EquivGroupL''} on their natural domains of definition and
take the modified composition of transformations as the group operation in~$G^{{\rm s\vphantom{g}}\sim}_{\mathcal L''}$.
Then $G^{{\rm s\vphantom{g}}\sim}_{\mathcal L''}$ indeed becomes a group rather than a pseudogroup.
}%
~$G^{{\rm s\vphantom{g}}\sim}_{\mathcal L''}$ of the class~$\mathcal L''$
consists of the transformations
\begin{subequations}\label{eq:EquivGroupL''}
\begin{gather}\label{eq:EquivGroupL''A}
\tilde t=\frac{\alpha t+\beta}{\gamma t+\delta}, \quad
\tilde{\boldsymbol x}=\frac1{\gamma t+\delta}C\boldsymbol x,
\\\label{eq:EquivGroupL''B}
\tilde V=(\gamma t+\delta)^4CVC^{-1},
\end{gather}
\end{subequations}
where $\alpha$, $\beta$, $\gamma$ and $\delta$ are arbitrary constants
with $\alpha\delta-\beta\gamma=1$ over~$\mathbb C$ and $\alpha\delta-\beta\gamma=\pm1$ over~$\mathbb R$,
and $C$ is an arbitrary constant invertible $n\times n$ matrix.%
\footnote{%
The parameter tuple $(\alpha,\beta,\gamma,\delta,C)$ is defined up to simultaneously alternating
the signs of its components, $(\alpha,\beta,\gamma,\delta,C)\mapsto(-\alpha,-\beta,-\gamma,-\delta,-C)$.
}
The canonical significant equivalence group of the subclass~$\mathcal L''_1$ coincides with~$G^\sim_{\mathcal L''}$,
whereas a significant equivalence group of the subclass~$\mathcal L''_0$
is the trivial (identity) prolongation of the point-symmetry group of the elementary system~$L'_0$
to the arbitrary-element matrix~$V$.

(ii) The partition of the class~$\mathcal L''$ into its subclasses~$\mathcal L''_0$ and~$\mathcal L''_1$
induces the partition of the groupoid~$\mathcal G^\sim_{\mathcal L''}$
into its subgroupoids~$\mathcal G^\sim_{\mathcal L''_0}$ and $\mathcal G^\sim_{\mathcal L''_1}$,
$\mathcal G^\sim_{\mathcal L''}=\mathcal G^\sim_{\mathcal L''_0}\sqcup\mathcal G^\sim_{\mathcal L''_1}$.

(iii) The subclass~$\mathcal L''_1$ is uniformly semi-normalized with respect to linear superposition of solutions.

(iv) The subclass~$\mathcal L''_0$ is trivially normalized in the usual sense
since \smash{$\mathcal G^\sim_{\mathcal L''_0}=\mathcal G_0$},
and thus the entire class~$\mathcal L''$ is semi-normalized in the usual sense as well.
\end{corollary}

\begin{proof}
In fact, this corollary is a consequence of the proof of Theorem~\ref{thm:EquivGroupoidL'} rather than of its statement.
In particular, the analysis of this proof shows that imposing the constraint $\mathop{\rm tr}V=0$
has no influence on splitting the systems~\eqref{eq:L'PrologationOfEquivTransToDerivativesOfV},
\eqref{eq:EquivGroupL'DetEqs0} and~\eqref{eq:EquivGroupoidL'DetEqs0}
with respect to the first derivatives~$x^a_t$ and the transformed arbitrary elements~$\tilde V^{ab}$.
Therefore, up to insignificant summands, which depend on $(t,\boldsymbol x,V)$ and vanish whenever $\mathop{\rm tr}V=0$,
each equivalence transformation in the class~$\mathcal L''$ is of the form~\eqref{eq:EquivGroupL'}.
Taking into account the constraints $\mathop{\rm tr}V=0$ and $\mathop{\rm tr}\tilde V=0$,
we derive from~\eqref{eq:EquivGroupL'B}
that the Schwarzian derivative of the function~$T$ vanishes, i.e., the function $T$ is fractional linear in~$t$.
\end{proof}

It is natural to call~$\bar{\mathcal L}_0$ and~$\bar{\mathcal L}_1$
the \emph{singular} and the \emph{regular} subclasses of~$\bar{\mathcal L}$.
Analogously, we define the \emph{singular} and the \emph{regular} subclasses
for the classes~$\mathcal L$, $\mathcal L'$ and~$\mathcal L''$.

\section{Equivalence algebras}\label{sec:SLODEEquivAlgebras}

The equivalence algebra~$\mathfrak g^\sim_{\mathcal K}$ of a class~$\mathcal K$ of systems of differential equations
consists of the vector fields (on the associated space with the independent and dependent variables and the arbitrary elements)
that are infinitesimal generators of one-parameter subgroups of the corresponding equivalence group~$G^\sim_{\mathcal K}$.
This is why it is needless to use the infinitesimal Lie method for computing~$\mathfrak g^\sim_{\mathcal K}$
if the group~$G^\sim_{\mathcal K}$ has been found,
which is the case for the classes~$\bar{\mathcal L}$, $\mathcal L$, $\mathcal L'$, $\mathcal L''$
and their singular and regular subclasses.

\begin{corollary}\label{cor:EquivAlgebrabarL}
The equivalence algebra~$\mathfrak g^\sim_{\bar{\mathcal L}}$ of the class~$\bar{\mathcal L}$ is spanned by of the vector fields
\begin{gather*}
\tau\p_t-\tau_tA^{ab}\p_{A^{ab}}-\tau_{tt}\p_{A^{aa}}-2\tau_tB^{ab}\p_{B^{ab}}-2\tau_tf^a\p_{f^a},
\\
\eta^{ab}x^b\p_{x^a}+(\eta^{ac}A^{cb}\!-\!A^{ac}\eta^{cb}\!+2\eta^{ab}_t)\p_{A^{ab}}
+(\eta^{ac}B^{cb}\!-\!B^{ac}\eta^{cb}\!-\!A^{ac}\eta^{cb}_t\!+\eta^{ab}_{tt})\p_{B^{ab}}+\eta^{ab}f^b\p_{f^a},
\\
\chi^a\p_{x^a}+(\chi^a_{tt}-A^{ab}\chi^b_t-B^{ab}\chi^b)\p_{f^a},
\end{gather*}
where $\tau$, $\eta^{ab}$ and~$\chi^a$ are arbitrary functions of $t$.
\end{corollary}

\begin{proof}
To construct families of vector fields jointly constituting a spanning set for~$\mathfrak g^\sim_{\bar{\mathcal L}}$,
we compute the infinitesimal generators for a special set of one-parameter subgroups of~$G^\sim_{\bar{\mathcal L}}$.
More specifically, we successively take one of the (scalar-, matrix- and vector-valued)
parameter functions~$T$, $H$ and~$\boldsymbol h$
to depend on a continuous subgroup parameter~$\varepsilon$,
set the other parameter functions to their values associated with the identity transformation,
which are $t$, $E$ and~$\boldsymbol0$ for $T$, $H$ and~$\boldsymbol h$, respectively,
differentiate the transformation components with respect to~$\varepsilon$ and evaluate the result at $\varepsilon=0$.
This gives the corresponding components of the required infinitesimal generators.
After denoting
\[
\tau:=\frac{\mathrm dT}{\mathrm d\varepsilon}\,\bigg|_{\varepsilon=0},\quad
\eta^{ab}:=\frac{\mathrm dH^{ab}}{\mathrm d\varepsilon}\,\bigg|_{\varepsilon=0},\quad
\chi^a:=\frac{\mathrm dh^a}{\mathrm d\varepsilon}\,\bigg|_{\varepsilon=0},
\]
we obtain the vector fields presented in the statement.
\end{proof}

\begin{corollary}\label{cor:EquivAlgebraL}
The equivalence algebra~$\mathfrak g^\sim_{\mathcal L}$ of the class~$\mathcal L$ is spanned by of the vector fields
\begin{gather*}
\tau\p_t-\tau_tA^{ab}\p_{A^{ab}}-\tau_{tt}\p_{A^{aa}}-2\tau_tB^{ab}\p_{B^{ab}},
\\
\eta^{ab}x^b\p_{x^a}+(\eta^{ac}A^{cb}\!-\!A^{ac}\eta^{cb}\!+2\eta^{ab}_t)\p_{A^{ab}}
+(\eta^{ac}B^{cb}\!-\!B^{ac}\eta^{cb}\!-\!A^{ac}\eta^{cb}_t\!+\eta^{ab}_{tt})\p_{B^{ab}},
\end{gather*}
where $\tau$ and~$\eta^{ab}$ are arbitrary functions of $t$.
\end{corollary}

\begin{proof}
To construct the algebra~$\mathfrak g^\sim_{\mathcal L}$,
it suffices to set the parameter functions~$\chi^a$ to zero
and project the obtained vector fields to the space with coordinates $(t,\boldsymbol x,A,B)$.
\end{proof}

\begin{corollary}\label{cor:EquivAlgebraL'}
The equivalence algebra~$\mathfrak g^\sim_{\mathcal L'}$ of the class~$\mathcal L'$ is spanned by of the vector fields
\begin{gather*}
\hat D(\tau)=\tau\p_t+\tfrac12\tau_tx^a\p_{x^a}-2\tau_tV^{ab}\p_{V^{ab}}+\tfrac12\tau_{ttt}\p_{V^{aa}},\\
\hat I^{ab}=x^b\p_{x^a}+V^{bc}\p_{V^{ac}}-V^{ca}\p_{V^{cb}},
\end{gather*}
where $\tau$ is an arbitrary function of $t$.
\end{corollary}

\begin{proof}
In contrast to the proof of Corollary~\ref{cor:EquivAlgebrabarL},
here we find all the infinitesimal generators~$Q$ of the equivalence group~$G^\sim_{\mathcal L'}$.
We take an arbitrary one-parameter subgroup of this group,
assuming that in~\eqref{eq:EquivGroupL'}, the function~$T$ and the matrix~$C$ additionally depend on a parameter~$\varepsilon$,
i.e., $T=T(t,\varepsilon)$ and $C=C(\varepsilon)$, and $T(t,0)=t$ and $C(0)=E$.
Denote $\tau(t):=T_\varepsilon(t,0)$ and $\Gamma:=C_\varepsilon(0)$.
Differentiating the general element of the subgroup with respect to~$\varepsilon$ at $\varepsilon=0$,
we derive the general form of infinitesimal generators for~$G^\sim_{\mathcal L'}$,
\begin{gather*}
Q=\tau\p_t+\big(\tfrac12\tau_tx^a+\Gamma^{ab}x^b\big)\p_{x^a}+\big(\tfrac12\tau_{ttt}\delta_{ab}-2\tau_tV^{ab}+[\Gamma,V]^{ab}\big)\p_{V^{ab}},
\end{gather*}
where $\tau=\tau(t)$ is an arbitrary function of~$t$,
$\Gamma=(\Gamma^{ab})$ is an arbitrary constant $n\times n$ matrix, and $\delta_{ab}$ is the Kronecker delta.
\end{proof}

\begin{corollary}
A basis of the significant equivalence algebra~$\mathfrak g^{{\rm s\vphantom{g}}\sim}_{\mathcal L''}$
of the class~$\mathcal L''$ consists of the vector fields
\begin{gather*}
\hat D(1)=\p_t, \quad
\hat D(t)=t\p_t+\tfrac12x^a\p_{x^a}-2V^{ab}\p_{V^{ab}}, \quad
\hat D(t^2)=t^2\p_t+tx^a\p_{x^a}-4tV^{ab}\p_{V^{ab}},\\
\hat I^{ab}=x^b\p_{x^a}+V^{bc}\p_{V^{ac}}-V^{ca}\p_{V^{cb}}.
\end{gather*}
\end{corollary}

\begin{proof}
Since the significant equivalence group~$G^{{\rm s\vphantom{g}}\sim}_{\mathcal L''}$ of the class~$\mathcal L''$
is a subgroup of~$G^\sim_{\mathcal L'}$ that is singled out by the condition that the parameter function $T$ is fractional linear in~$t$,
the algebra~$\mathfrak g^{{\rm s\vphantom{g}}\sim}_{\mathcal L''}$ is the subalgebra of~$\mathfrak g^\sim_{\mathcal L'}$ with
$\tau$ running $\langle1,t,t^2\rangle$.
\end{proof}

It obviously follows from assertions of Section~\ref{sec:SLODEEquivGroupoidsAndEquivGroups} and of this section that \
$\mathfrak g^\sim_{\bar{\mathcal L}_0}=\mathfrak g^\sim_{\bar{\mathcal L}_1}=\mathfrak g^\sim_{\bar{\mathcal L}}$, \
$\mathfrak g^\sim_{\mathcal L_0}=\mathfrak g^\sim_{\mathcal L_1}=\mathfrak g^\sim_{\mathcal L}$, \
$\mathfrak g^{{\rm s\vphantom{g}}\sim}_{\mathcal L'_0}=\mathfrak g^\sim_{\mathcal L'_1}=\mathfrak g^\sim_{\mathcal L'}$ \ and \
$\mathfrak g^{{\rm s\vphantom{g}}\sim}_{\mathcal L''_1}=\mathfrak g^{{\rm s\vphantom{g}}\sim}_{\mathcal L''}$.

\section{Preliminary analysis of Lie symmetries}\label{sec:SLODEPreliminaryAnalysisOfLieSyms}

For each pair $(\mathcal K,\tilde{\mathcal K})$ of classes in each of the chains
\begin{gather}\label{eq:SLODESubclassEmbeddings}
\bar{\mathcal L}\hookleftarrow\mathcal L\hookleftarrow\mathcal L'\supset\mathcal L'',\quad
\bar{\mathcal L}_0\hookleftarrow\mathcal L_0\hookleftarrow\mathcal L'_0\supset\mathcal L''_0,\quad
\bar{\mathcal L}_1\hookleftarrow\mathcal L_1\hookleftarrow\mathcal L'_1\supset\mathcal L''_1,
\end{gather}
where $\tilde{\mathcal K}$ is a subsequent class for~$\mathcal K$,
any system from the class~$\mathcal K$ is mapped by an element of the equivalence group~$G^\sim_{\mathcal K}$
to a system from the class~$\tilde{\mathcal K}$.
In other words, the class~$\mathcal K$ is mapped to its subclass~$\tilde{\mathcal K}$
by a wide family of admissible transformations from the action groupoid of the equivalence group of~$\mathcal K$.
Moreover, the \smash{$G^\sim_{\mathcal K}$}- and \smash{$G^\sim_{\tilde{\mathcal K}}$}-equivalences
within the classes~$\mathcal K$ and~$\tilde{\mathcal K}$, respectively,
are consistent in the sense that systems from the class~$\mathcal K$ are $G^\sim_{\mathcal K}$-equivalent
if and only if their counterparts in the class~$\tilde{\mathcal K}$ are \smash{$G^\sim_{\tilde{\mathcal K}}$}-equivalent.
This follows from the fact that any class~$\mathcal K$ among the above ones is semi-normalized in the usual sense,
and thus the $G^\sim_{\mathcal K}$-equivalence coincides with the $\mathcal G^\sim_{\mathcal K}$-equivalence.
As a result, the solution of the group classification problem for the class~$\mathcal K$
reduces to that for its subclass~$\tilde{\mathcal K}$,
and it is not essential which kind of equivalences ($G^\sim$-equivalence or $\mathcal G^\sim$-equivalence) is used.
Since the chain members on the same position, except the singleton~$\mathcal L''_0$,
have the same equivalence groups (up to neglecting insignificant equivalence transformations within singular subclasses),
the group classification of a class in the first chain splits into
the group classification of the respective classes in the second and third chains.\looseness=-1

The group classification of each class from the second chain is trivial,
and the optimal classification list consists of the single elementary system $\boldsymbol x_{tt}=\boldsymbol0$
with its well-known maximal Lie invariance algebra~$\mathfrak g_0$ isomorphic to $\mathfrak{sl}(n+2,\mathbb F)$,
\begin{gather}\label{eq:SLODEMIAOfElementarySystem}
\mathfrak g_0=\langle\p_t,\,\p_{x^a},\,t\p_t,\,x^a\p_t,\,t\p_{x^a},\,x^a\p_{x^b},\,tx^a\p_t+x^ax^c\p_{x^c},\,t^2\p_t+tx^c\p_{x^c}\rangle;
\end{gather}
the corresponding complete point-symmetry group~$G_0$ is the general projective group of $\mathbb F^{n+1}$
consisting of the transformations with components \mbox{\cite[S.~554]{lie1888A}}
\[
\tilde x^\iota=\frac{\alpha_{\iota0}x^0+\dots+\alpha_{\iota n}x^n+\alpha_{\iota,n+1}}{\alpha_{n+1,0}x^0+\dots+\alpha_{n+1,n}x^n+\alpha_{n+1,n+1}},\qquad \iota=0,\dots,n,
\]
where $\alpha_{00}$, $\alpha_{01}$, \dots, $\alpha_{n+1,n+1}$ are homogeneous group parameters, and $x^0=t$.
See~\cite{shap2014a} for computing~$G_0$ by the direct method.

At the same time, we will show that for the third (and hence the first) chain,
it is the most convenient to use the class~$\mathcal L'_1$  (resp.\ the class~$\mathcal L'$)
as basic in the course of group classification, switching to the class~$\mathcal L''_1$ for classifying some particular cases
and to the class~$\mathcal L_1$ for interpreting certain results,
see Section~\ref{sec:DescriptionOfEssLieSymExtensions} and, especially, Remark~\ref{rem:SLODEConvenientClassesForGroupClassification} below.
Any system~$\bar L_\theta$ from the class~$\bar{\mathcal L}_1$ is mapped to its homogeneous counterpart
by the simple subtraction of a particular solution of~$\bar L_\theta$ from~$\boldsymbol x$,
and it is natural to choose homogeneous systems
as canonical representatives of $G^\sim_{\bar{\mathcal L_1}}$-cosets in~$\bar{\mathcal L}_1$.
This is why we start the consideration of Lie symmetries with the systems from the class~$\mathcal L_1$
and then we carry out a preliminary analysis of Lie symmetries for systems from the classes~$\mathcal L'_1$ and~$\mathcal L''_1$ as well.
Denote by~$\vartheta$ the arbitrary-element tuple $(A,B)$ of the class~$\mathcal L_1$, $\vartheta=(A,B)$.

Instead of applying the standard technique based on the infinitesimal Lie-invariance criterion,
it is easier to derived determining equations for Lie symmetries of systems
from the class~$\mathcal L_1$ (resp.\ from its subclasses~$\mathcal L'_1$ and~$\mathcal L''_1$)
using results of Section~\ref{sec:SLODEEquivGroupoidsAndEquivGroups}.

\begin{lemma}\label{lem:MAIofLvartheta}
The maximal Lie invariance algebra~$\mathfrak g_\vartheta$ of a system~$L_\vartheta$ from the class~$\mathcal L_1$
consists of the vector fields of the form
\begin{gather}\label{eq:GenElementOfInvAlgOfLvartheta}
Q=\tau\p_t+(\eta^{ab}x^b+\chi^a)\p_{x^a},
\end{gather}
where the vector-valued function $\boldsymbol\chi=(\chi^1,\dots,\chi^n)^{\mathsf T}$ of~$t$ is an arbitrary solution of the system~$L_\vartheta$,
whereas $\tau$ is an arbitrary function of~$t$
and $\eta=(\eta^{ab})$ is an arbitrary $n\times n$ matrix-valued function of~$t$ that satisfy the classifying condition
\begin{subequations}\label{eq:ClassifyingCondLvartheta}
\begin{gather}\label{eq:ClassifyingCondLvarthetaA}
\tau A_t=[\eta,A]-\tau_tA+2\eta_t-\tau_{tt}E,
\\\label{eq:ClassifyingCondLvarthetaB}
\tau B_t=[\eta,B]-2\tau_tB-A\eta_t+\eta_{tt}.
\end{gather}
\end{subequations}
\end{lemma}

\begin{proof}
Since the subclass~$\mathcal L_1$ is uniformly semi-normalized with respect to linear superposition of solutions
in view of item~(iii) of Corollary~\ref{cor:EquivGroupoidL},
any element of the vertex group $\mathcal G_\vartheta\subset\mathcal G^\sim_{\mathcal L_1}$
is the composition of an element of the action groupoid~\smash{$\mathcal G^{G^\sim_{\mathcal L_1}}$}
of the equivalence group~$G^\sim_{\mathcal L_1}$
and of an element of the groupoid \smash{$\mathcal G^{\rm lin}_{\mathcal L_1}$}
related to the linear superposition of solutions within the class~$\mathcal L_1$.
Therefore, the point-symmetry group~$G_\vartheta$ of the system~$L_\vartheta\in\mathcal L_1$
consists of the transformations of the form~\eqref{eq:EquivGroupbarLA},
where $\boldsymbol h$ is a particular solution of~$L_\vartheta$,
and the parameter function~$T$ and the matrix-valued parameter function~$H$ satisfy the equations
\begin{subequations}\label{eq:SymGroupConditionForL1}
\begin{gather}\label{eq:SymGroupConditionForL1A}
A\circ T=\frac1{T_t^{\,2}}\big(T_tHA+2T_tH_t-T_{tt}H\big)H^{-1}, \\
\label{eq:SymGroupConditionForL1B}
B\circ T=\frac1{T_t^{\,3}}\big(T_tHB-T_t^{\,2}(A\circ T)H_t+T_tH_{tt}-T_{tt}H_t\big)H^{-1}.
\end{gather}
\end{subequations}
These conditions for $T$, $H$ and~$\boldsymbol h$ follow from~\eqref{eq:EquivGroupbarLB}--\eqref{eq:EquivGroupbarLD}
under the substitution $\tilde A=A\circ T$, $\tilde B=B\circ T$, $\boldsymbol f=\boldsymbol0$ and $\,\,\tilde{\!\!\boldsymbol f}= \boldsymbol0$.
Here $A\circ T$ and~$B\circ T$ denote the matrix-valued functions~$A$ and~$B$ composed with the scalar function~$T$,
respectively.

Analogously to the proof of Corollary~\ref{cor:EquivAlgebraL'},
we compute all the infinitesimal generators~$Q$ of the group~$G_\vartheta$.
We take an arbitrary one-parameter subgroup of this group,
assuming that in~\eqref{eq:EquivGroupbarLA},
the function~$T$, the matrix-valued function~$H$ and the vector-valued function~$\boldsymbol h$ additionally depend on a parameter~$\varepsilon$,
i.e., $T=T(t,\varepsilon)$, $H=H(t,\varepsilon)$ and $\boldsymbol h=\boldsymbol h(t,\varepsilon)$, and $T(t,0)=t$, $H(t,0)=E$ and $\boldsymbol h(t,0)=\boldsymbol0$.
Denote $\tau(t):=T_\varepsilon(t,0)$, $\eta(t):=H_\varepsilon(t,0)$ and $\boldsymbol\chi(t):=\boldsymbol h_\varepsilon(t,0)$.
Differentiating the general element of the subgroup with respect to~$\varepsilon$ at $\varepsilon=0$,
we obtain the general form~\eqref{eq:GenElementOfInvAlgOfLvartheta} of infinitesimal generators~$Q$ of the group~$G_\vartheta$.
The classifying condition~\eqref{eq:ClassifyingCondLvartheta} is the infinitesimal counterpart of~\eqref{eq:SymGroupConditionForL1}.
\end{proof}

\begin{corollary}\label{cor:MAIofBarLtheta}
The maximal Lie invariance algebra~$\mathfrak g_\theta$ of a system~$L_\theta$ from the class~$\bar{\mathcal L}_1$
is constituted by the vector fields of the form
\[
Q=\tau\p_t+(\eta^{ab}x^b+\tau h^a_t-\eta^{ab}h^b+\chi^a)\p_{x^a},
\]
where
the vector-valued function $\boldsymbol h=(h^1,\dots,h^n)^{\mathsf T}$ of~$t$ is a fixed particular solution of the system~$L_\theta$,
the vector-valued function $\boldsymbol\chi=(\chi^1,\dots,\chi^n)^{\mathsf T}$ of~$t$
is an arbitrary solution of the corresponding homogeneous system,
whereas $\tau$ is an arbitrary function of~$t$
and $\eta=(\eta^{ab})$ is an arbitrary $n\times n$ matrix-valued function of~$t$ that satisfy the classifying condition~\eqref{eq:ClassifyingCondLvartheta}.
\end{corollary}

In a similar way, we can derive
the more specified assertions on Lie symmetries of systems from the classes~$\mathcal L'_1$ and~$\mathcal L''_1$
from Theorem~\ref{thm:EquivGroupoidL'} and Corollary~\ref{cor:EquivGroupoidL''}, respectively.
However, it is easier to successively obtain these assertions as direct consequences of Lemma~\ref{lem:MAIofLvartheta}
under additional constraints for the arbitrary elements~$A$ and~$B$.

\begin{corollary}\label{cor:MAIofL'_V}
The maximal Lie invariance algebra~$\mathfrak g_V^{}$ of a system~$L'_V$ from the class~$\mathcal L'_1$
consists of the vector fields of the form
\begin{equation*}
Q=\tau\p_t+\big(\tfrac12\tau_tx^a+\Gamma^{ab}x^b+\chi^a\big)\p_{x^a},
\end{equation*}
where the vector-valued function $\boldsymbol\chi=(\chi^1,\dots,\chi^n)^{\mathsf T}$ of~$t$ is an arbitrary solution of the system~$L'_V$,
whereas $\tau$ is an arbitrary function of~$t$
and $\Gamma=(\Gamma^{ab})$ is an arbitrary constant $n\times n$ matrix that satisfy the classifying condition
\begin{gather}\label{eq:ClassifyingCondL'_V}
\tau V_t=[\Gamma,V]-2\tau_tV+\tfrac12\tau_{ttt}E.
\end{gather}
\end{corollary}

\begin{proof}
Setting $A=0$ in the equation~\eqref{eq:ClassifyingCondLvarthetaA},
we get $\eta_t=\frac12\tau_{tt}E$, i.e., $\eta=\frac12\tau_tE+\Gamma$,
where $\Gamma$ is the matrix of integration constants.
Then, after re-denoting $B=V$, the equation~\eqref{eq:ClassifyingCondLvarthetaA}
implies the classifying condition~\eqref{eq:ClassifyingCondL'_V}.
\end{proof}

\begin{corollary}\label{cor:MAIofL''_V}
If a system~$L'_V$ belongs to the narrower class~$\mathcal L''_1$,
then its maximal Lie invariance algebra~$\mathfrak g_V^{}$ satisfies Corollary~\ref{cor:MAIofL'_V},
where additionally $\tau_{ttt}=0$.
\end{corollary}

\begin{proof}
Under the constraint $\mathop{\rm tr}V=0$ singling out the subclass~$\mathcal L''_1$ from the class~$\mathcal L'_1$,
computing the trace of the matrix condition~\eqref{eq:ClassifyingCondL'_V} implies $\tau_{ttt}=0$.
\end{proof}

\begin{lemma}\label{lem:SLODEKernelPointSymGroupsOfAllClasses}
The kernel point-symmetry groups of systems from each of the classes under consideration are the following:
\begin{gather*}
G^\cap_{\mathcal L''_1}=G^\cap_{\mathcal L''}=G^\cap_{\mathcal L'_1}=G^\cap_{\mathcal L'}
=G^\cap_{\mathcal L_1}=G^\cap_{\mathcal L}=G^\cap_{\mathcal L_0}
=\big\{\Phi\colon\tilde t=t,\,\tilde{\boldsymbol x}=\gamma\boldsymbol x\mid\gamma\in\mathbb F\setminus\{0\}\big\},
\\
G^\cap_{\mathcal L''_0}=G_0,\quad
G^\cap_{\mathcal L'_0}=\big\{\Phi\colon\tilde t=t,\,\tilde{\boldsymbol x}=C\boldsymbol x\mid C\in{\rm GL}(n,\mathbb F)\big\},
\\
G^\cap_{\bar{\mathcal L}}=G^\cap_{\bar{\mathcal L}_0}=G^\cap_{\bar{\mathcal L}_1}
=\big\{{\rm id}\colon\tilde t=t,\,\tilde{\boldsymbol x}=\boldsymbol x\big\}.
\end{gather*}
\end{lemma}

\begin{proof}
We only present some hints on computing these groups.
Since the class~$\mathcal L''_0$ consists of the single system~$L'_0$,
the computation of its equivalence group~\smash{$G^\cap_{\mathcal L''_0}$} is trivial.

For the classes~$\mathcal L'_0$ and~$\mathcal L''_1$, we consider
the determining equations~\eqref{eq:EquivGroupoidL'DetEqs0} with $\tilde V^{ab}=V^{ab}\circ T$
and additionally replace $V^{ab}$ by $\lambda V^{ab}$, where $\lambda$ is an arbitrary constant.
We can repeat the splitting procedure presented after~\eqref{eq:EquivGroupoidL'DetEqs0}
in the modified determining equations
and derive the expressions~\eqref{eq:EquivGroupL'A} for the components of transformations
from the groups~\smash{$G^\cap_{\mathcal L'_0}$} and~\smash{$G^\cap_{\mathcal L''_1}$}.
The splitting of the modified condition~\eqref{eq:EquivGroupL'B} with respect to the parameter~$\lambda$
leads to the constraint $CV=T_t^{\,2}(V\circ T)C$, which implies $T_t=\pm1$
via taking the determinants of the left- and right-hand sides for constant invertible values of~$V$.
Moreover, for~$\mathcal L''_1$ after varying constant values of~$V$, we get $C=\gamma E$ for some nonzero constant~$\gamma$.
Taking special values of~$V$, e.g., $V=tK$ with an arbitrary constant nonzero traceless matrix~$K$,
we also derive that $T=t$.

Due to the subclass embeddings~\eqref{eq:SLODESubclassEmbeddings},
we have that
\begin{gather*}
G^\cap_{\mathcal L''_1}\supseteq G^\cap_{\mathcal L'_1}\supseteq G^\cap_{\mathcal L_1}\supseteq G^\cap_{\bar{\mathcal L}_1},
\quad
G^\cap_{\mathcal L'_0}\supseteq G^\cap_{\mathcal L_0}\supseteq G^\cap_{\bar{\mathcal L}_0},
\\
G^\cap_{\mathcal L''}=G^\cap_{\mathcal L''_0}\cap G^\cap_{\mathcal L''_1}, \quad
G^\cap_{\mathcal L'}=G^\cap_{\mathcal L'_0}\cap G^\cap_{\mathcal L'_1}, \quad
G^\cap_{\mathcal L}=G^\cap_{\mathcal L_0}\cap G^\cap_{\mathcal L_1}, \quad
G^\cap_{\bar{\mathcal L}}=G^\cap_{\bar{\mathcal L}_0}\cap G^\cap_{\bar{\mathcal L}_1}.
\end{gather*}
The direct computation shows that in fact the equalities
\smash{$G^\cap_{\mathcal L'_1}=G^\cap_{\mathcal L_1}=G^\cap_{\mathcal L''_1}$} hold true.
In view of the inclusion \smash{$G^\cap_{\mathcal L'_0}\supseteq G^\cap_{\mathcal L_0}$},
the further constraints for elements of~\smash{$G^\cap_{\mathcal L_0}$} are obtained
by splitting~\eqref{eq:EquivGroupbarLB} with $T=t$, $H=C$ and $\tilde A=A$ with respect to~$A$.
This gives $C=\gamma E$ for an arbitrary nonzero constant~$\gamma$, i.e., \smash{$G^\cap_{\mathcal L_0}=G^\cap_{\mathcal L'_1}$}.
It follows from $G^\cap_{\mathcal L_1}\supseteq G^\cap_{\bar{\mathcal L}_1}$ and $G^\cap_{\mathcal L_0}\supseteq G^\cap_{\bar{\mathcal L}_0}$
that for the common point symmetries of systems from the class~$\bar{\mathcal L}_1$ (resp.\ $\bar{\mathcal L}_0$),
the equation~\eqref{eq:EquivGroupbarLD} degenerates to the equation $(1-\gamma)\boldsymbol f=\boldsymbol0$, i.e., $\gamma=1$.
\end{proof}

\begin{corollary}\label{cor:SLODEKernelMIAsOfAllClasses}
The kernel Lie invariance algebras of systems from each of the classes under consideration are the following:
\begin{gather*}
\mathfrak g^\cap_{\mathcal L''_1}=\mathfrak g^\cap_{\mathcal L''}=\mathfrak g^\cap_{\mathcal L'_1}=\mathfrak g^\cap_{\mathcal L'}
=\mathfrak g^\cap_{\mathcal L_1}=\mathfrak g^\cap_{\mathcal L}=\mathfrak g^\cap_{\mathcal L_0}
=\langle x^a\p_{x^a}\rangle,
\quad
\mathfrak g^\cap_{\mathcal L''_0}=\mathfrak g_0,\quad
\mathfrak g^\cap_{\mathcal L'_0}=\langle x^b\p_{x^a}\rangle,
\\
\mathfrak g^\cap_{\bar{\mathcal L}}=\mathfrak g^\cap_{\bar{\mathcal L}_0}=\mathfrak g^\cap_{\bar{\mathcal L}_1}=\{0\}.
\end{gather*}
\end{corollary}

\begin{proof}
These algebras can be obtained as the set of infinitesimal generators
of one-parameter subgroups of the corresponding kernel point-symmetry groups.
The algebras~$\mathfrak g^\cap_{\mathcal L_1}$, $\mathfrak g^\cap_{\mathcal L'_1}$ and $\mathfrak g^\cap_{\mathcal L''_1}$
can also be computed using Lemma~\ref{lem:MAIofLvartheta} and Corollaries~\ref{cor:MAIofL'_V} and~\ref{cor:MAIofL''_V}
via respectively splitting the classifying conditions~\eqref{eq:ClassifyingCondLvartheta}, \eqref{eq:ClassifyingCondL'_V}
and \eqref{eq:ClassifyingCondL'_V} with $\tau_{ttt}=0$ with respect to the corresponding arbitrary elements.
\end{proof}

\begin{remark}
Although systems from the class~$\mathcal L$ are linear and homogeneous,
an operator technique~\cite{fush1994A,mill1977A} allows to find all Lie symmetries
(up to the trivial ones associated with linear superposition of solutions)
only for the systems from the regular subclass~$\mathcal L_1$.
\end{remark}

Lemma~\ref{lem:MAIofLvartheta} implies that for any system~$L_\vartheta$ from the class~$\mathcal L_1$,
its maximal Lie invariance algebra~$\mathfrak g_\vartheta^{}$ can be represented as the semidirect sum
$\mathfrak g_\vartheta^{}=\mathfrak g^{\rm ess}_\vartheta\lsemioplus\mathfrak g^{\rm lin}_\vartheta$.
Here
\[
\mathfrak g^{\rm lin}_\vartheta:=\big\{\chi^a(t)\p_{x^a}\mid\boldsymbol\chi=(\chi^1,\dots,\chi^n)^{\mathsf T}\ \mbox{is a solution of}\ L_\vartheta\big\}
\]
is the $2n$-dimensional abelian ideal associated with the \emph{linear superposition of solutions} of the system~$L_\vartheta$,
and the complementary subalgebra
\[
\mathfrak g^{\rm ess}_\vartheta:=\big\{\tau(t)\p_t+\eta^{ab}(t)x^b\p_{x^a}\mid(\tau,\eta)\ \mbox{is a solution of}\ \eqref{eq:ClassifyingCondLvartheta}\big\}
\]
is called the \emph{essential} Lie invariance algebra of the system~$L_\vartheta$, cf.\ \cite{kuru2018a,popo2008a}.
It is obvious that for any system~$L_\vartheta$ from the class~$\mathcal L$,
we have the representation $\mathfrak g^{\rm ess}_\vartheta=\mathfrak g_\vartheta^{}\cap\varpi_*\mathfrak g^\sim_{\mathcal L}$,
where $\varpi$ denotes the natural projection
from \smash{$\mathbb F_t\times\mathbb F^n_{\boldsymbol x}\times\mathbb F^{2n^2}_\vartheta$} onto $\mathbb F_t\times\mathbb F^n_{\boldsymbol x}$.
Moreover,
the algebra~$\mathfrak g^{\rm ess}_\vartheta$ necessarily contains the vector field $I:=x^a\p_{x^a}$, and thus
$\mathfrak g_\vartheta^{}\supseteq\langle I\rangle\lsemioplus\mathfrak g^{\rm lin}_\vartheta$.
The classifying condition~\eqref{eq:ClassifyingCondLvartheta} also implies
that $\mathfrak g^{\rm ess}_\vartheta=\langle I\rangle$ and
$\mathfrak g_\vartheta^{}=\langle I\rangle\lsemioplus\mathfrak g^{\rm lin}_\vartheta$
for general systems in the class~$\mathcal L_1$.
Therefore, the minimum dimension of the maximal Lie invariance algebras of systems from the class~$\mathcal L_1$
and, hereupon, from the class~$\mathcal L$ is equal to~$2n+1$.

In view of the above properties, which are rather common for classes of homogeneous linear systems of differential equations,
it is more natural to classify the essential Lie invariance algebras for systems from the class~$\mathcal L_1$
instead of their maximal Lie invariance algebras.

\begin{remark}\label{rem:EssLieInvAlgebrasInL'1AndBarL1}
The definition of the essential Lie invariance algebra~$\mathfrak g^{\rm ess}_V$
of the system~$L'_V$ from the class~$\mathcal L'_1$
is analogous the definition of~$\mathfrak g^{\rm ess}_\vartheta$,
\[
\mathfrak g^{\rm ess}_V:=\big\{\tau(t)\p_t+\big(\tfrac12\tau_tx^a+\Gamma^{ab}x^b\big)\p_{x^a}\mid\tau\ \mbox{and}\ \Gamma\ \mbox{satisfy}\ \eqref{eq:ClassifyingCondL'_V}\big\}.
\]
We also have the representations
$\mathfrak g_V^{}=\mathfrak g^{\rm ess}_V\lsemioplus\mathfrak g^{\rm lin}_V$ and
$\mathfrak g^{\rm ess}_V=\mathfrak g_V^{}\cap\varpi_*\mathfrak g^\sim_{\mathcal L'}$,
where now $\varpi$ denotes the natural projection
from $\mathbb F_t\times\mathbb F^n_{\boldsymbol x}\times\mathbb F^{n^2}_V$ onto $\mathbb F_t\times\mathbb F^n_{\boldsymbol x}$.
On the other hand, for any inhomogeneous system~$\bar L_\theta$ from the class~$\bar{\mathcal L}_1$,
its essential Lie invariance algebra cannot be canonically defined and, therefore, is not unique.
More specifically, the ideal~$\mathfrak g^{\rm lin}_\theta$ of the maximal Lie invariance algebra~$\mathfrak g_\theta$ of~$\bar L_\theta$
is associated with the possibility of adding solutions of the homogeneous counterpart~$\bar L_{\theta_0}\simeq L_\vartheta$ of~$\bar L_\theta$
to solutions of~$\bar L_\theta$,\looseness=-1
\[
\mathfrak g^{\rm lin}_\theta:=\big\{\chi^a(t)\p_{x^a}\mid\boldsymbol\chi=(\chi^1,\dots,\chi^n)^{\mathsf T}\ \mbox{is a solution of}\ L_\vartheta\big\}.
\]
Here we mean $\theta_0:=(A,B,0)$ and $\vartheta:=(A,B)$ for $\theta=(A,B,\boldsymbol f)$.
The system~$\bar L_\theta$ is reduced to the system~$\bar L_{\theta_0}$
by the point transformation~$\Phi_\chi$: $\tilde t=t$, $\tilde{\boldsymbol x}=\boldsymbol x-\boldsymbol\chi(t)$,
where $\boldsymbol\chi$ is an arbitrary particular solution of~$\bar L_\theta$.
For any particular solution~$\boldsymbol\chi$ of~$\bar L_\theta$,
the pullback $(\Phi_\chi)^*\mathfrak g^{\rm ess}_{\theta_0}$ of $\mathfrak g^{\rm ess}_{\theta_0}:=\mathfrak g^{\rm ess}_\vartheta$ by~$\Phi_\chi$
is a complementary subalgebra to~$\mathfrak g^{\rm lin}_\theta$ in $\mathfrak g_\theta^{}$.
Moreover, every complementary subalgebra~$\mathfrak g^{\rm ess}_\theta$ to~$\mathfrak g^{\rm lin}_\theta$ in $\mathfrak g_\theta^{}$
is the pullback of~$\mathfrak g^{\rm ess}_{\theta_0}$ by~$\Phi_\chi$ for some particular solution~$\boldsymbol\chi$ of~$\bar L_\theta$.
This implies both the existence and the non-uniqueness of such subalgebras.
Therefore, the algebra~$\mathfrak g_\theta^{}$ admits the representation
$\mathfrak g_\theta^{}=\mathfrak g^{\rm ess}_\theta\lsemioplus\mathfrak g^{\rm lin}_\theta$.
Although the subalgebra $\mathfrak g^{\rm ess}_\theta$ of~$\mathfrak g_\theta^{}$ is not defined in a unique way,
it can be still called an \emph{essential Lie invariance algebra} of the system~$\bar L_\theta$.
Since the class~$\bar{\mathcal L}_1$ is normalized, and $G^\sim_{\bar{\mathcal L}_1}=G^\sim_{\bar{\mathcal L}}$,
the inclusion $\mathfrak g_\theta^{}\subset\varpi_*\mathfrak g^\sim_{\bar{\mathcal L}}$
holds for any $\bar L_\theta\in\bar{\mathcal L}_1$,
and thus $\mathfrak g_\theta^{}\cap\varpi_*\mathfrak g^\sim_{\bar{\mathcal L}}=\mathfrak g_\theta^{}\ne\mathfrak g^{\rm ess}_\theta$.
\end{remark}

Analogously to the notions of regular and singular Lie-symmetry extensions presented in \cite[Definition~4]{vane2020b},
we introduce the notions of regular and singular essential Lie-symmetry extensions
for classes of linear systems of differential equations.

\begin{definition}\label{def:RegularAndSingularEssLieSymExtensions}
Given a class~$\mathcal K$ of linear systems of differential equations
parameterized by an arbitrary-element tuple~$\theta$ and given a system~$K_\theta$ from this class
whose essential Lie invariance algebra~$\mathfrak g^{\rm ess}_\theta$ is well defined,
we call the algebra~$\mathfrak g^{\rm ess}_\theta$ \emph{regular} in~$\mathcal K$
if $\mathfrak g^{\rm ess}_\theta\subseteq\varpi_*\mathfrak g^\sim_{\mathcal K}$,
and \emph{singular} in~$\mathcal K$ otherwise.
\end{definition}

As discussed above, for every system from any of the classes~$\bar{\mathcal L}_1$, $\mathcal L_1$, $\mathcal L'_1$ and $\mathcal L''_1$,
its essential Lie invariance algebra is well defined.
Moreover, since the class~$\bar{\mathcal L}_1$ is normalized in the usual sense,
and the classes~$\mathcal L_1$, $\mathcal L'_1$ and $\mathcal L''_1$ are uniformly semi-normalized with respect to linear superposition of solutions,
all essential Lie-symmetry extensions in these classes are regular.

\section{Description of essential Lie-symmetry extensions}\label{sec:DescriptionOfEssLieSymExtensions}

The claims presented in the end of the previous Section~\ref{sec:SLODEPreliminaryAnalysisOfLieSyms}
definitely hold for systems from the class~$\mathcal L'_1$ and its subclass~$\mathcal L''_1$
if we replace $\vartheta$ by~$V$.
In particular, the classifying condition~\eqref{eq:ClassifyingCondL'_V} implies that
the essential Lie invariance algebra~$\mathfrak g^{\rm ess}_V$ of a system~$L'_V\in\mathcal L'_1$
necessarily contains the vector field $I=x^a\p_{x^a}$,
which corresponds to the parameter values $\tau=0$, $\Gamma=E$ and $\boldsymbol\chi=\boldsymbol0$,
and $\mathfrak g^{\rm ess}_V=\langle I\rangle$ for general values of~$V$.
In the present section, we study Lie-symmetry extensions within the class~$\mathcal L'_1$.

Consider a system~$L'_V$ from this class.
Denote by~$\pi$ the natural projection from $\mathbb F_t\times\mathbb F^n_{\boldsymbol x}$ onto $\mathbb F_t$,
and~let
\[
k=k_V:=\dim\pi_*\mathfrak g^{\rm ess}_V=\dim\pi_*\mathfrak g_V^{}.
\]
The actions of the group~$G^\sim_{\mathcal L'}$ on the spaces $\mathbb F_t\times\mathbb F^n_{\boldsymbol x}$ and~$\mathbb F_t$
are well defined via $\varpi_*G^\sim_{\mathcal L'}$ and $\pi_*(\varpi_*G^\sim_{\mathcal L'})$.
According to Theorem~\ref{thm:EquivGroupoidL'} and Corollary~\ref{cor:MAIofL'_V},
the map~$\pi$ is equivariant under the action of~$G^\sim_{\mathcal L'}$,
and the pushforward~$\pi_*$ by $\pi$ is well defined for all vector fields
from~$\mathfrak g^{\rm ess}_V$ and from~$\mathfrak g_V^{}$ for any~$V$.
Hence the value~$k$ is $G^\sim_{\mathcal L'}$-invariant.
Corollary~\ref{cor:MAIofL''_V} implies that $k\leqslant3$ for any system from the subclass~$\mathcal L''_1$
and thus for any system from the entire class~$\mathcal L'_1$.
We separately study each of the values of~$k$, which are 0, 1, 2 and 3.
Lie's theorem on realization of finite-dimensional Lie algebras by vector fields on the real or complex line
\cite[Satz~6, p.~455]{lie1880a} (see also \cite[Theorem~2.70]{olve1995A}),
Corollary~\ref{cor:MAIofL'_V} and item~(i) of Theorem~\ref{thm:EquivGroupoidL'} jointly imply
that modulo the $G^\sim_{\mathcal L_1}$-equivalence, we have
$\pi_*\mathfrak g^{\rm ess}_V=\{0\}$,
$\pi_*\mathfrak g^{\rm ess}_V=\langle\p_t\rangle$,
$\pi_*\mathfrak g^{\rm ess}_V=\langle\p_t,t\p_t\rangle$ and
$\pi_*\mathfrak g^{\rm ess}_V=\langle\p_t,t\p_t,t^2\p_t\rangle$
in the cases $k=0$, $k=1$, $k=2$ and~$k=3$, respectively. Note that the comprehensive group classification of the class~$\mathcal L'_1$ for a fixed value of~$n$
needs the classification of subalgebras of~$\mathfrak{sl}(n,\mathbb F)$ up to the ${\rm SL}(n,\mathbb F)$-equivalence.
A complete list of ${\rm SL}(n,\mathbb F)$-inequivalent subalgebras of~$\mathfrak{sl}(n,\mathbb F)$
is well known for $n=2$, see, e.g.,~\cite{pate1977a,popo2003a},
and for $n=3$ it was elegantly constructed by Winternitz~\cite{wint2004a}
and essentially revisited in~\cite{chap2024a}.
At the same time, as far as we know, there are no such lists for greater~$n$ in the literature.

Therefore, the algebra $\mathfrak g^{\rm ess}_V$ is spanned by
\begin{itemize}\itemsep=0.ex
\item
the basis vector field $I=x^a\p_{x^a}$ of the kernel Lie invariance algebra~$\mathfrak g^\cap_{\mathcal L'}$,
\item
$p$ vector fields $Q_s=\Gamma_s^{ab}x^b\p_{x^a}$, $s=1,\dots,p$,
where the matrices~$\Gamma_s=(\Gamma_s^{ab})$
constitute a basis of a subalgebra~$\mathfrak s=\mathfrak s_V$ of~$\mathfrak{sl}(n,\mathbb F)$,
$0\leqslant p:=\dim\mathfrak s\leqslant\dim\mathfrak{sl}(n,\mathbb F)=n^2-1$, and
\item
$k$ vector fields $Q_{p+\iota}\!=\tau^\iota\p_t+\big(\tfrac12\tau^\iota_tx^a+\Gamma^{ab}_{p+\iota}x^b\big)\p_{x^a}$
with linearly independent $t$-components~$\tau^\iota$, $\iota=1,\dots,k$, where $k\in\{0,1,2,3\}$.
\end{itemize}
In other words, the algebra~$\mathfrak g^{\rm ess}_V$ can be represented as
$\mathfrak g^{\rm ess}_V=\mathfrak i\oplus(\mathfrak t\lsemioplus\mathfrak s^{\rm vf})$,
where $\mathfrak i:=\langle I\rangle$ is an ideal of~$\mathfrak g^{\rm ess}_V$, which is common for all $L'_V\in\mathcal L'_1$,
$\mathfrak s^{\rm vf}=\mathfrak s^{\rm vf}_V:=\langle Q_s,\,s=1,\dots,p\rangle
=\{\Gamma^{ab}x^b\p_{x^a}\mid\Gamma\in\mathfrak s\}$
is an ideal of~$\mathfrak g^{\rm ess}_V$,
and $\mathfrak t=\mathfrak t_V:=\langle Q_{p+\iota},\,\iota=1,\dots,k\rangle$ is a subspace of~$\mathfrak g^{\rm ess}_V$.
A more precise description of the structure of~$\mathfrak g^{\rm ess}_V$ is given in Theorem~\ref{thm:StructureOfgessInL'1}.
In particular, we can always choose $\mathfrak t$ to be a subalgebra of~$\mathfrak g^{\rm ess}_V$.
Moreover, in fact $k=\dim\mathfrak t\in\{0,1,2\}$ and $p=\dim\mathfrak s^{\rm vf}\leqslant n^2-2n+1$.

\medskip\par\noindent{$\boldsymbol{k=0.}$}
It is convenient (and possible due to the \smash{$G^\sim_{\mathcal L'_1}$}-equivalence)
to assume from the very beginning that $\mathop{\rm tr}V=0$, i.e., $L'_V\in\mathcal L''_1$.
The corresponding extension of~$\mathfrak g^{\rm ess}_V$
is generated by the vector fields $Q_s:=\Gamma_s^{ab}x^b\p_{x^a}$, $s=1,\dots,p$.
It suffices to consider only ${\rm SL}(n,\mathbb F)$-inequivalent subalgebras of~$\mathfrak{sl}(n,\mathbb F)$
as candidates for using in the course of constructing Lie-symmetry extensions with $k=0$
up to the \smash{$G^\sim_{\mathcal L'_1}$}-equivalence.
It follows from the classifying condition~\eqref{eq:ClassifyingCondL'_V} that
$Q_s\in\mathfrak g^{\rm ess}_V$ if and only if $[\Gamma_s,V]=0$, $s=1,\dots,p$,
i.e., the matrix-valued function~$V$ is a linear combination of matrices
from the centralizer ${\rm C}_{\mathfrak{sl}(n,\mathbb F)}(\mathfrak s)$
of the subalgebra~$\mathfrak s$ in~$\mathfrak{sl}(n,\mathbb F)$ with coefficients depending on~$t$.
Therefore, $\mathfrak s$ is a proper subalgebra of~$\mathfrak{sl}(n,\mathbb F)$
since the subalgebra $\{0\}$ corresponds to the case with no Lie-symmetry extension,
and $\{V(t)\mid t\in\mathcal I\}\subseteq\langle E\rangle$  if $\mathfrak s=\mathfrak{sl}(n,\mathbb F)$.

In the further consideration, we use the following properties of centralizers
of subalgebras of Lie algebras.

1. Given a Lie algebra~$\mathfrak g$, for any subset~$\mathsf S$ of~$\mathfrak g$
we have $\mathsf S\subseteq{\rm C}_{\mathfrak g}\big({\rm C}_{\mathfrak g}(\mathsf S)\big)$ in view of the definition of centralizer.
If $\mathsf S_1\subseteq \mathsf S_2\subseteq\mathfrak g$, then ${\rm C}_{\mathfrak g}(\mathsf S_1)\supseteq{\rm C}_{\mathfrak g}(\mathsf S_2)$.
This is why a subalgebra~$\mathfrak s$ of~$\mathfrak g$ coincides with ${\rm C}_{\mathfrak g}(\mathsf S)$ for some subset~$\mathsf S$ of~$\mathfrak g$
if and only if $\mathfrak s={\rm C}_{\mathfrak g}\big({\rm C}_{\mathfrak g}(\mathfrak s)\big)$.
Indeed, the sufficiency in this claim directly follows from taking $\mathsf S={\rm C}_{\mathfrak g}(\mathfrak s)$.
If $\mathfrak s={\rm C}_{\mathfrak g}(\mathsf S)$, then the inclusions
$\mathfrak s\subseteq{\rm C}_{\mathfrak g}\big({\rm C}_{\mathfrak g}(\mathfrak s)\big)$,
$\mathsf S\subseteq{\rm C}_{\mathfrak g}\big({\rm C}_{\mathfrak g}(\mathsf S)\big)={\rm C}_{\mathfrak g}(\mathfrak s)$
and, therefore, $\mathfrak s={\rm C}_{\mathfrak g}(\mathsf S)\supseteq{\rm C}_{\mathfrak g}\big({\rm C}_{\mathfrak g}(\mathfrak s)\big)$
imply $\mathfrak s={\rm C}_{\mathfrak g}\big({\rm C}_{\mathfrak g}(\mathfrak s)\big)$,
which leads to the necessity part of the above claim.

2. For any subalgebra~$\mathfrak c$ of~$\mathfrak g$ that is a centralizer of a subset of~$\mathfrak g$,
the set~$\Sigma$ of all subsets of~$\mathfrak g$ with the centralizer~$\mathfrak c$
has the maximal element~$\mathfrak s$ with respect to inclusion, which coincides with~${\rm C}_{\mathfrak g}(\mathfrak c)$.
Indeed, if $\mathfrak s={\rm C}_{\mathfrak g}(\mathfrak c)$, then ${\rm C}_{\mathfrak g}(\mathfrak s)=\mathfrak c$,
i.e., $\mathfrak s\in\Sigma$.
Moreover, for any $\mathsf S\subseteq\mathfrak g$ with ${\rm C}_{\mathfrak g}(\mathsf S)=\mathfrak c$,
we have $\mathsf S\subseteq{\rm C}_{\mathfrak g}\big({\rm C}_{\mathfrak g}(\mathsf S)\big)={\rm C}_{\mathfrak g}(\mathfrak c)=\mathfrak s$,
i.e., the subalgebra~$\mathfrak s$ contains any subset from~$\Sigma$.
In other words, an element~$\mathfrak s$ from~$\Sigma$ is maximal if and only if
$\mathfrak s={\rm C}_{\mathfrak g}\big({\rm C}_{\mathfrak g}(\mathfrak s)\big)$.

3. Given a Lie group~$G$ with Lie algebra~$\mathfrak g$,
if subalgebras~$\mathfrak s_1$ and~$\mathfrak s_2$ of~$\mathfrak g$ are $G$-equivalent,
then their centralizers~${\rm C}_{\mathfrak g}(\mathfrak s_1)$ and~${\rm C}_{\mathfrak g}(\mathfrak s_2)$ are $G$-equivalent
to each other as well.
Moreover, for any subalgebra~$\mathfrak s$ of~$\mathfrak g$,
the stabilizer subgroup ${\rm St}_G(\mathfrak s)$ of $G$ with respect to~$\mathfrak s$
is contained in the stabilizer subgroup ${\rm St}_G\big({\rm C}_{\mathfrak g}(\mathfrak s)\big)$
of $G$ with respect to~${\rm C}_{\mathfrak g}(\mathfrak s)$.
It is obvious that ${\rm St}_G(\mathfrak s)={\rm St}_G\big({\rm C}_{\mathfrak g}(\mathfrak s)\big)$
if $\mathfrak s={\rm C}_{\mathfrak g}\big({\rm C}_{\mathfrak g}(\mathfrak s)\big)$.

\medskip\par\noindent{\it Classification procedure.}
Therefore, we can suggest the following algorithm for classifying Lie-symmetry extensions with $k=0$,
at least for low values of~$n$.
\begin{itemize}\itemsep=0ex
\item
For each subalgebra from a complete list of ${\rm SL}(n,\mathbb F)$-inequivalent subalgebras of~$\mathfrak{sl}(n,\mathbb F)$,
compute the centralizer of its centralizer.
\item
Select all the proper subalgebras $\mathfrak s$ from the list for which
$\mathfrak s={\rm C}_{\mathfrak{sl}(n,\mathbb F)}\big({\rm C}_{\mathfrak{sl}(n,\mathbb F)}(\mathfrak s)\big)$.
\item
A complete list of inequivalent essential Lie-symmetry extensions with $k=0$ in the class~$\mathcal L'_1$ is constituted by
the algebras $\mathfrak g^{\rm ess}_V=\langle I\rangle\oplus\mathfrak s^{\rm vf}$
for all the obtained subalgebras $\mathfrak s\varsubsetneq\mathfrak{sl}(n,\mathbb F)$,
where the arbitrary-element matrix~$V$ takes general values of the form $V=v^l(t)K_l$
for a~basis $(K_l,\,l=1,\dots,m)$ of \smash{${\rm C}_{\mathfrak{sl}(n,\mathbb F)}(\mathfrak s)$}.
\end{itemize}
A necessary condition for a particular value of~$V$ in the above form to be general
is that after rewriting it as $V=\tilde v^{l'}(t)\tilde K_{l'}$
with linearly independent functions~$\tilde v^{l'}$ and linearly independent matrices~$\tilde K_{l'}$, $l'=1,\dots,m'\leqslant m$,
we have
\[
{\rm C}_{\mathfrak{sl}(n,\mathbb F)}\big(\{V(t)\mid t\in\mathcal I\}\big)=
{\rm C}_{\mathfrak{sl}(n,\mathbb F)}\big(\{\tilde K_1,\dots,\tilde K_{m'}\}\big)=\mathfrak s.
\]
This definitely includes the values of~$V$ with linearly independent~$v^1$, \dots, $v^m$.
Sufficient conditions for general values of~$V$ in these classification cases
additionally contain more delicate constraints, which are related to the imposed condition $k=0$,
cf.\ the cases with greater values of~$k$ below.
It is obvious that in the representation $V=\tilde v^{l'}(t)\tilde K_{l'}$,
the coefficients~$\tilde v^1$,~\dots, $\tilde v^{m'}$ and the matrices $\tilde K_1$,~\dots, $\tilde K_{m'}$
are defined up to the transformations of the form
$\tilde v^{l'}\to\tilde v^{l''}\alpha^{l'l''}$, $\tilde K_{l'}\to\hat\alpha^{l'l''}\tilde K_{l''}$
with invertible $m'\times m'$ matrix~$(\alpha^{l'l''})$ and $(\hat\alpha^{l'l''})=(\alpha^{l'l''})^{-1}$.
The systems~\smash{$L'_V,L'_{\tilde V}\in\mathcal L''_1$}
with the same essential Lie invariance algebra $\langle I\rangle\oplus\mathfrak s^{\rm vf}$
are \smash{$G^\sim_{\mathcal L'_1}$}-equivalent if and only if there exists a transformation
from the (significant) equivalence group~$G^{{\rm s\vphantom{g}}\sim}_{\mathcal L''}$ of the class~$\mathcal L''$,
with the matrix~$C$ from the stabilizer subgroup \smash{${\rm St}_{{\rm SL}(n,\mathbb F)}(\mathfrak s)$}
of ${\rm SL}(n,\mathbb F)$ with respect to~$\mathfrak s$
that maps~$V$ to~$\tilde V$, cf.\ Corollary~\ref{cor:EquivGroupoidL''}.

In the above algorithm, we can permute the roles of subalgebras spanned by the matrices~$\Gamma_s$
and of subalgebras spanned by the matrices~$K_l$,
selecting those among subalgebras of $\mathfrak{sl}(n,\mathbb F)$ associated with the representations for~$V$
that coincide with the centralizers of their centralizers.
Then elements~$\Gamma$ of their centralizers correspond to vector fields giving Lie-symmetry extensions.

\medskip\par\noindent{$\boldsymbol{k=1.}$}
Up to the \smash{$G^\sim_{\mathcal L'_1}$}-equivalence, the algebra~$\mathfrak g^{\rm ess}_V$
contains a vector field $P:=\p_t+\Upsilon^{ab}x^b\p_{x^a}$ with some (constant) matrix~$\Upsilon\in\mathfrak{sl}(n,\mathbb F)$.
The classifying condition~\eqref{eq:ClassifyingCondL'_V} with substituted components of this vector field implies
that $V(t)={\rm e}^{t\Upsilon}V(0){\rm e}^{-t\Upsilon}$.
Hence the system~$L'_V$ is reduced by the point transformation
$\tilde t=t$, $\tilde{\boldsymbol x}={\rm e}^{-t\Upsilon}\boldsymbol x$
to the system $\tilde{\boldsymbol x}_{\tilde t\tilde t}=-2\Upsilon\tilde{\boldsymbol x}_{\tilde t}+\big(V(0)-\Upsilon^2\big)\tilde{\boldsymbol x}$
whose matrix-valued coefficients are constant.
Hence the system~$L'_V$ is easily integrable.

We can represent the matrix~$V(0)$ in the form $V(0)=W+\varepsilon E$,
where $W$ is a nonzero (constant) traceless matrix, $\varepsilon=n^{-1}\mathop{\rm tr}V(0)$,
and set $\varepsilon\in\mathscr E$ up to equivalence transformations that scale~$t$,
where $\mathscr E:=\{-1,0,1\}$ if $\mathbb F=\mathbb R$ and $\mathscr E:=\{0,1\}$ if $\mathbb F=\mathbb C$,
i.e., according to the Baker--Hausdorff formula,
\begin{gather}\label{eq:k=1RepresentationForV}
V(t)=\varepsilon E+{\rm e}^{t\Upsilon}W{\rm e}^{-t\Upsilon}
=\varepsilon E+\sum_{l=0}^\infty\frac{t^l}{l!}K_l,
\quad
K_0:=W,\ K_l:=[\Upsilon,K_{l-1}],\ l=1,2,\dots\,.
\end{gather}
Note that $W\ne0$ since $L'_V\notin\mathcal L'_0$.
It follows from the classifying condition~\eqref{eq:ClassifyingCondL'_V}
that a vector field $Q_\Gamma=\Gamma^{ab}x^b\p_{x^a}$ with $\Gamma\in\mathfrak{sl}(n,\mathbb F)$ belongs to~$\mathfrak g^{\rm ess}_V$
if and only if $[V(t),\Gamma]=0$ for any~$t$ from the domain of~$V$,
which is equivalent to the condition
\begin{gather}\label{eq:CentralizerCondForKl}
[K_l,\Gamma]=0,\quad l\in\mathbb N_0=\mathbb N\cup\{0\}.
\end{gather}
The set~$\mathfrak s$ of all such matrices~$\Gamma$ is the centralizer of $\{K_l,\,l\in\mathbb N_0\}$,
$\mathfrak s={\rm C}_{\mathfrak{sl}(n,\mathbb F)}(\{K_l,\,l\in\mathbb N_0\})$,
and thus it is a subalgebra of~$\mathfrak{sl}(n,\mathbb F)$.
In fact,
$\mathfrak s={\rm C}_{\mathfrak{sl}(n,\mathbb F)}(\{K_0,\dots,K_m\})$,
where $m$ the maximum value of~$l$ such that $K_0$, \dots, $K_l$ are linearly independent.
(Since $K_{m+1}\in\langle K_0,\dots,K_m\rangle$,
we can show by induction that $K_l\in\langle K_0,\dots,K_m\rangle$ for any $l>m$.)
Therefore, the subalgebra~$\mathfrak s$ can be found constructively if the matrices~$\Upsilon$ and~$W$ are given,
and $\mathfrak s\ne\mathfrak{sl}(n,\mathbb F)$.
In view of the Jacobi identity,
the condition~\eqref{eq:CentralizerCondForKl} holds true if and only if
\[
[\Upsilon,\Gamma]\in\mathfrak s,\quad [W,\Gamma]=0,
\]
i.e., $\Upsilon\in{\rm N}_{\mathfrak{sl}(n,\mathbb F)}(\mathfrak s)$
and $W\in{\rm C}_{\mathfrak{sl}(n,\mathbb F)}(\mathfrak s)$.
The condition $[\Upsilon,\Gamma]\in\mathfrak s$ is natural
since for $\Gamma\in\mathfrak s$ we have $Q_\Gamma:=\Gamma^{ab}x^b\p_{x^a}\in\mathfrak g^{\rm ess}_V$ and thus
$[P,Q_\Gamma]=[\Upsilon,\Gamma]^{ab}x^b\p_{x^a}\in\mathfrak g^{\rm ess}_V$, which means that $[\Upsilon,\Gamma]\in\mathfrak s$.
As in the case $k=0$, denote $\mathfrak s^{\rm vf}:=\{Q_\Gamma\mid\Gamma\in\mathfrak s\}$.
Note that varying the parameter~$\varepsilon$ has no influence
on the algebra~$\mathfrak s^{\rm vf}$ and the vector field~$P$, and thus on the entire algebra~$\mathfrak g^{\rm ess}_V$ as well.

Since $\varepsilon=n^{-1}\mathop{\rm tr}V(0)$ and $W=V(0)-\varepsilon E$,
both $\varepsilon$ and~$W$ are defined by~$V$ in a unique way, but this is not the case for~$\Upsilon$.
More specifically, the tuples $(\varepsilon,W,\Upsilon)$ and $(\tilde\varepsilon,\tilde W,\tilde\Upsilon)$
are associated with the same values $V=\tilde V$ if and only if
$\varepsilon=\tilde\varepsilon$, $W=\tilde W$ and $\Upsilon-\tilde\Upsilon\in\mathfrak s$.
This claim follows from the second representation for~$V$ in~\eqref{eq:k=1RepresentationForV}.
Indeed, the sufficiency is obvious since then
$\tilde K_0=K_0$ and by induction $\tilde K_l=[\tilde\Upsilon,\tilde K_{l-1}]=[\Upsilon,K_{l-1}]=K_l$, $l\in\mathbb N$,
which implies $V=\tilde V$ in view of~\eqref{eq:k=1RepresentationForV}.
Conversely, if $V=\tilde V$, then $\varepsilon=\tilde\varepsilon$, $W=\tilde W$,
$K_l-\tilde K_l=[\Upsilon-\tilde\Upsilon,K_{l-1}]=0$, $l\in\mathbb N$, and hence $\Upsilon-\tilde\Upsilon\in\mathfrak s$.
The proved claim is natural within the Lie-symmetry interpretation of the matrix~$\Upsilon$ and matrices from $\mathfrak s$
since, parameterizing the Lie-symmetry vector field~$P$ of the system~$L'_V$,
the matrix~$\Upsilon$ is defined up to adding an arbitrary matrix from~$\mathfrak s$
due to the possibility of linearly combining~$P$ with elements of~$\mathfrak s^{\rm vf}$.

\medskip\par\noindent{\it Classification procedure.}
Analogously to the case $k=0$ and again at least for low values of~$n$,
there are two ways for classifying Lie-symmetry extensions with $k=1$
but, in contrast to the former case, these ways are not conjugate to each other.

The way that starts with the classification of possible Lie-symmetry algebras
is quite similar to its counterpart in the case $k=0$
and is also based on properties of centralizers.
We suppose that the first two steps of the algorithm for $k=0$ have been realized,
i.e., a complete list of ${\rm SL}(n,\mathbb F)$-inequivalent subalgebras~$\mathfrak s$ of~$\mathfrak{sl}(n,\mathbb F)$
satisfying the condition
$\mathfrak s={\rm C}_{\mathfrak{sl}(n,\mathbb F)}\big({\rm C}_{\mathfrak{sl}(n,\mathbb F)}(\mathfrak s)\big)\ne\mathfrak{sl}(n,\mathbb F)$
has been constructed.
In contrast to the case $k=0$, the subalgebra~$\mathfrak s=\{0\}$ may be appropriate here.
Then a complete list of inequivalent essential Lie-symmetry extensions with $k=1$ in the class~$\mathcal L'_1$
consists of the algebras $\mathfrak g^{\rm ess}_V=\langle I\rangle\oplus\big(\langle P\rangle\lsemioplus\mathfrak s^{\rm vf}\big)$
for all the obtained subalgebras $\mathfrak s\varsubsetneq\mathfrak{sl}(n,\mathbb F)$,
where the arbitrary-element matrix~$V$ takes the general values of the form~\eqref{eq:k=1RepresentationForV}
with $\varepsilon\in\mathscr E$, nonzero $W\in\mathfrak c:={\rm C}_{\mathfrak{sl}(n,\mathbb F)}(\mathfrak s)$
and $\Upsilon$ from a (fixed) complementary subspace~$\mathsf s$ to~$\mathfrak s$ in ${\rm N}_{\mathfrak{sl}(n,\mathbb F)}(\mathfrak s)$.
For arbitrary matrices~$W\in\mathfrak c$ and~$\Upsilon\in\mathsf s$,
we have $K_l\in\mathfrak c$ for all the matrices $K_l$ defined in~\eqref{eq:k=1RepresentationForV}
or, equivalently, $\mathfrak s\subseteq{\rm C}_{\mathfrak{sl}(n,\mathbb F)}(\{K_l,\,l\in\mathbb N_0\})$.
The essential Lie-symmetry extension case associated with a subalgebra~$\mathfrak s$ is genuine
if and only if there exist such matrices~$W\in\mathfrak c$ and~$\Upsilon\in\mathsf s$
that $\mathfrak s={\rm C}_{\mathfrak{sl}(n,\mathbb F)}(\{K_l,\,l\in\mathbb N_0\})$ and $\dim\pi_*\mathfrak g^{\rm ess}_V<2$.
Similarly to the case $k=0$, we call such values of the parameter-matrix pair $(\Upsilon,W)$ general.
In other words, the algebra $\langle I\rangle\oplus\big(\langle P\rangle\lsemioplus\mathfrak s^{\rm vf}\big)$ indeed coincides with
the entire algebra $\mathfrak g^{\rm ess}_V$ for $V=\varepsilon E+{\rm e}^{t\Upsilon}W{\rm e}^{-t\Upsilon}$
if and only if  the pair $(\Upsilon,W)$ takes a general value in $\mathsf s\times\mathfrak c$.
In addition to~$\mathfrak s$,
the adjoint actions of matrices from the stabilizer subgroup \smash{${\rm St}_{{\rm SL}(n,\mathbb F)}(\mathfrak s)$}
of ${\rm SL}(n,\mathbb F)$ with respect to~$\mathfrak s$ preserve $\mathfrak c$ and ${\rm N}_{\mathfrak{sl}(n,\mathbb F)}(\mathfrak s)$.
This is why they induce well-defined actions on the quotient space ${\rm N}_{\mathfrak{sl}(n,\mathbb F)}(\mathfrak s)/\mathfrak s$
and, therefore, on $\mathsf s$ as a set of representatives for cosets from this quotient space.
Systems~$L'_V$ and~\smash{$L'_{\tilde V}$} with $\varepsilon,\tilde\varepsilon\in\mathscr E$ and general values
$(\Upsilon,W),(\tilde\Upsilon,\tilde W)\in\mathsf s\times\mathfrak c$
are $G^\sim_{\mathcal L'}$-equivalent if and only if $\varepsilon=\tilde\varepsilon$ and
there exists a matrix from \smash{${\rm St}_{{\rm SL}(n,\mathbb F)}(\mathfrak s)$}
whose adjoint action maps $(\Upsilon,W)$ to $(\tilde\Upsilon,\tilde W)$ if $\varepsilon\ne0$
or to $(\alpha\tilde\Upsilon,\alpha^2\tilde W)$ with $\alpha\in\mathbb F\setminus\{0\}$ otherwise.

The initial point of the second way is the form of the matrix-valued parameter function~$V$,
which is defined for $k=1$ by~\eqref{eq:k=1RepresentationForV}
and values of $\varepsilon\in\mathscr E$ and of the constant traceless parameter matrices~$W$ and~$\Upsilon$.
We take a complete list of ${\rm SL}(n,\mathbb F)$-inequivalent matrix pairs $(\Upsilon,W)$.
Such lists were (and can be) constructed only for low values of~$n$,
see, e.g., \cite{serg1993a} for $n=4$.
For a general value of~$n$, the problem of classifying pairs of $n\times n$ matrices up to matrix similarity
is wild, see \cite{beli2000a,beli2003a,frie1983a} and references therein
for discussions and results on this problem.
Then we compute the maximum number of linearly independent matrices~$K_0$,~\dots, $K_m$ and~$\mathfrak s$
according to~\eqref{eq:k=1RepresentationForV} and
\smash{$\mathfrak s:={\rm C}_{\mathfrak{sl}(n,\mathbb F)}(\{K_0,\dots,K_m\})$}.
Pairs $(\Upsilon,W)$ and $(\tilde\Upsilon,\tilde W)$
with $\tilde W={\rm e}^{t_0\Upsilon}W{\rm e}^{-t_0\Upsilon}$ for some $t_0\in\mathbb F$ and $\tilde\Upsilon-\Upsilon\in\mathfrak s$
define the same value of~$V$ up to the shifts of~$t$,
which gives an additional equivalence relation on matrix pairs.
We reduce the chosen complete list of ${\rm SL}(n,\mathbb F)$-inequivalent matrix pairs
using this equivalence relation.
If $\mathfrak g^{\rm ess}_V=\langle I\rangle\oplus\big(\langle P\rangle\lsemioplus\mathfrak s^{\rm vf}\big)$,
then the pair $(\Upsilon,W)$ gives a proper case of Lie-symmetry extension with $k=1$.

\begin{remark}\label{rem:SLODEConvenientClassesForGroupClassification}
In contrast to the case $k=0$,
the class~$\mathcal L'_1$ is a better choice for group classification in the case $k=1$
than its subclass~$\mathcal L''_1$.
This is justified by the fact that the single canonical (up to the $G^\sim_{\mathcal L'}$-equivalence) value $\tau=1$
for Lie-symmetry vector fields with nonzero $t$-components~$\tau$ within the class~$\mathcal L'_1$
corresponds to three or two different $G^\sim_{\mathcal L''}$-inequivalent values,
$\tau\in\{1,t,t^2+1\}$ or $\tau\in\{1,t\}$ if $\mathbb F=\mathbb R$ or $\mathbb F=\mathbb C$, respectively,
when the class~$\mathcal L''_1$ is considered.
See Corollaries~\ref{cor:EquivGroupoidL''} and~\ref{cor:MAIofL''_V}
and the well-known classifications of subalgebras of $\mathfrak{sl}(2,\mathbb F)$
for $\mathbb F=\mathbb R$~\cite{pate1977a} and $\mathbb F=\mathbb C$.
Another explanation is that the parameter~$\varepsilon$ does not have influence
on solving the group classification problem for the class~$\mathcal L'_1$
but making the matrix-valued function~$V$ traceless when $\varepsilon\ne0$,
we (after rescaling $\varepsilon$ to $1/4$ or, only if $\varepsilon<0$ in the real case, to $-1/2$ for convenience)
perform the point transformation $\tilde t=T(t)$, $\tilde{\boldsymbol x}=\boldsymbol x$,
where $T(t)={\rm e}^{t}$ if $\varepsilon=1/4$ and $T(t)=\tan t$ if $\varepsilon=-1/2$.
As a result, instead of the single pair
\[\big({\rm e}^{t\Upsilon}V(0){\rm e}^{-t\Upsilon},\,\p_t+\Upsilon^{ab}x^b\p_{x^a}\big)\]
of expressions for $(V,P)$ up to the $G^\sim_{\mathcal L'}$-equivalence in the class~$\mathcal L'_1$,
we obtain two or one more, more complicated, pairs over~$\mathbb R$ or~$\mathbb C$,
\begin{gather*}
\big({\rm e}^{\ln\!|t|\,\Upsilon}V(1){\rm e}^{-\ln\!|t|\,\Upsilon},\,t\p_t+\Upsilon^{ab}x^b\p_{x^a}\big),
\\
\big({\rm e}^{\arctan(t)\Upsilon}V(0){\rm e}^{-\arctan(t)\Upsilon},\,(t^2+1)\p_t+\Upsilon^{ab}x^b\p_{x^a}\big)
\end{gather*}
or the first of these pairs with $\ln|t|$ replaced by $\mathop{\rm Log}t$, respectively,
up to the $G^\sim_{\mathcal L''}$-equivalence in the class~$\mathcal L''_1$.
Here $\mathop{\rm Log}$ denotes the principal value of logarithm.
The form of systems admitting Lie-symmetry vector fields with nonzero $t$-components
has a proper and clear interpretation only after specifically mapping these systems to the superclass~$\mathcal L_1$.
At the same time, this class is not too convenient for directly carrying out group classification
since the parameter functions~$\eta^{ab}$ in Lie-symmetry vector fields of systems from~$\mathcal L_1$
can take nonconstant values, see Lemma~\ref{lem:MAIofLvartheta},
which needs additional efforts in the course of classifying such systems.
\end{remark}

\noindent{$\boldsymbol{k=2.}$}
The algebra~$\mathfrak g^{\rm ess}_V$ contains two vector fields with linearly independent $t$-components,
which can be assumed to take, modulo their linear recombination and the \smash{$G^\sim_{\mathcal L'_1}$}-equivalence,
in the form $P:=\p_t+\Upsilon^{ab}x^b\p_{x^a}$ and \mbox{$D:=t\p_t+\Lambda^{ab}x^b\p_{x^a}$}
with some (constant) matrices~$\Upsilon,\Lambda\in\mathfrak{sl}(n,\mathbb F)$.
Analogously to the case $k=1$, the classifying condition~\eqref{eq:ClassifyingCondL'_V} with substituted components of~$P$ implies
the general form of~$V$, $V(t)={\rm e}^{t\Upsilon}W{\rm e}^{-t\Upsilon}$, where we denote $W:=V(0)$.
The condition $L'_V\notin\mathcal L'_0$ is equivalent to $W\notin\langle E\rangle$.
Now we look for restrictions on~$V$ that arise due to the invariance of~$L'_V$ with respect to~$D$.
The substitution of the components of this vector field into the classifying condition~\eqref{eq:ClassifyingCondL'_V}
gives the matrix equation $tV_t+2V=[\Lambda,V]$.
We differentiate the last equation $l\in\mathbb N_0$ times with respect to~$t$,
deriving $t\p_t^{l+1}V_t+(l+2)\p_t^lV=[\Lambda,\p_t^lV]$,
and then estimate the result at $t=0$.
This gives the series of matrix equations
\begin{gather}\label{eq:CommutatorOfLambdaKl}
[\Lambda,K_l]=(l+2)K_l,\quad l\in\mathbb N_0,
\end{gather}
with the matrices~$K_l$ defined as in~\eqref{eq:k=1RepresentationForV}
since $\p_t^lV(0)=(\mathop{\rm ad}_\Upsilon)^lW=K_l$.
Here and in what follows we consider $\mathop{\rm ad}_M:=[M,\cdot]$ for a matrix~$M$
as a linear operator on $\mathfrak{sl}(n,\mathbb F)$.
In view of~\eqref{eq:CommutatorOfLambdaKl},
Lemma~4 of \cite[p.~44]{jaco1962A} implies that
for each $l\in\mathbb N_0$ the matrix $K_l$ is nilpotent.
Moreover, if this matrix is nonzero,
then it is an eigenvector of~$\mathop{\rm ad}_\Lambda$ with eigenvalue $l+2$.
Since a linear operator on a finite-dimensional linear space has a finite number of eigenvalues,
and eigenvectors associated with different eigenvalues are linearly independent,
there are only a finite number of nonzero~$K_l$.
Recall that at least $K_0:=W$ is nonzero, and thus $\Lambda\ne0$.
Let $m$ be the minimum value of~$l$ such that $K_{l+1}=0$.
Then $K_l=0$ for all $l>m$.
In other words, the matrix-valued function~$V$ is a polynomial in~$t$ with nonzero nilpotent matrix coefficients,
\begin{gather}\label{eq:k=2RepresentationForV}
V={\rm e}^{t\Upsilon}W{\rm e}^{-t\Upsilon}
=\sum_{l=0}^m\frac{t^l}{l!}K_l,
\quad
K_0:=W\ne0,\quad K_l:=[\Upsilon,K_{l-1}]\ne0,\ l=1,2,\dots,m.
\end{gather}

The commutation relations~\eqref{eq:CommutatorOfLambdaKl} imply that
$[\Lambda,K_lK_{l'}]=(l+l'+4)K_lK_{l'}$, $l,l'\in\mathbb N_0$, and thus
\[
[\Lambda,[K_l,K_{l'}]]=(l+l'+4)[K_l,K_{l'}],\quad l,l'\in\mathbb N_0.
\]
By induction, we can prove analogous relations for the commutators of~$\Lambda$
with the multifold products and commutators of the matrices~$K_l$.
Hence the associative algebra~$\mathfrak K$ and the Lie algebra~$\mathfrak k$
that are generated by $\{K_0,\dots,K_m\}$ are graded nilpotent algebras
constituted by nilpotent matrices,
$\mathfrak K=\bigoplus_{l\geqslant2}\mathfrak K_l$
with $\mathfrak K_l\mathfrak K_{l'}\subseteq\mathfrak K_{l+l'}$,
$\mathfrak k=\bigoplus_{l\geqslant2}\mathfrak k_l$
with $[\mathfrak k_l,\mathfrak k_{l'}]\subseteq\mathfrak k_{l+l'}$ and
\mbox{$[\Lambda,M]=lM$} for any $M\in\mathfrak K_l$ as well as for any $M\in\mathfrak k_l$.
Since the Lie algebra~$\mathfrak k$ consists of nilpotent matrices,
by Engel's theorem these matrices can all be simultaneously brought to the strictly upper triangular form
up to matrix similarity.
In other words, \emph{modulo the ${\rm SL}(n,\mathbb F)$-equivalence,
the matrix-valued function~$V$ is a polynomial in~$t$ with strictly upper triangular matrix coefficients},
cf.\ the sentence with the representation~\eqref{eq:k=2RepresentationForV}.

Denoting $\mathfrak s:={\rm C}_{\mathfrak{sl}(n,\mathbb F)}(\{K_0,\dots,K_m\})$
and $\mathfrak s^{\rm vf}:=\{Q_\Gamma:=\Gamma^{ab}x^b\p_{x^a}\mid\Gamma\in\mathfrak s\}$
as in the previous cases for values of~$k$,
we have $\mathfrak g^{\rm ess}_V=\langle I\rangle\oplus\big(\langle P,D\rangle\lsemioplus\mathfrak s^{\rm vf}\big)$,
and $\mathfrak s\ne\mathfrak{sl}(n,\mathbb F)$.
In particular, there are matrix commutation relations
\[
[\Lambda,\Upsilon]-\Upsilon\in\mathfrak s,\quad
[\Upsilon,\Gamma],[\Lambda,\Gamma]\in\mathfrak s\ \mbox{for any}\ \Gamma\in\mathfrak s,
\]
and the matrices~$\Upsilon$ and $\Lambda$ are defined up independently adding any matrices from~$\mathfrak s$.
The indefiniteness of~$\Upsilon$ also follows from the representation~\eqref{eq:k=2RepresentationForV}.
Since the Lie algebra~$\mathfrak k$ is nilpotent,
its center ${\rm Z}(\mathfrak k):={\rm C}_{\mathfrak k}(\mathfrak k)$ is nonzero.
Therefore, $\mathfrak s={\rm C}_{\mathfrak{sl}(n,\mathbb F)}(\mathfrak k)\supseteq{\rm Z}(\mathfrak k)\ne\{0\}$,
i.e., $\dim\mathfrak s^{\rm vf}\geqslant1$.

Now we will first consider the case $\mathbb F=\mathbb C$ and will then show
that the real case can be reduced to the complex case.

Let $\sigma(\Lambda)=\{\mu_1,\dots,\mu_r\}$ be the spectrum of the matrix~$\Lambda$, i.e., the set of its distinct eigenvalues,
and thus $r\in\{1,\dots,n\}$.
In what follows the indices~$i$ and~$j$ run from~1 to~$r$.
Denote by~$\mathscr U_i$ the generalized eigenspace for the eigenvalue~$\mu_i$.
Then $\mathbb F^n=\bigoplus_i\mathscr U_i$,
$n_i:=\dim\mathscr U_i$ is the algebraic multiplicity of the eigenvalue~$\mu_i$, and hence $n_1+\dots+n_r=n$.
We choose a canonical basis in each~$\mathscr U_i$ for the restriction~\smash{$\Lambda\big|_{\mathscr U_i}$} of~$\Lambda$ on~$\mathscr U_i$,
which is composed entirely of Jordan chains for~\smash{$\Lambda\big|_{\mathscr U_i}$},
and sequentially concatenate the chosen bases into a basis of~$\mathbb F^n$.
We partition each matrix~$M$ into blocks according to the decomposition $\mathbb F^n=\bigoplus_i\mathscr U_i$,
$M=(M_{ij})$, where $M_{ij}$ is the $n_i\times n_j$ submatrix of~$M$ obtained by deleting all the rows and columns,
except those associated with the basis elements of~$\mathscr U_i$ and~$\mathscr U_j$, respectively.
In particular, $\Lambda=\bigoplus_i\Lambda_{ii}$.

For any $\mu\in\mathbb F$, the commutation relations~\eqref{eq:CommutatorOfLambdaKl} imply
$(\Lambda-(\mu+l+2)E)K_l=K_l(\Lambda-\mu E)$, and thus
\[
(\Lambda-(\mu+l+2)E)^{l'}K_l=K_l(\Lambda-\mu E)^{l'},\quad l,l'\in\mathbb N_0.
\]
In particular, this means that $K_l\mathscr U_j\subseteq\mathscr U_i$ for~$(i,j)$ with $\mu_i=\mu_j+l+2$,
or, equivalently, $K_{l,ij}=0$ if $\mu_i\ne\mu_j+l+2$.
Consider the Jordan--Chevalley decomposition of~$\Lambda$, $\Lambda=\Lambda_{\rm s}+\Lambda_{\rm n}$.
Thus, $\Lambda_{{\rm s},ii}=\Lambda_{ii,\rm s}=\mu_iE_{ii}$, $\Lambda_{{\rm n},ii}=\Lambda_{ii,\rm n}$,
and $\Lambda_{{\rm s},ij}=\Lambda_{{\rm n},ij}=0$ if $i\ne j$.
We split the matrix commutation relations~\eqref{eq:CommutatorOfLambdaKl} into blocks,
\[
[\Lambda,K_l]_{ij}=\Lambda_{ii}K_{l,ij}-K_{l,ij}\Lambda_{jj}=(l+2)K_{l,ij}.
\]
The last condition is a trivial equality of zero matrices for $(i,j)$ with $\mu_i\ne\mu_j+l+2$.
Otherwise, we can expand it into
\[
((\mu_j+l+2)E_{ii}+\Lambda_{ii,\rm n})K_{l,ij}-K_{l,ij}(\mu_jE_{jj}+\Lambda_{jj,\rm n})=(l+2)K_{l,ij}.
\]
Hence $[\Lambda_{\rm n},K_l]_{ij}=\Lambda_{ii,\rm n}K_{l,ij}-K_{l,ij}\Lambda_{jj,\rm n}=0$.
We obviously have the same equality for $(i,j)$ with $\mu_i\ne\mu_j+l+2$.
Therefore,  $[\Lambda_{\rm n},K_l]=0$, $l\in\mathbb N_0$, i.e., $\Lambda_{\rm n}\in\mathfrak s$.
Due to the indefiniteness of~$\Lambda$, we can set $\Lambda_{\rm n}=0$ by subtracting $\Lambda_{\rm n}$ from~$\Lambda$,
i.e., without loss of generality we can always replace~$\Lambda$ by $\Lambda_{\rm s}$, which is necessarily nonzero,
and assume that the matrix~$\Lambda$ is semisimple (or, equivalently, diagonalizable since $\mathbb F=\mathbb C$ here).

We decompose the matrix~$\Upsilon$ as $\Upsilon=\hat\Upsilon+\check\Upsilon$,
where $\hat\Upsilon_{ij}:=\Upsilon_{ij}$ if $\mu_i=\mu_j+1$, $\hat\Upsilon_{ij}:=0$ otherwise,
and $\check\Upsilon:=\Upsilon-\hat\Upsilon$.
Hence $\check\Upsilon_{ij}:=0$ if $\mu_i=\mu_j+1$ and $\check\Upsilon_{ij}:=\Upsilon_{ij}$ otherwise.
The notation in~\eqref{eq:k=2RepresentationForV} then expands to the equality
$[\Upsilon,K_l]=[\hat\Upsilon,K_l]+[\check\Upsilon,K_l]=K_{l+1}$.
We split this equality into blocks, $[\hat\Upsilon,K_l]_{ij}+[\check\Upsilon,K_l]_{ij}=K_{l+1,ij}$.

Recall that $K_{l,ij}=0$ if $\mu_i\ne\mu_j+l+2$.
This is why $[\check\Upsilon,K_l]_{ij}=\check\Upsilon_{ij'}K_{l,j'j}-K_{l,ii'}\check\Upsilon_{i'j}=0$ if $\mu_i=\mu_j+l+3$.
Indeed, $K_{l,j'j}\ne0$ only if $\mu_{j'}=\mu_j+l+2$ and hence $\mu_i=\mu_{j'}+1$, which implies $\check\Upsilon_{ij'}=0$.
Analogously, $K_{l,ii'}\ne0$ only if $\mu_i=\mu_{i'}+l+2$, giving $\mu_{i'}=\mu_j+1$ and $\check\Upsilon_{i'j}=0$.

We also have that $[\hat\Upsilon,K_l]_{ij}=\hat\Upsilon_{ij'}K_{l,j'j}-K_{l,ii'}\hat\Upsilon_{i'j}=0$ if $\mu_i\ne\mu_j+l+3$.
Indeed, the first (resp.\ second) matrix product may be nonzero only if both its factors are nonzero,
which needs $\mu_i=\mu_{j'}+1$ and $\mu_{j'}=\mu_j+l+2$ (resp.\ $\mu_i=\mu_{i'}+l+2$ and $\mu_{i'}=\mu_j+1$),
i.e., $\mu_i=\mu_j+l+3$.
Since $K_{l+1,ij}=0$ if $\mu_i\ne\mu_j+l+3$, we obtain that $[\check\Upsilon,K_l]_{ij}=K_{l+1,ij}-[\hat\Upsilon,K_l]_{ij}=0$
if $\mu_i\ne\mu_j+l+3$.

As a result, $[\check\Upsilon,K_l]=0$, $l\in\mathbb N_0$, i.e., $\check\Upsilon\in\mathfrak s$.
Due to the indefiniteness of~$\Upsilon$, we can set $\check\Upsilon=0$ by subtracting $\check\Upsilon$ from~$\Upsilon$,
i.e., without loss of generality \emph{we can always replace the matrix~$\Upsilon$ by its hat-part $\hat\Upsilon$ with respect to~$\Lambda$}.
We have $[\Lambda,\Upsilon]_{ij}=\mu_iE_{ii}\Upsilon_{ij}-\Upsilon_{ij}\mu_jE_{jj}=(\mu_i-\mu_j)\Upsilon_{ij}$.
Hence $[\Lambda,\Upsilon]_{ij}=\Upsilon_{ij}$ if $\mu_i=\mu_j+1$, and $[\Lambda,\Upsilon]_{ij}=0=\Upsilon_{ij}$ otherwise,
i.e., $[\Lambda,\Upsilon]=\Upsilon$.
In other words, replacing the matrix~$\Upsilon$ by its hat-part,
we obtain, instead of $[\Lambda,\Upsilon]-\Upsilon\in\mathfrak s$, the more restrictive commutation relation
\[[\Lambda,\Upsilon]=\Upsilon.\]

We reorder and partition the tuple $(\mu_1,\dots,\mu_r)$ into such maximal subtuples and accordingly modify the chosen basis of~$\mathbb F^n$
that within each of these subtuples, we have the descending order of the real parts of its elements,
and $\mu_i-\mu_{i+1}\in\{1,2\}$ for each pair of its successive elements.
We call these subtuples the \emph{chains of eigenvalues} of~$\Lambda$.
Then the matrices~$\Upsilon$, $W=K_0$ and all other~$K_l$ become strictly upper triangular,
i.e., we have proved in one more way that up to the $G^\sim_{\mathcal L_1}$-equivalence,
the matrix-valued function~$V$ is a polynomial in~$t$ with strictly upper triangular matrix coefficients.
The degree~$m$ of this polynomial is less than
the maximum among the lengths of the eigenvalue chains of~$\Lambda$ reduced by two.
Moreover, $\Upsilon_{ij}=K_{l,ij}=0$, $l\in\mathbb N_0$, if $\mu_i$ and~$\mu_j$ belong to different chains.

The above constraints that the matrix~$\Lambda$ is semisimple
and the matrix~$\Upsilon$ coincides with its hat-part with respect to~$\Lambda$
do still not define these matrices uniquely.
Thus, the matrix~$\Lambda$ is defined up to adding an arbitrary matrix of the form $\bigoplus_i\nu_iE_{ii}$ in the chosen basis,
where $\nu_i=\nu_j$ if $\mu_i$ and $\mu_j$ belong to the same chain.
In other words, within each of the chains of eigenvalues of~$\Lambda$,
only their differences, which are integers by construction, are significant.
The structure of such a chain is completely defined by the tuple of differences (equal to~1 or~2) of successive eigenvalues from this chain.

In the case $\mathbb F=\mathbb R$, we make the complexification of~$\mathbb R^n$ to~$\mathbb C^n$.
In view of the above consideration, we can assume
that the complexification~$\Lambda^{\mathbb C}$ of the matrix~$\Lambda$ is diagonalizable.
If the matrix~$\Lambda^{\mathbb C}$ has a chain of nonreal eigenvalues,
then it has the chain of the complex conjugates of these eigenvalues as well.
It is obvious that the structures of the original chain and of its conjugate coincide,
i.e., differences of successive eigenvalues
are equal to each other for the respective pairs of such eigenvalues from these two chains.
We can choose a canonical basis for~$\Lambda^{\mathbb C}$ in such a way that respective generalized eigenvectors from this basis
that are associated with these chains are the complex conjugates of each other.
Due to the residual indefiniteness of~$\Lambda^{\mathbb C}$,
we can replace the eigenvalues in both the chains by their real parts.
The real and imaginary parts of the above generalized eigenvectors are generalized eigenvectors of the matrix obtained by this replacement.
Then we additionally replace each pair of conjugate original generalized eigenvectors in the chosen canonical basis by
the pair consisting of the real and imaginary parts of the first eigenvector in this pair.
To avoid merging new chains with each other or with preserved old chains,
we shift the eigenvalues from to the same chains by the same real number, which needs to be selected,
but the shifts are independent for different chains.
The constructed matrix~$\Lambda^{\mathbb C}_{\rm new}$ has only real eigenvalues and
possesses a canonical basis constituted by generalized eigenvectors with real components.
The total chain structure of~$\Lambda^{\mathbb C}_{\rm new}$ coincides with that of~$\Lambda^{\mathbb C}$,
and $\Lambda^{\mathbb C}-\Lambda^{\mathbb C}_{\rm new}\in\mathfrak s^{\mathbb C}$.
Consider the counterpart~$\Lambda_{\rm new}$ of~$\Lambda^{\mathbb C}_{\rm new}$ on~$\mathbb R^n$.
It is diagonalizable over~$\mathbb R$, and $\Lambda-\Lambda_{\rm new}\in\mathfrak s$.

As a result, in both the cases $\mathbb F=\mathbb C$ and $\mathbb F=\mathbb R$,
we can assume, without loss of generality, that \emph{the matrix~$\Lambda$ is diagonalizable}.

\medskip\par\noindent{\it Classification procedure.}
Counterparts of two ways for classifying Lie-symmetry extensions with $k=0$ or $k=1$
can also be suggested for the case $k=2$.

Within the first way, we again start with
a complete list of ${\rm SL}(n,\mathbb F)$-inequivalent subalgebras~$\mathfrak s$ of~$\mathfrak{sl}(n,\mathbb F)$
and select those satisfying the condition
$\mathfrak s={\rm C}_{\mathfrak{sl}(n,\mathbb F)}\big({\rm C}_{\mathfrak{sl}(n,\mathbb F)}(\mathfrak s)\big)
\notin\big\{\mathfrak{sl}(n,\mathbb F),\,\{0\}\big\}$.
Recall that in contrast to the case $k=1$, the subalgebra~$\mathfrak s=\{0\}$ is not appropriate here.
A~complete list of inequivalent essential Lie-symmetry extensions with $k=2$ in the class~$\mathcal L'_1$
consists of the algebras $\mathfrak g^{\rm ess}_V=\langle I\rangle\oplus\big(\langle P,D\rangle\lsemioplus\mathfrak s^{\rm vf}\big)$
for all the obtained subalgebras~$\mathfrak s$, where \mbox{$P:=\p_t+\Upsilon^{ab}x^b\p_{x^a}$} and \mbox{$D:=t\p_t+\Lambda^{ab}x^b\p_{x^a}$},
and the arbitrary-element matrix~$V$ takes the general values of the form~\eqref{eq:k=2RepresentationForV}.
The problem is to classify, up to the $G^\sim_{\mathcal L'}$-equivalence,
all possible general values of the matrix triple $(\Lambda,\Upsilon,W)$
defined up to adding arbitrary elements of~$\mathfrak s$ to~$\Lambda$ and~$\Upsilon$.
Following the case $k=1$, we fix a complementary subspace~$\mathsf s$ to~$\mathfrak s$ in ${\rm N}_{\mathfrak{sl}(n,\mathbb F)}(\mathfrak s)$
and denote $\mathfrak c:={\rm C}_{\mathfrak{sl}(n,\mathbb F)}(\mathfrak s)$.
At the same time, we modify the selection of matrices.
We first choose a non-nilpotent matrix $Z\in\mathsf s$, compute its semisimple part~$Z_{\rm s}$ and set $\Lambda:=Z_{\rm s}$.
We reorder and split the spectrum~$\sigma(\Lambda)=\{\mu_1,\dots,\mu_r\}$ of~$\Lambda$ into (maximal) chains,
where $\mu_i-\mu_{i+1}\in\{1,2\}$ for the eigenvalues within each of the chains.
If $\mathbb F=\mathbb R$ and $\sigma(\Lambda)\not\subset\mathbb R$,
we modify~$\Lambda$ by adding such a matrix from~$\mathfrak s$
that for the new~$\Lambda$ we have $\sigma(\Lambda)\subset\mathbb R$.
Then we choose a matrix $Y\in\mathsf s$ and take its hat-part~$\hat Y$ with respect to~$\Lambda$ as~$\Upsilon$,
i.e., we set $\Upsilon_{ij}:=Y_{ij}$ if $\mu_i=\mu_j+1$ and $\Upsilon_{ij}:=0$ otherwise
in the block form according to the decomposition $\mathbb F^n$ into the eigenspaces of~$\Lambda$.
Since $Z,Y\in{\rm N}_{\mathfrak{sl}(n,\mathbb F)}(\mathfrak s)$ and $Z-\Lambda,Y-\Upsilon\in\mathfrak s$,
the matrices~$\Lambda$ and~$\Upsilon$ also belong to~${\rm N}_{\mathfrak{sl}(n,\mathbb F)}(\mathfrak s)$.
Finally, we choose a nonzero matrix $W\in\mathfrak c$ for which $W_{ij}=0$ if $\mu_i\ne\mu_j+2$.
By the construction, we necessarily have $[\Lambda,\Upsilon]=\Upsilon$, $[\Lambda,K_l]=(l+2)K_l$, $K_l\in\mathfrak c$
for all the matrices $K_l$ defined in~\eqref{eq:k=2RepresentationForV},
and $\mathfrak s\subseteq{\rm C}_{\mathfrak{sl}(n,\mathbb F)}(\{K_0,\dots,K_m\})$,
where $m$ is the maximum value of~$l$ with $K_l\ne0$.
The essential Lie-symmetry extension case associated with a subalgebra~$\mathfrak s$ indeed exists and is genuine
if and only if there exist matrices~$Z$, $Y$ and~$W$ satisfying the above conditions and, moreover,
$\mathfrak s={\rm C}_{\mathfrak{sl}(n,\mathbb F)}(\{K_0,\dots,K_m\})$.
This is a value of the parameter-matrix triple $(Z,Y,W)$ that we call general among appropriate ones.
In other words, the algebra $\mathfrak g^{\rm ess}_V=\langle I\rangle\oplus\big(\langle P,D\rangle\lsemioplus\mathfrak s^{\rm vf}\big)$
indeed coincides with the entire algebra~$\mathfrak g^{\rm ess}_V$
for $V={\rm e}^{t\Upsilon}W{\rm e}^{-t\Upsilon}={\rm e}^{tY}W{\rm e}^{-tY}$
if and only if the triple $(Z,Y,W)$ takes a general value
among appropriate values in $\mathsf s\times\mathsf s\times\mathfrak c$.
We use the well-defined action of the stabilizer subgroup \smash{${\rm St}_{{\rm SL}(n,\mathbb F)}(\mathfrak s)$}
on $\mathsf s\times\mathsf s\times\mathfrak c$
for the further selection of the triple $(Z,Y,W)$ up to the $G^\sim_{\mathcal L'}$-equivalence.
The systems~$L'_V$ and~\smash{$L'_{\tilde V}$} associated with 
$(Z,Y,W),(\tilde Z,\tilde Y,\tilde W)\in\mathsf s\times\mathsf s\times\mathfrak c$
are $G^\sim_{\mathcal L'}$-equivalent if and only if
there exists a matrix in the group \smash{${\rm St}_{{\rm SL}(n,\mathbb F)}(\mathfrak s)$}
whose adjoint action maps the triple $(\Lambda,\Upsilon,W)$ to the triple $(\tilde\Lambda,\alpha\tilde\Upsilon,\alpha^2\tilde W)$
or, equivalently, $(Z,Y,W)$ to~$(\tilde Z,\alpha\tilde Y,\alpha^2\tilde W)$
for some $\alpha\in\mathbb F\setminus\{0\}$,
see Lemma~\ref{lem:SLODEEquivOfSystemsSimilarToConstCoeffOnes} below.%
\footnote{\label{fnt:ScalingsAsMatrixSimilarity}%
In view of the commutation relations $[\Lambda,\Upsilon]=\Upsilon$ and $[\Lambda,W]=2W$
implying a specific structure of~$\Upsilon$ and~$W$,
the matrix transformation $\tilde\Upsilon=\alpha\Upsilon$, $\tilde W=\alpha^2W$ with $\alpha\in\mathbb F\setminus\{0\}$,
which is induced by the scaling $\tilde t=t/\alpha$ of~$t$,
coincides with the adjoint action of a matrix \smash{$M\in{\rm St}_{{\rm SL}(n,\mathbb F)}(\mathfrak s)$}.
If $\alpha\in\mathbb R_{>0}$, then an appropriate matrix~$M$ is, e.g., $M={\rm e}^{-\ln(\alpha)\Lambda}$.
For the general $\alpha\in\mathbb F\setminus\{0\}$,
we can take $M=\bigoplus_{i=1}^r\alpha^{\mu_{i'}-\mu_i}E_{n_i}$,
where $\mu_{i'}$ is the first element in the chain containing~$\mu_i$,
and $E_{n_i}$ is the $n_i\times n_i$ identity matrix.
This is why in the course of checking the equivalence of matrix triples,
we can assume $\alpha=1$.
}

The second way is based on the classification of the parameter-matrix triples $(\Lambda,\Upsilon,W)$
that satisfy the above conditions, which have been derived up to the ${\rm SL}(n,\mathbb F)$-equivalence
and up to adding elements of~$\mathfrak s$ to~$\Lambda$ and to~$\Upsilon$.
Here~$\mathfrak s$ is defined as the maximal subalgebra of $\mathfrak{sl}(n,\mathbb F)$ such that
$[\Upsilon,\mathfrak s]\subseteq\mathfrak s$ and $[W,\mathfrak s]=\{0\}$ or, equivalently,
$\mathfrak s:={\rm C}_{\mathfrak{sl}(n,\mathbb F)}(\{K_0,\dots,K_m\})$,
where $K_0$,~\dots,~$K_m$ are the nonzero matrices given in~\eqref{eq:k=2RepresentationForV}.
Without loss of generality, we can assume $\Lambda$ to be diagonal.
We consider a complete set of chain structures for~$\Lambda$ that are inequivalent with respect to chain permutation.
Recall that the eigenvalues~$\mu_i$ of~$\Lambda$ are defined up to the uniform shifts of entire chains.
For each chain structure from the above set, we take arbitrary matrices~$\Upsilon$ and~$W$ in the block form
according to the decomposition $\mathbb F^n$ into the generalized eigenspaces~$\mathscr U_i$ of~$\Lambda$
such that $\Upsilon_{ij}=0$ if $\mu_i\ne\mu_j+1$ and $W_{ij}=0$ if $\mu_i\ne\mu_j+2$.
In addition, it suffices to consider only the regular values of~$(\Upsilon,W)$ for fixed~$\Lambda$,
i.e., such values that for each of them the chain structure of~$\Lambda$ is the finest among possible ones.
By the construction, we necessarily have $[\Lambda,\Upsilon]=\Upsilon$, $[\Lambda,W]=2W$, and hence $[\Lambda,K_l]=(l+2)K_l$, $l\in\mathbb N_0$.
We compute~$m:=\max\{l\in\mathbb N_0\mid K_l\ne0\}$ and~$\mathfrak s:={\rm C}_{\mathfrak{sl}(n,\mathbb F)}(\{K_0,\dots,K_m\})$.
We classify such matrices~$\Upsilon$ and~$W$ up to the equivalence generated by the transformations
\[
\tilde\Upsilon=M(\Upsilon+\Gamma)M^{-1},\quad
\tilde W=MWM^{-1},
\]
where $M$ and~$\Gamma$ are arbitrary matrices from ${\rm SL}(n,\mathbb F)$ and~$\mathfrak s$, respectively,
such that $M_{ij}=0$ if $i\ne j$ and $\Gamma_{ij}=0$ if $\mu_i\ne\mu_j+1$,
see footnote~\ref{fnt:ScalingsAsMatrixSimilarity}.
The different obtained tuples $(K_0,\dots,K_m)$ with the same value of~$m$
should additionally be checked for simultaneous similarity of their respective components
with respect to the same matrix from ${\rm SL}(n,\mathbb F)$.

\noprint{
\todo ??? For each general values of~$\Upsilon$ and~$W$ of the form imposed according to the chain structure of~$\Lambda$,
the restrictions of the semisimple parts of matrices from the corresponding algebra~$\mathfrak s$
to the subspaces associated with the eigenvalue chains of~$\Lambda$
are proportional to the unit matrices of appropriate size.
}

\begin{remark}\label{rem:SLODEConstCoeffsRepresentatives}
Considering the group classification for the class~$\mathcal L_1$,
we can require that all $G^\sim_{\mathcal L}$-inequivalent cases of Lie-symmetry extensions with $k\geqslant1$
are represented by systems with constant matrix coefficients.
More specifically, any system~$L'_V$ with $V(t)={\rm e}^{t\Upsilon}V(0){\rm e}^{-t\Upsilon}$
is mapped by the point transformation $\Phi$: $\tilde t=t$, $\tilde{\boldsymbol x}={\rm e}^{-t\Upsilon}\boldsymbol x$
to the system~$L_\vartheta$: $\tilde{\boldsymbol x}_{\tilde t\tilde t}=A\tilde{\boldsymbol x}_{\tilde t}+B\tilde{\boldsymbol x}$
whose matrix-valued coefficients $A=-2\Upsilon$ and $B=V(0)-\Upsilon^2$ are constant.
The Lie-symmetry vector fields $P=\p_t+\Upsilon^{ab}x^b\p_{x^a}$ and $Q_\Gamma=\Gamma^{ab}x^b\p_{x^a}$
of~$L'_V$, where $\Gamma\in\mathfrak s$, are pushed forward by~$\Phi$ to
the Lie-symmetry vector fields $\tilde P:=\p_{\tilde t}$
and $\tilde Q_\Gamma:=({\rm e}^{-\tilde t\Upsilon}\Gamma{\rm e}^{\tilde t\Upsilon})^{ab}\tilde x^b\p_{\tilde x^a}$ of~$L_\vartheta$.
In the case $k=2$, we can choose the matrices~$\Upsilon$ and~$\Lambda$ in such a way that $[\Lambda,\Upsilon]=\Upsilon$,
and then the vector field $D=t\p_t+\Lambda^{ab}x^b\p_{x^a}$ from $\mathfrak g^{\rm ess}_V$ is pushed forward by~$\Phi$ to
$\tilde D:=\tilde t\p_{\tilde t}+\Lambda^{ab}\tilde x^b\p_{\tilde x^a}$ from~$\mathfrak g^{\rm ess}_\vartheta$.
Therefore, reducing $t$-shift-invariant systems from the class~$\mathcal L'_1$
to their constant-coefficient counterparts from the class~$\mathcal L_1$,
we obtain more complicated counterparts for~$Q_\Gamma$,
whose components may depend on~$t$, in contrast to those of~$Q_\Gamma$.
\end{remark}

\begin{remark}\label{rem:SLODECommutingUpsilonAndW}
The simplest particular subcase for both cases $k=1$ and $k=2$ is $m=0$,
i.e., that the matrices~$\Upsilon$ and~$W$ commute~\cite{boyk2013a}.
Then the corresponding value of~$V$ is a constant matrix, $V=W$, \smash{$\mathfrak s={\rm C}_{\mathfrak{sl}(n,\mathbb F)}(\{W\})$},
and we can set $\Upsilon=0$ up to adding elements of~$\mathfrak s$,
which is obvious since $\p_t\in\mathfrak g^{\rm ess}_V$ for such~$V$.
The classification of systems from the class~$\mathcal L'_1$ with constant values of the matrix-valued parameter function~$V$
reduces to considering all possible Jordan normal forms of such~$V$.
Finding $\mathfrak s$ for a constant matrix~$V$ is the standard Frobenius problem in matrix theory,
whose solution is well known, see, e.g.,~\cite[Chapter~VIII]{gant1959A}.
\end{remark}

\noindent{$\boldsymbol{k=3.}$}
This means that there are three vector fields with linearly independent $t$-components
from the complementary subalgebra to $\langle I\rangle$ in the algebra~$\mathfrak g^{\rm ess}_V$.
Up to the \smash{$G^\sim_{\mathcal L'}$}-equivalence,
they take the form $P:=\p_t+\Upsilon^{ab}x^b\p_{x^a}$, $D:=t\p_t+\Lambda^{ab}x^b\p_{x^a}$ and $S:=t^2\p_t+\Theta^{ab}x^b\p_{x^a}$
with some (constant) matrices~$\Upsilon,\Lambda,\Theta\in\mathfrak{sl}(n,\mathbb F)$.
According to the consideration in the case \mbox{$k=2$}, the invariance of the system~$L'_V$
with respect to the vector fields~$P$ and~$D$ implies that
the matrix-valued function~$V$ is given by~\eqref{eq:k=2RepresentationForV}.
Collecting coefficients of $t^{m+1}$ in the matrix equation \mbox{$t^2V_t+4tV=[\Theta,V]$}
obtained by substituting the components of the vector field~$S$
into the classifying condition~\eqref{eq:ClassifyingCondL'_V},
we derive the equation $(m+4)K_m=0$, which contradicts the inequality $K_m\ne0$.
Therefore, \emph{the case $k=3$ is impossible}.

Since the $G^\sim_{\bar{\mathcal L}}$-orbit of any system from the class~$\bar{\mathcal L}_1$
contains systems from the subclass~$\mathcal L'_1$
and, in view of Theorem~\ref{thm:EquivGroupoidbarL},
the $t$-component of any element of $G^\sim_{\bar{\mathcal L}}$ depends only on~$t$,
we in fact prove the following theorem.

\begin{theorem}\label{thm:DimOfProjectionOfSLODEMAItoT}
$\dim\pi_*\mathfrak g_\theta\leqslant2$ for any $L_\theta\in\bar{\mathcal L}_1$.
\end{theorem}

Recall that $\pi$ denotes the natural projection from $\mathbb F_t\times\mathbb F^n_{\boldsymbol x}$ onto $\mathbb F_t$.

Theorems~\ref{thm:EquivGroupoidbarL} and~\ref{thm:DimOfProjectionOfSLODEMAItoT} imply
that a normal linear system of second-order ordinary differential equations
admits more than two Lie-symmetry vector fields with linearly independent \mbox{$t$-components}
if and only if it is $G^\sim_{\bar{\mathcal L}}$-equivalent to the elementary system $\boldsymbol x_{tt}=\boldsymbol0$,
i.e., it belongs to the class~$\bar{\mathcal L}_0$.
A similar claim for single second-order linear ordinary differential equations is trivial.

\section{Criteria of similarity}\label{sec:SLODESummaryOfGroupClassification}

Using the notation and results of Section~\ref{sec:DescriptionOfEssLieSymExtensions},
we present explicit criteria of similarity of systems from the classes~$\mathcal L'_1$ and~$\mathcal L_1$
that possess, as their Lie symmetries, shifts with respect to~$t$,
which can be composed with transformations of~$\boldsymbol x$ in the case of the class~$\mathcal L'_1$.
These criteria are useful in the course of the complete group classification
of the classes~$\mathcal L'_1$ and~$\mathcal L_1$ for particular values of~$n$.
\looseness=-1

We call a system $L'_V$ from the class~$\mathcal L'_1$
(resp.\ a system $L_\vartheta$ from the class~$\mathcal L_1$ with constant matrix coefficients $\vartheta=(A,B)$)
\emph{properly $t$-shift-invariant} if
$\pi_*\mathfrak g^{\rm ess}_V$ (resp.\ $\pi_*\mathfrak g^{\rm ess}_\vartheta$) belongs to
$\big\{\langle\p_t\rangle,\,\langle\p_t,t\p_t\rangle\big\}$.
Thus, the \emph{improper $t$-shift invariance} of such a system means that
$\pi_*\mathfrak g^{\rm ess}_V$ (resp.\ $\pi_*\mathfrak g^{\rm ess}_\vartheta$)
coincides with $\langle\p_t,{\rm e}^{\gamma t}\p_t\rangle$ for some $\gamma\in\mathbb F\setminus\{0\}$.

\begin{lemma}\label{lem:SLODEEquivOfSystemsSimilarToConstCoeffOnes}
Properly $t$-shift-invariant systems~$L'_V$ and~\smash{$L'_{\tilde V}$} from the class~$\mathcal L'_1$ with
\begin{gather}\label{eq:SLODEt-shiftInvFormOfV}
V(t)={\rm e}^{t\Upsilon}V(0){\rm e}^{-t\Upsilon}, \quad
\tilde V(\tilde t)={\rm e}^{\tilde t\tilde\Upsilon}\tilde V(0){\rm e}^{-\tilde t\tilde\Upsilon}
\end{gather}
are similar with respect to a point transformation if and only if there exist a nonzero $\alpha\in\mathbb F$,
an invertible matrix~$M$ and a matrix $\Gamma\in\mathfrak s:={\rm C}_{\mathfrak{sl}(n,\mathbb F)}\big(\{K_l,\,l\in\mathbb N_0\}\big)$,
where $K_0:=V(0)$, $K_{l+1}:=[\Upsilon,K_l]$, $l\in\mathbb N_0$, such that
\[
\tilde\Upsilon=\alpha M(\Upsilon+\Gamma)M^{-1},\quad
\tilde V(0)=\alpha^2MV(0)M^{-1}.
\]
\end{lemma}

\begin{proof}
The ``if'' part can be checked by direct computation.
A point transformation establishing the similarity of~$L'_V$ and~\smash{$L'_{\tilde V}$} is, e.g.,
$\tilde t=t/\alpha$, $\tilde{\boldsymbol x}=M\boldsymbol x$.

Let us prove the ``only if'' part.
Suppose that the systems~$L'_V$ and~\smash{$L'_{\tilde V}$} from the lemma's statement
are similar with respect to a point transformation.
This means in view of Theorem~\ref{thm:EquivGroupoidL'} that
these systems are $G^\sim_{\mathcal L'}$-equivalent,
and let an equivalence transformation $\mathscr T\in G^\sim_{\mathcal L'}$, which is of the form~\eqref{eq:EquivGroupL'},
establish their equivalence, \smash{$\mathscr T_*L'_V=L'_{\tilde V}$}.
The matrix~$\Upsilon$ is defined up to adding a matrix~$S$ from
${\rm C}_{\mathfrak{sl}(n,\mathbb F)}\big(\{K_l,\,l\in\mathbb N_0\}\big)$.
The systems~$L'_V$ and~\smash{$L'_{\tilde V}$} respectively possess the Lie-symmetry vector fields
$P:=\p_t+\Upsilon^{ab}x^b\p_{x^a}$ and $\tilde P:=\p_{\tilde t}+\tilde\Upsilon^{ab}\tilde x^b\p_{\tilde x^a}$.
Due to the similarity, we have that
$k_V:=\dim\pi_*\mathfrak g^{\rm ess}_V=k_{\tilde V}:=\dim\pi_*\mathfrak g^{\rm ess}_V\in\{1,2\}$.

The condition $k_V=k_{\tilde V}=1$ implies that $(\varpi_*\mathscr T^{-1})_*\tilde P=\alpha P+Q_\Gamma$
for some nonzero $\alpha\in\mathbb F$ and $\Gamma\in\mathfrak s$, where $Q_\Gamma:=\Gamma^{ab}x^b\p_{x^a}\in\mathfrak g^{\rm ess}_V$.
Recall that for the class~$\mathcal L'$, by $\varpi$ we denote the natural projection
from $\mathbb F_t\times\mathbb F^n_{\boldsymbol x}\times\mathbb F^{n^2}_V$ onto $\mathbb F_t\times\mathbb F^n_{\boldsymbol x}$.

If $k_V=k_{\tilde V}=2$, then it follows from the study of the corresponding case in Section~\ref{sec:DescriptionOfEssLieSymExtensions}
that we can choose the vector fields~$P$ and~$\tilde P$ belonging to
the derived algebras~$(\mathfrak g^{\rm ess}_V)'$ and~\smash{$(\mathfrak g^{\rm ess}_{\tilde V})'$}
of the algebras~$\mathfrak g^{\rm ess}_V$ and~\smash{$\mathfrak g^{\rm ess}_{\tilde V}$}, respectively,
\mbox{$P\in(\mathfrak g^{\rm ess}_V)'\subseteq\langle P\rangle\lsemioplus\mathfrak s^{\rm vf}$}
and \smash{$\tilde P\in(\mathfrak g^{\rm ess}_{\tilde V})'$}.
Since the pushforward of vector fields by $\varpi_*\mathscr T^{-1}$ induces an isomorphism
from~\smash{$\mathfrak g^{\rm ess}_{\tilde V}$} onto~$\mathfrak g^{\rm ess}_V$
and any such isomorphism maps~\smash{$(\mathfrak g^{\rm ess}_{\tilde V})'$} onto~$(\mathfrak g^{\rm ess}_V)'$,
we have \smash{$(\varpi_*\mathscr T^{-1})_*(\mathfrak g^{\rm ess}_{\tilde V})'=(\mathfrak g^{\rm ess}_V)'$}.
Hence the above condition $(\varpi_*\mathscr T^{-1})_*\tilde P=\alpha P+Q_\Gamma$
for some nonzero $\alpha\in\mathbb F$ and $\Gamma\in\mathfrak s$ holds true here as well.

Therefore, in both the cases $T_t=1/\alpha$, i.e., $T=(t+\beta)/\alpha$ for some $\beta\in\mathbb F$.
Denoting $M:=|\alpha|^{-1/2}C{\rm e}^{\beta\Upsilon}$ completes the proof.
\end{proof}

\begin{corollary}\label{cor:IntergrationOfEqsFromL1WithTwoCertainSyms}
Properly $t$-shift-invariant systems~$L_\vartheta$ and~\smash{$L_{\tilde\vartheta}$} from the class~$\mathcal L_1$ with
constant matrix coefficients $\vartheta=(A,B)$ and $\tilde\vartheta=(\tilde A,\tilde B)$
are similar with respect to a point transformation if and only if
there exist a nonzero $\alpha\in\mathbb F$,
an invertible matrix~$M$ and a matrix \mbox{$\check\Gamma\in\mathfrak s:={\rm C}_{\mathfrak{sl}(n,\mathbb F)}\big(\{K_l,\,l\in\mathbb N_0\}\big)$},
where $K_0:=B$, $K_{l+1}:=[A,K_l]$, $l\in\mathbb N_0$, such that
\[
\tilde A=\alpha M\big(A+\check\Gamma\big)M^{-1},\quad
\tilde B=\alpha^2M\big(B-\tfrac14(A\check\Gamma+A\check\Gamma+\check\Gamma^2)\big)M^{-1}.
\]
\end{corollary}

\begin{proof}
The systems~$L_\vartheta$ and~\smash{$L_{\tilde\vartheta}$} are mapped to the systems~$L'_V$ and~\smash{$L'_{\tilde V}$}
with~$V$ and~$\tilde V$ of the form~\eqref{eq:SLODEt-shiftInvFormOfV} in the variables $(\hat t,\hat{\boldsymbol x})$
by the transformations $\hat t=t$, $\hat{\boldsymbol x}={\rm e}^{-t\Upsilon}\boldsymbol x$
and $\hat t=t$, $\hat{\boldsymbol x}={\rm e}^{-t\tilde\Upsilon}\boldsymbol x$, respectively,
where
$\Upsilon=-\frac12A$, $V(0)=B+\Upsilon^2$,
$\tilde\Upsilon=-\frac12\tilde A$, $\tilde V(0)=\tilde B+\tilde\Upsilon^2$.
The systems~$L_\vartheta$ and~\smash{$L_{\tilde\vartheta}$} are similar with respect to a point transformation if and only if
the systems~$L'_V$ and~\smash{$L'_{\tilde V}$} are similar with respect to a point transformation.
The relation between $\vartheta$ and $\tilde\vartheta$ follows from that between $(\Upsilon,V(0))$ and $(\tilde\Upsilon,\tilde V(0))$
in Lemma~\ref{lem:SLODEEquivOfSystemsSimilarToConstCoeffOnes}, where $\check\Gamma=-\frac12\Gamma$.
\end{proof}

\section{Properties of Lie invariance algebras}\label{sec:PropertiesOfIAs}

Using results of Section~\ref{sec:DescriptionOfEssLieSymExtensions},
we can more precisely characterize Lie invariance algebras of systems
not only from the regular subclasses $\bar{\mathcal L}_1$, $\mathcal L_1$, $\mathcal L'_1$ and $\mathcal L''_1$
but also from their singular counterparts.

\begin{theorem}\label{thm:MaxDimOfMIAinL1}
For any system from the class~$\bar{\mathcal L}_1$ (resp.\ $\mathcal L_1$, $\mathcal L'_1$ or $\mathcal L''_1$),
the dimension of its maximal Lie invariance algebra is greater than or equal to $2n+1$ and less than or equal to $n^2+4$.
The lower bound is greatest and is attained for a general system of the class.
The upper bound is least, and it is attained only for the systems from the orbit
of the system~\eqref{eq:CanonicalSystemOfMaxDimOfMIAinL1} with respect to the corresponding equivalence group.
\end{theorem}

\begin{proof}
The claim on the lower bound is trivial, see the end of Section~\ref{sec:SLODEPreliminaryAnalysisOfLieSyms}.

For each of the classes $\bar{\mathcal L}_1$, $\mathcal L_1$, $\mathcal L'_1$ and $\mathcal L''_1$,
the system~\eqref{eq:CanonicalSystemOfMaxDimOfMIAinL1} admits the formal interpretation as
an element of this class.
The difference of these interpretations is just originated from the different class parameterizations.
Thus, the corresponding values of the arbitrary elements are
$A=0$, $B=J$ and $\boldsymbol f=\boldsymbol0$ for the class~$\bar{\mathcal L}_1$,
the same $A$ and~$B$ for the class~$\mathcal L_1$,
and $V=J$ for the classes~$\mathcal L'_1$ and~$\mathcal L''_1$.
Here and in what follows $J=J_0^2\oplus \big(\bigoplus_{i=1}^{n-2}J_0^1\big)$,
and $J_{\mu}^m$ denotes the Jordan block of dimension $m$ with eigenvalue~$\mu$.
In addition, an equality like $V=M$ for the arbitrary-element matrix~$V$ and a constant matrix~$M$
means that $V(t)=M$ for any $t\in\mathcal I$,
and then we can use the notation~$L'_M$ and~$\mathfrak g_M^{}$ instead of~$L'_V$ and~$\mathfrak g_V^{}$
for this particular value of the arbitrary-element matrix~$V$.

The class~$\bar{\mathcal L}_1$ (resp.\ $\mathcal L_1$) can be mapped by a family of its equivalence transformations
to its subclass that is the embedded counterpart of the class~$\mathcal L'_1$.
The class~$\mathcal L''_1$ is a subclass of~$\mathcal L'_1$.
The equivalence transformations within each of the above classes preserve
the required properties of Lie invariance algebras.
This is why it suffices to prove the theorem only for the class~$\mathcal L'_1$.

For any system~$L'_V$ from the class~$\mathcal L'_1$, we have
$\mathfrak g_V^{}=\mathfrak g^{\rm ess}_V\lsemioplus\mathfrak g^{\rm lin}_V$, and $\dim\mathfrak g^{\rm lin}_V=2n$.
Hence the proof reduces to estimating the dimensions of the essential Lie invariance algebras of systems from this class.

Given a system~$L'_V$, where the independent variable~$t$ runs a domain~$\mathcal I$,
for any fixed $t_0\in\mathcal I$, we consider the system~$L'_{V_0}$ from~$\mathcal L'_1$,
where $V_0(t):=V(t_0)$ for any $t\in\mathcal I$.
We necessarily have $\dim\mathfrak g^{\rm ess}_V\leqslant\dim\mathfrak g^{\rm ess}_{V_0}$.
Indeed, $\dim\mathfrak g^{\rm ess}_V=\dim\langle I\rangle+\dim\mathfrak s_V+k_V$ and $\dim\mathfrak g^{\rm ess}_{V_0}=\dim\langle I\rangle+\dim\mathfrak s_{V_0}+k_{V_0}$,
where $I:=x^a\p_{x^a}$, $k_V:=\dim\pi_*\mathfrak g^{\rm ess}_V$, $k_{V_0}:=\dim\pi_*\mathfrak g^{\rm ess}_{V_0}$, and
\begin{gather}\label{eq:CentralizersOfVAndItsValue}
\mathfrak s_V:={\rm C}_{\mathfrak{sl}(n,\mathbb F)}\big(\{V(t)\mid t\in\mathcal I\}\big)
\subseteq{\rm C}_{\mathfrak{sl}(n,\mathbb F)}\big(\{V(t_0)\}\big)
={\rm C}_{\mathfrak{sl}(n,\mathbb F)}\big(\{V_0(t)\mid t\in\mathcal I\}\big)=:\mathfrak s_{V_0},
\end{gather}
see the notation and results in Section~\ref{sec:DescriptionOfEssLieSymExtensions}, especially those related to the case $k=2$.
The last inclusion means that $\dim\mathfrak s_V\leqslant\dim\mathfrak s_{V_0}$.
Since $V_0$ is a constant (matrix-valued) function, it is obvious that
the system~$L'_{V_0}$ possesses the vector field~$\p_t$ as its Lie symmetry, i.e., $k_{V_0}\geqslant1$.
If $k_V=2$, then $V(t)$ is a nilpotent matrix for any~$t\in\mathcal I$, including~$t_0$,
and thus $k_{V_0}=2$ as well.
This is why in any case $k_V\leqslant k_{V_0}$, which finally proves the required claim.

Therefore, the maximum dimension of the essential Lie invariance algebras of systems from the class~$\mathcal L'_1$
is attained by a system with a constant value of the arbitrary-element matrix~$V$, $V=K_0$
for some $K_0\in\mathfrak{gl}(n,\mathbb F)\setminus\langle E\rangle$.
For this system, we have $\mathfrak s_V={\rm C}_{\mathfrak{sl}(n,\mathbb F)}(\{K_0\})$.
As a result, the Frobenius problem comes into play: find all the matrices~$\Gamma$ commuting with the fixed matrix~$K_0$,
which is a standard matrix problem~\cite[Chapter~VIII]{gant1959A}.
By $N(K_0)$ we denote the number of linearly independent matrices that commute with~$K_0$, i.e., $\dim\mathfrak s_V=N(K_0)-1$.
Since the matrix~$K_0$ is not proportional to the identity matrix,
the value~$N(K_0)$ can be at most $n^2-2n+2$, which is reached for only
for the matrices that are similar to either
$M_{1\nu}:=J_{\nu}^2\oplus \big(\bigoplus_{i=1}^{n-2}J_{\nu}^1\big)$ or
$M_{2\nu_1\nu_2}:=\big(\bigoplus_{i=1}^{n-1}J_{\nu_1}^1\big)\oplus J_{\nu_2}^1$ with $\nu_1\ne\nu_2$,
see the derivation of Corollaries~2 and~3 in~\cite{boyk2013a}.
The first matrix with $\nu=0$ coincides with~$J$ and is nilpotent.
Setting $K_0=J$, we have $k_J=2$, and then $\dim\mathfrak g^{\rm ess}_J=N(J)+k_J=n^2-2n+4$,
which is maximum in the class~$\mathcal L'_1$ by the construction.

Now we show that any system~$L'_V$ from the class~$\mathcal L'_1$ with $\dim\mathfrak g^{\rm ess}_V=n^2-2n+4$
belongs to the \smash{$G^\sim_{\mathcal L'}$}-orbit of the system~$L'_J$.
The value of $\dim\mathfrak g^{\rm ess}_V$ is maximum if and only if
both the values $\dim\mathfrak s_V$ and $k_V$ are maximum as well, $\dim\mathfrak s_V=n^2-2n+2$ and $k_V=2$.
In view of~\eqref{eq:CentralizersOfVAndItsValue}, the former equality implies
that for any fixed $t\in\mathcal I$, where the matrix~$V(t)$ is nonzero, it is similar to
either $M_{1\nu}$ for some~$\nu\in\mathbb F$ or $M_{2\nu_1\nu_2}$ for some~$\nu_1,\nu_2\in\mathbb F$ with $\nu_1\ne\nu_2$.
The latter equality means that the algebra $\mathfrak g^{\rm ess}_V$
contains two vector fields with linearly independent $t$-components~$\tau^1$ and~$\tau^2$.
Modulo the \smash{$G^\sim_{\mathcal L'}$}-equivalence, we can assume that $\tau^1=1$ and $\tau^2=t$,
i.e., $\mathfrak g^{\rm ess}_V\ni P,D$, where
$P:=\p_t+\Upsilon^{ab}x^b\p_{x^a}$ and $D:=t\p_t+\Lambda^{ab}x^b\p_{x^a}$
with some (constant) matrices~$\Upsilon,\Lambda\in\mathfrak{sl}(n,\mathbb F)$.
Then, it follows from the consideration of the case $k=2$ in Section~\ref{sec:DescriptionOfEssLieSymExtensions}
that for any $t\in\mathcal I$, the matrix~$V(t)$ is nilpotent,
and thus it is necessarily similar to $M_{10}=J$ if it is nonzero.
We fix a point~$t_0\in\mathcal I$ with $V(t_0)\ne0$.
Up to the \smash{$G^\sim_{\mathcal L'}$}-equivalence, we can assume that $V(t_0)=J$.
(For this, we should use the equivalence transformation~\eqref{eq:EquivGroupL'} with $T=t$
and a matrix~$C$ establishing the similarity of~$J$ to~$V(t_0)$, $V(t_0)=C^{-1}JC$.)
According to~\eqref{eq:CentralizersOfVAndItsValue},
\[
\mathfrak s_V:=\bigcap_{t\in\mathcal I}{\rm C}_{\mathfrak{sl}(n,\mathbb F)}\big(\{V(t)\}\big)
\subseteq{\rm C}_{\mathfrak{sl}(n,\mathbb F)}\big(\{V(t_0)\}\big)=
{\rm C}_{\mathfrak{sl}(n,\mathbb F)}\big(\{J\}\big)=\mathfrak s_J.
\]
On the other hand, $\dim\mathfrak s_V=n^2-2n+2=\dim\mathfrak s_J$, and hence $\mathfrak s_V=\mathfrak s_J$.
This means that for any $t\in\mathcal I$ with $V(t)\ne0$, the equality
${\rm C}_{\mathfrak{sl}(n,\mathbb F)}\big(\{V(t)\}\big)={\rm C}_{\mathfrak{sl}(n,\mathbb F)}\big(\{J\}\big)$
holds true,
i.e., $V(t)\in {\rm C}_{\mathfrak{sl}(n,\mathbb F)}\big({\rm C}_{\mathfrak{sl}(n,\mathbb F)}\big(\{J\}\big)\big)=\langle J\rangle$.
(The last equality can be checked by the direct two-step computation of the corresponding centralizers.)
For points~$t$ with $V(t)=0$, we also have $V(t)\in\langle J\rangle$.
In other words, the matrix-valued function~$V$ is defined by $V(t)=v(t)J$ for any $t\in\mathcal I$,
where $v$ is a nonzero function of~$t$.
Successively substituting the vector fields~$P$ and~$D$ into the classifying condition~\eqref{eq:ClassifyingCondL'_V} with this~$V$
leads to the equations $v_tJ=v[\Lambda,J]$ and $(tv_t+2v)J=v[\Upsilon,J]$.
Therefore, $[\Lambda,J]=\lambda_1J$ and $[\Upsilon,J]=\lambda_2J$ for some constants~$\lambda_1$ and~$\lambda_2$,
and thus $v_t=\lambda_1v$ and $(tv_t+2v)=\lambda_2v$.
The last system of two ordinary differential equations with respect to~$v$ is consistent
if and only if $(\lambda_1,\lambda_2)=(0,2)$, i.e., $v$ is a~nonzero constant function,
which is equal to one up to the \smash{$G^\sim_{\mathcal L'}$}-equivalence.
This means that the initial system~$L'_V$ is \smash{$G^\sim_{\mathcal L'}$}-equivalent to the system~$L'_J$,
which completes the proof.
\end{proof}

\begin{theorem}\label{thm:StructureOfgessInL'1}
For any system~$L'_V$ from the class~$\mathcal L'_1$,
its essential Lie invariance algebra~$\mathfrak g^{\rm ess}_V$ can be represented as
$
\mathfrak g^{\rm ess}_V=\mathfrak i\oplus(\mathfrak t_V\lsemioplus\mathfrak s^{\rm vf}_V)
$,
where
\begin{itemize}\itemsep=0.ex
\item
$\mathfrak i:=\langle x^a\p_{x^a}\rangle$ is an ideal of~$\mathfrak g^{\rm ess}_V$,
which is common for all systems from the class~$\mathcal L'_1$,
\item
$\mathfrak s^{\rm vf}_V:=\{\Gamma^{ab}x^b\p_{x^a}\mid\Gamma\in\mathfrak{sl}(n,\mathbb F)\colon[\Gamma,V]=0\}$
is an ideal of~$\mathfrak g^{\rm ess}_V$ with $\dim\mathfrak s^{\rm vf}_V\leqslant n^2-2n+1$, and
\item
$\mathfrak t_V:=\big\langle\tau^\iota\p_t+\big(\tfrac12\tau^\iota_tx^a+\Lambda^{ab}_\iota x^b\big)\p_{x^a},\,\iota=1,\dots,k_V\big\rangle$
is a subalgebra of~$\mathfrak g^{\rm ess}_V$ with $\dim\mathfrak t_V=k_V\in\{0,1,2\}$,
the components $\tau^\iota=\tau^\iota(t)$ are linearly independent, and
each of $\tau^\iota$ satisfies~\eqref{eq:ClassifyingCondL'_V} with $\Gamma=\Lambda_\iota\in\mathfrak{sl}(n,\mathbb F)$.
\end{itemize}
\end{theorem}

\begin{proof}
In view of Corollary~\ref{cor:MAIofL'_V},
each vector field from the algebra~$\mathfrak g^{\rm ess}_V$ takes the form
\begin{gather}\label{eq:FormOfElementsOfEssLieInvAlgebras}
Q=\tau\p_t+\big(\tfrac12\tau_tx^a+\Gamma^{ab}x^b\big)\p_{x^a},
\end{gather}
where the $t$-component $\tau$ depends only on~$t$ and satisfies,
jointly with the (constant) matrix $\Gamma=(\Gamma^{ab})$,
the classifying condition~\eqref{eq:ClassifyingCondL'_V}.
An obvious solution of~\eqref{eq:ClassifyingCondL'_V} for any~$V$ is given by $\tau=0$ and $\Gamma=E$.
The corresponding vector field $I:=x^a\p_{x^a}$
commutes with any vector field of the form~\eqref{eq:FormOfElementsOfEssLieInvAlgebras}.
Hence $\mathfrak i$ is an ideal of~$\mathfrak g^{\rm ess}_V$ irrelatively of~$V$.
The set $\mathfrak s^{\rm vf}_V$ consists of the elements~$Q$ of~$\mathfrak g^{\rm ess}_V$ with $\pi_*Q=0$
and is clearly an ideal of~$\mathfrak g^{\rm ess}_V$.
The estimate $\dim\mathfrak s^{\rm vf}_V\leqslant n^2-2n+1$ follows
from the proof of Theorem~\ref{thm:MaxDimOfMIAinL1}.

It has been proved in Section~\ref{sec:DescriptionOfEssLieSymExtensions}
that $k_V:=\dim\pi_*\mathfrak g^{\rm ess}_V\in\{0,1,2\}$.
Note that $\dim\mathfrak t_V=k_V$.
Now we show that one can always select a complementary subspace~$\mathfrak t_V$
to~$\mathfrak i\oplus\mathfrak s^{\rm vf}_V$ in~$\mathfrak g^{\rm ess}_V$
in such a way that it is a subalgebra of~$\mathfrak g^{\rm ess}_V$.
The cases $k_V=0$ and $k_V=1$ are trivial.
Indeed, $\mathfrak t_V=\{0\}$ if $k_V=0$, and this set is an improper subalgebra of~$\mathfrak g^{\rm ess}_V$.
If $k_V=1$, then the span $\mathfrak t_V$ is one-dimensional,
and thus it is again a subalgebra of~$\mathfrak g^{\rm ess}_V$.
Let us prove the above claim for the case $k_V=2$,
which is not so evident as the previous cases.

It follows from the consideration of the case $k=2$ in Section~\ref{sec:DescriptionOfEssLieSymExtensions}
that two vector fields with linearly independent $t$-components from the algebra~$\mathfrak g^{\rm ess}_V$
take, up to their linear recombination and the \smash{$G^\sim_{\mathcal L'_1}$}-equivalence,
the form $P:=\p_t+\Upsilon^{ab}x^b\p_{x^a}$ and $D:=t\p_t+\Lambda^{ab}x^b\p_{x^a}$
with some (constant) matrices~$\Upsilon,\Lambda\in\mathfrak{sl}(n,\mathbb F)$.
Let $\Lambda_{\rm s}$ and~$\Lambda_{\rm n}$ be the semisimple and nilpotent parts
in the Jordan--Chevalley decomposition of the matrix $\Lambda=(\Lambda^{ab})$, $\Lambda=\Lambda_{\rm s}+\Lambda_{\rm n}$.
The vector field $\Lambda_{\rm n}^{ab}x^b\p_{x^a}$ belongs to~$\mathfrak s^{\rm vf}_V\subset\mathfrak g^{\rm ess}_V$.
Hence the vector field $D-\Lambda_{\rm n}^{ab}x^b\p_{x^a}=t\p_t+\Lambda_{\rm s}^{ab}x^b\p_{x^a}$
is an element of~$\mathfrak g^{\rm ess}_V$ as well,
i.e., we can assume from the very beginning that the matrix~$\Lambda$ is semisimple.
Then $\Upsilon=\hat\Upsilon+\check\Upsilon$, where $\hat\Upsilon$ and $\check\Upsilon$
are the hat- and check-parts of~$\Upsilon$ with respect to the (semisimple) matrix~$\Lambda$,
which are defined in Section~\ref{sec:DescriptionOfEssLieSymExtensions}.
It has been shown therein that the vector field $\check\Upsilon^{ab}x^b\p_{x^a}$
also belongs to~$\mathfrak s^{\rm vf}_V\subset\mathfrak g^{\rm ess}_V$.
This means that we can replace the vector field~$P$ by $P-\check\Upsilon^{ab}x^b\p_{x^a}$,
thus replacing~$\Upsilon$ by~\smash{$\hat\Upsilon$}.
Then $[\Lambda,\Upsilon]=\Upsilon$, which leads to the commutation relation $[P,D]=P$.
In other words, the span $\mathfrak t_V:=\langle P,D\rangle$ is a subalgebra of~$\mathfrak g^{\rm ess}_V$.
To complete the proof of the claim, merely note that elements of~\smash{$G^\sim_{\mathcal L'_1}$}
induce isomorphisms between the essential Lie invariance algebras of equivalent systems from the class~$\bar{\mathcal L}_1$.

It is obvious that $\mathfrak i\cap\mathfrak s^{\rm vf}_V=\mathfrak i\cap\mathfrak t_V=\mathfrak s^{\rm vf}_V\cap\mathfrak t_V=\{0\}$,
$[\mathfrak i,\mathfrak s^{\rm vf}_V]=[\mathfrak i,\mathfrak t_V]=\{0\}$,
and $[\mathfrak s^{\rm vf}_V,\mathfrak t_V]\subseteq\mathfrak s^{\rm vf}_V$.
\end{proof}

\begin{corollary}\label{cor:StructureOfgessInL1}
For any system~$L_\vartheta$ from the class~$\mathcal L_1$,
its essential Lie invariance algebra~$\mathfrak g^{\rm ess}_\vartheta$ can be represented as
$\mathfrak g^{\rm ess}_\vartheta=\mathfrak i\oplus(\mathfrak t_\vartheta\lsemioplus\mathfrak s^{\rm vf}_\vartheta)$,
where $\mathfrak i:=\langle I\rangle$ is an ideal of~$\mathfrak g^{\rm ess}_\vartheta$,
which is common for all the systems from the class~$\mathcal L_1$,
$\mathfrak s^{\rm vf}_\vartheta$ is an ideal of~$\mathfrak g^{\rm ess}_\vartheta$ with $\pi_*\mathfrak s^{\rm vf}_\vartheta=\{0\}$
and $\dim\mathfrak s^{\rm vf}_\vartheta\leqslant n^2-2n+1$, and
$\mathfrak t_\vartheta$ is a subalgebra of~$\mathfrak g^{\rm ess}_\vartheta$
with $\dim\mathfrak t_\vartheta=\dim\pi_*\mathfrak t_\vartheta\in\{0,1,2\}$.
\end{corollary}

\begin{proof}
To derive this corollary from Theorem~\ref{thm:StructureOfgessInL'1},
we use the same arguments as in the third paragraph of the proof of Theorem~\ref{thm:MaxDimOfMIAinL1}.
More specifically, the class~$\mathcal L_1$ is mapped by a family of its equivalence transformations
to the embedded counterpart of the class~$\mathcal L'_1$,
see the discussion related to~\eqref{eq:SLODESubclassEmbeddings}.
The equivalence transformations within the class~$\mathcal L_1$
preserve the structure of the essential Lie invariance algebras of systems from this class.
\end{proof}

An assertion similar to Corollary~\ref{cor:StructureOfgessInL1}
also holds for all systems from the regular subclass~$\bar{\mathcal L}_1$ of the class~$\bar{\mathcal L}$
with modifications allowing for specific features of the notion of essential Lie invariance algebras
in a class of inhomogeneous linear systems of differential equations,
see Remark~\ref{rem:EssLieInvAlgebrasInL'1AndBarL1}.

The following proposition collects various properties of Lie symmetries of systems from the singular subclass~$\bar{\mathcal L}_0$
of the class~$\bar{\mathcal L}$.
A part of this proposition over the complex field was proved in~\cite[Theorem 1]{gonz1988d}. 
We combine Theorem~\ref{thm:EquivGroupoidbarL},
the description of the subclass~$\bar{\mathcal L}_0$ before this theorem,
the well-known facts presented in the paragraph with the equation~\eqref{eq:SLODEMIAOfElementarySystem}
and Theorem~\ref{thm:MaxDimOfMIAinL1}.

\begin{proposition}\label{pro:SLODEsl(n+2,F)Invariance}
The following are equivalent for a system~$\bar L_\theta$ of the form~\eqref{SLODE}:

(i) The maximal Lie invariance algebra~$\mathfrak g_\theta$ of~$\bar L_\theta$  is isomorphic to $\mathfrak{sl}(n+2,\mathbb F)$.

(ii) $\dim\mathfrak g_\theta=(n+2)^2-1$.

(iii) $\dim\mathfrak g_\theta>n^2+4$.

(iv) The system~$\bar L_\theta$ is invariant with respect to a Lie algebra of vector fields with zero $t$-com\-po\-nents
that is isomorphic to $\mathfrak{sl}(n,\mathbb F)$.

(v) The system~$\bar L_\theta$ is reduced by a point transformation in $\mathbb F_t\times\mathbb F^n_{\boldsymbol x}$
to the system~$\boldsymbol x_{tt}=\boldsymbol0$.

(vi) The matrix $B-\frac12A_t-\frac14A^2$ is proportional to the identity matrix with time-dependent proportionality factor.
\end{proposition}

\section[Reduction of order and integration using Lie symmetries and equivalence transformations]%
{Reduction of order and integration using Lie symmetries\\ and equivalence transformations}\label{sec:SLODEOrderReductionAndIntegrationUsingLieSyms}

Lie symmetries of normal linear systems of $n$ second-order ordinary differential equations
and equivalence point transformations between such systems can be efficiently used
for reducing their total order and their integration.

We begin with the singular subclass~$\bar{\mathcal L}_0$ of~$\bar{\mathcal L}$.
Recall that the belonging of systems to this subclass is checked by verifying the condition
that the matrix-valued function $B-\frac12A_t+\frac14A^2$
is proportional to the identity matrix~$E$ with time-dependent proportionality factor.

\begin{proposition}\label{pro:IntergrationOfeqsFrombarL0}
The integration of any system~$\bar L_\theta$ from the class~$\bar{\mathcal L}_0$ is reduced to at most $2n$ quadratures
via finding a fundamental matrix of the system of $n$ first-order ordinary differential equations
$\boldsymbol y_t+\frac12A^{\mathsf T}\boldsymbol y=\boldsymbol0$
with respect to the vector-valued function~$\boldsymbol y$ of~$t$
and a fundamental set of solutions of the second-order ordinary differential equation
$n\varphi_{tt}=\mathop{\rm tr}\big(B-\frac12A_t+\frac14A^2\big)\varphi$
with respect to the function~$\varphi$ of~$t$.
\end{proposition}

\begin{proof}
Fix a value of the arbitrary-element tuple $\theta=(A,B,\boldsymbol f)$.
Let a matrix-valued function~$M$ of~$t$ be a fundamental matrix of the system $\boldsymbol y_t+\frac12A^{\mathsf T}\boldsymbol y=\boldsymbol0$
and let functions~$\varphi^1$ and~$\varphi^2$ form
a fundamental set of solutions of the equation $\varphi_{tt}=U\varphi$ with $U:=n^{-1}\mathop{\rm tr}\big(B-\frac12A_t+\frac14A^2\big)$,
where $\varphi^2(t)\ne0$ for~$t$ from the interval under consideration.
Then the ratio $\varphi^1/\varphi^2=:T$ satisfies the ordinary differential equation $\{T,t\}=-2U$
with the Schwarzian derivative~$\{T,t\}$ of~$T$ with respect to~$t$.
Hence the equivalence transformation~$\mathscr T$ of the form~\eqref{eq:EquivGroupbarL}
with this $T$, $H:=M^{\mathsf T}$ and $\boldsymbol h=\boldsymbol0$
reduces the system~$\bar L_\theta$ to the system $\tilde{\boldsymbol  x}_{\tilde t\tilde t}=\,\,\tilde{\!\!\boldsymbol f}(\tilde t)$,
where $\,\,\tilde{\!\!\boldsymbol f}\circ T=T_t^{\,-2}H\boldsymbol f$, which is obviously integrated by at most $2n$ quadratures.
\end{proof}

In other words, for the systems from the class~$\bar{\mathcal L}_0$,
the total order of systems to be integrated can be efficiently lowered by~$n-2$.
It is obvious that we need no further quadratures after the equivalence transformation~$\mathscr T$ at all
if the system~$\bar L_\theta$ is homogeneous,
i.e., it belongs, modulo neglecting the arbitrary elements~$\boldsymbol f$, to the class~$\mathcal L_0$.

Solving systems from the class~$\bar{\mathcal L}_1$ requires more delicate approaches
involving specific Lie symmetries of these systems.
At first consider an approach based on knowing single specific symmetries of systems from the class~$\mathcal L_1$.

\begin{proposition}\label{pro:SLODEIntergrationOfSytemsFromL1WithOneNontrivSym}
If a system~$L_\vartheta$ from the class~$\mathcal L_1$ admits a known Lie-symmetry vector field with nonzero $t$-component,
then its integration is reduced by one quadrature and algebraic operations to constructing a fundamental matrix
of a linear system of $n$ first-order ordinary differential equations,
i.e., the total order of system to be integrated can be efficiently lowered by~$n$.
\end{proposition}

\begin{proof}
According to the assumption, the system~$L_\vartheta$ is invariant with respect to a vector field~$Q$
of the form~\eqref{eq:GenElementOfInvAlgOfLvartheta} with known functions $\tau$, $\eta^{ab}$ and $\chi^a$ of~$t$, and $\tau\ne0$.
Then this system is also invariant with respect to the vector field~$\hat Q$
with the same values of $\tau$ and $\eta^{ab}$, whereas $\chi^a=0$.
We can straighten the vector field~$\hat Q$ to~$\p_{\tilde t}$
by a point transformation~$\Phi$: $\tilde t=T(t)$, $\tilde{\boldsymbol x}=H(t)\boldsymbol x$,
where $\tau T_t=1$ and $\tau H_t+H\eta=0$.
In the variables $(\tilde t,\tilde{\boldsymbol x})$, the system~$L_\vartheta$ takes the form
$\tilde{\boldsymbol x}_{\tilde t\tilde t}=\tilde A\tilde{\boldsymbol x}_{\tilde t}+\tilde B\tilde{\boldsymbol x}$,
where $\tilde A$ and~$\tilde B$ are constant matrices
in view of the invariance of the system with respect to the vector field~$\p_{\tilde t}$.
Since the last form of the system~$L_\vartheta$ is solved using only algebraic operations,
the solution of this system reduces to finding $T$ from the equation $T_t=1/\tau$ by quadrature
and constructing~$H$ as the transpose of a fundamental matrix of the system of $n$ first-order ordinary differential equations
$\tau\boldsymbol y_t+\eta\boldsymbol y=\boldsymbol0$.
\end{proof}

Based on the proof of Proposition~\ref{pro:SLODEIntergrationOfSytemsFromL1WithOneNontrivSym},
we formulate the following step-by-step procedure
for the integration of a system~$L_\vartheta$ from the class~$\mathcal L_1$
with involving its known Lie-symmetry vector field
$Q=\tau\p_t+(\eta^{ab}x^b+\chi^{a})\p_{x^a}$ with the nonzero $t$-component~$\tau$:

\begin{enumerate}\itemsep=0ex
\item
Replace $Q$ by its counterpart $Q'=\tau\p_t+\eta^{ab}x^b\p_{x^a}$
with the zero values of~$\chi^a$.
\item
Find~$H$ as the transpose of a fundamental matrix of solutions
of the system of $n$ first-order ordinary differential equations
$\tau\boldsymbol y_t+\eta\boldsymbol y=\boldsymbol0$.
\item
According to~\eqref{eq:EquivGroupbarLB}--\eqref{eq:EquivGroupbarLC},
compute the matrices~$\tilde A$ and~$\tilde B$, which are necessarily constant.
\item
Solve the system~$L_{\tilde\vartheta}$ with constant matrix coefficients $\tilde\vartheta=(\tilde A,\tilde B)$.
\item
Find $T$ by integrating the equation $T_t=1/\tau$.
\item
Push forward the found general solution of the system~$L_{\tilde\vartheta}$
to the general solution of the system~$L_\vartheta$
by the transformation $\Phi^{-1}$, $\boldsymbol x(t)=H^{-1}(t)\tilde{\boldsymbol x}\big(T(t)\big)$.
\end{enumerate}

\begin{corollary}\label{cor:SLODEIntergrationOfSytemsFromBarL1WithOneNontrivSym}
If a system~$\bar L_\theta$ from the class~$\bar{\mathcal L}_1$ admits a known Lie-symmetry vector field with nonzero $t$-component,
then its integration is reduced to finding a fundamental matrix
of a~linear system of $n$ first-order ordinary differential equations,
$n+1$ quadratures and algebraic operations,
i.e., the total order of system to be integrated can be efficiently lowered by~$n$.
\end{corollary}

\begin{proof}
The only difference from the proof of Proposition~\ref{pro:SLODEIntergrationOfSytemsFromL1WithOneNontrivSym}
is that we straighten the vector field~$Q$ instead of~$\hat Q$ to~$\p_{\tilde t}$ using
a point transformation~$\Phi$: $\tilde t=T(t)$, $\tilde{\boldsymbol x}=H(t)\boldsymbol x+\boldsymbol h(t)$,
where $\tau T_t=1$, $\tau H_t+H\eta=0$ and $\tau\boldsymbol h_t+H\boldsymbol\chi=0$.
Thus, we in addition need $n$ quadratures to find~$\boldsymbol h$ from the system $\boldsymbol h_t=-H\boldsymbol\chi/\tau$.
\end{proof}

For the above integration procedure to be applicable to systems from the class~$\bar{\mathcal L}_1$,
it suffices to only change the last two steps of this procedure
via additionally constructing $\boldsymbol h$ as a solution of the system $\boldsymbol h_t=-H\boldsymbol\chi/\tau$
in step~5 and modifying the transformation in step~6 to  $\boldsymbol x(t)=H^{-1}(t)\tilde{\boldsymbol x}\big(T(t)\big)-H^{-1}(t)\boldsymbol h(t)$.

Given a system~$L_\vartheta$ from the class~$\mathcal L_1$
that is Lie-invariant with respect to two vector fields with linearly independent $t$-components,
the simultaneous use of these two symmetries for integrating~$L_\vartheta$ is more sophisticated
than the above procedure.

\begin{theorem}\label{thm:SLODEIntergrationOfSytemsFromL1WithTwoNontrivSyms}
Suppose that a system~$L_\vartheta$ from the class~$\mathcal L_1$ admits a Lie invariance algebra
that is spanned by two known Lie-symmetry vector fields with linearly independent $t$-components.
Then the integration of this system is reduced by algebraic operations to constructing
a fundamental matrix of a linear system of $n$ first-order ordinary differential equations
that splits into subsystems not coupled to each other,
and each of these subsystems itself is partially coupled
(see the end of the proof for a more specific description of the block structure of the system).
\end{theorem}

\begin{proof}
The beginning of the proof is similar to that of Proposition~\ref{pro:SLODEIntergrationOfSytemsFromL1WithOneNontrivSym}.
Under the theorem's assumption, Lemma~\ref{lem:MAIofLvartheta} implies that
the system~$L_\vartheta$ is invariant with respect to two vector fields of the form~\eqref{eq:GenElementOfInvAlgOfLvartheta},
$Q_\iota=\tau^\iota\p_t+(\eta^{\iota ab}x^b+\chi^{\iota a})\p_{x^a}$, $\iota=1,2$,
with known coefficients~$\tau^\iota$, $\eta^{\iota ab}$ and $\chi^{\iota a}$ depending at most on~$t$, and
the $t$-components~$\tau^1$ and~$\tau^2$ are linearly independent.
Then this system is also invariant with respect to the vector fields~$Q'_\iota$
with the same values of~$\tau^\iota$ and~$\eta^{\iota ab}$, whereas $\chi^{\iota a}=0$.
In other words, $Q'_\iota\in\mathfrak g^{\rm ess}_\vartheta$.
Consider an open domain~$\mathcal I$ in~$\mathbb F$ such that $\tau^1(t)\ne0$ for any $t\in\mathcal I$.
The algebra $\langle Q'_1,Q'_2\rangle$ is not commutative
since otherwise $\tau^1\tau^2_t-\tau^2\tau^1_t=0$, i.e., the \mbox{$t$-components} $\tau^1$ and~$\tau^2$ are linearly dependent on~$\mathcal I$.
We linearly recombine $Q'_1$ and~$Q'_2$ to have the commutation relation $[Q'_1,Q'_2]=Q'_1$,
and hence $\tau^1\tau^2_t-\tau^2\tau^1_t=\tau^1$, i.e., $(\tau^2/\tau^1)_t=1/\tau^1$ on~$\mathcal I$.

On the domain $\mathcal I\times\mathbb R^n$, we make a point transformation~$\Phi$: $\tilde t=\tau^2(t)/\tau^1(t)$, $\tilde{\boldsymbol x}=H(t)\boldsymbol x$,
where $H$ is an invertible matrix-valued function on~$\mathcal I$ satisfying the matrix differential equation $\tau^1H_t+H\eta^1=0$.
The transformation~$\Phi$ straightens the vector field~$Q'_1$ to~$\tilde Q_1:=\p_{\tilde t}$,
and thus it maps the system~$L_\vartheta$ to a system with constant matrix coefficients,
which is easily integrated using algebraic operations.

At this stage, we have reduced the integration of the system~$L_\vartheta$
to finding a particular invertible solution~$H$ of the matrix differential equation $\tau^1H_t+H\eta^1=0$
and algebraic operations.

The vector field~$Q'_2$ is pushed forward by~$\Phi$ to the vector field
$\tilde Q_2:=\tilde t\p_{\tilde t}+\Lambda^{ab}\tilde x^b\p_{\tilde x^a}$,
where the matrix~$\Lambda=(\Lambda^{ab}):=(\tau^2H_t+H\eta^2)H^{-1}$ is constant in view of the commutation relation \mbox{$[Q'_1,Q'_2]=Q'_1$}
pushed forward by~$\Phi$ to $[\tilde Q_1,\tilde Q_2]=\tilde Q_1$.
Combining the matrix equations \mbox{$\tau^1H_t+H\eta^1=0$} and $\tau^2H_t+H\eta^2=\Lambda H$,
we derive the equality
\[
H\zeta H^{-1}=\Lambda\quad\mbox{with}\quad\zeta:=\eta^2-\frac{\tau^2}{\tau^1}\eta^1.
\]
In other words, the matrix~$\zeta$ depending on~$t$ is similar to the constant matrix~$\Lambda$
with respect to the matrix~$H^{-1}$ also depending on~$t$.
Without loss of generality, we can assume that the matrix~$\Lambda$ has been put in Jordan normal form,
and then $H^{-1}$ is a generalized modal matrix for~$\zeta$.
(If this is not the case, we reduce the matrix~$\Lambda$ to its Jordan normal form by a linear transformation of~$\tilde{\boldsymbol x}$
with a constant invertible matrix~$C$ and merge this matrix with~$H$, $CH\rightsquigarrow H$.)
Since the matrix~$\zeta$ is known, we can find, using only algebraic tools, its Jordan normal form~$\Lambda$
and an invertible matrix~$\hat H$ whose inverse establishes the similarity of~$\zeta$ to~$\Lambda$,
$\hat H\zeta\hat H^{-1}=\Lambda$.

If $\mathbb F=\mathbb R$ and the matrix~$\Lambda$ has nonreal eigenvalues,
we can modify it and, consequently, the matrices~$\hat H^{-1}$, $\zeta$ and~$\eta^2$ in such a way
that the system~$L_\vartheta$ is invariant with respect to the new~$Q'_2$,
all the eigenvalues of the new matrix-valued function~$\eta^2$ are constant and real,
and the equality $\hat H\zeta\hat H^{-1}=\Lambda$ still holds;
cf.\ the discussion for $\mathbb F=\mathbb R$
in the end of the consideration of the case $k=2$ in Section~\ref{sec:DescriptionOfEssLieSymExtensions}.

Denote $\check H:=H\hat H^{-1}$.
Then $H=\check H\hat H$ and $\check H\Lambda\check H^{-1}=\Lambda$, i.e., $[\Lambda,\check H]=0$.
Therefore, $\check H$ is a solution of the Frobenius problem~\cite[Chapter~VIII]{gant1959A} with the matrix~$\Lambda$.
In particular, each generalized eigenspace of~$\Lambda$
or, moreover, the subspace $\mathop{\rm ker}(\Lambda-\lambda E)^l$
for arbitrary $\lambda\in\mathbb F$ and $l\in\mathbb N$
is invariant with respect to~$\check H$.
Since $L_\vartheta\in\mathcal L_1$, the matrix~$\Lambda$ necessarily has distinct eigenvalues;
see again the consideration of the case $k=2$ in Section~\ref{sec:DescriptionOfEssLieSymExtensions}.
This implies that the matrix~$\check H$ has a block-diagonal form
corresponding to the partition $\mathbb F^n=\bigoplus_{i=1}^r\mathscr U_i$ of~$\mathbb F^n$
into the generalized eigenspaces~$\mathscr U_i$ of~$\Lambda$,
and each of the diagonal block of~$H$ is of block-triangular form associated with the flag
\begin{gather}\label{eq:SLODEKernelFlag}
\mathop{\rm ker}(\Lambda-\mu_iE)\subset\mathop{\rm ker}(\Lambda-\mu_iE)^2\subset\dots\subset\mathop{\rm ker}(\Lambda-\mu_iE)^{m_i},
\end{gather}
where $\mu_i$ is the eigenvalue of~$\Lambda$ corresponding to this block, and
$m_i$ is the index of the eigenvalue~$\mu_i$, i.e.,
the maximum rank of generalized eigenvectors of~$\Lambda$ for this eigenvalue.

We represent the transformation~$\Phi$ as the composition $\Phi=\check\Phi\circ\hat\Phi$ of the transformations
$\hat\Phi$:~$\check t=t$, $\check{\boldsymbol x}=\hat H(t)\boldsymbol x$ and
$\check\Phi$: $\tilde t=\tau^2(\check t)/\tau^1(\check t)$, $\tilde{\boldsymbol x}=\check H(\check t)\check{\boldsymbol x}$.
The transformation~$\hat\Phi$ pushes forward the vector fields $Q'_\iota$
to the vector fields $\check Q_\iota=\tau^\iota(\check t)\p_{\check t}+\check\eta^{\iota ab}(\check t)\check x^b\p_{\check x^a}$, $\iota=1,2$, respectively,
where $\check\eta^\iota=(\tau^\iota\hat H_t+\hat H\eta^\iota)\hat H^{-1}$.
From these expressions for $\check\eta^\iota$, we derive that
\[
\check\zeta:=\check\eta^2-\frac{\tau^2}{\tau^1}\check\eta^1=\hat H\zeta\hat H^{-1}=\Lambda.
\]
The commutation relation $[\check Q_2,\check Q_1]=\check Q_1$
obtained via pushing forward the commutation relation $[Q'_1,Q'_2]=Q'_1$ by~$\hat\Phi$
implies that $[\Lambda,\check\eta^1]=0$,
i.e., any subspace $\mathop{\rm ker}(\Lambda-\lambda E)^l$ with $\lambda\in\mathbb F$ and $l\in\mathbb N$
is invariant with respect to~$\check\eta^1$.
Therefore, the matrix~$\check\eta^1$ is of the same block structure as that described above for the matrix~$\check H$.

The matrix equation $\tau^1H_t+H\eta^1=0$ in~$H$ reduces to the matrix equation
\begin{gather}\label{eq:SLODE2NontrivLieSymsSplitSystemForReducingPoinTrans}
\tau^1\check H_t+\check H\check\eta^1=0
\end{gather}
in~$\check H$,
which is natural since it is equivalent to the fact that
the transformation~$\check\Phi$ straightens the vector field~$\check Q_1$ to~$\tilde Q_1:=\p_{\tilde t}$.
The equation~\eqref{eq:SLODE2NontrivLieSymsSplitSystemForReducingPoinTrans}
considered as a system for the entries of~$\check H$
splits into the subsystems associated with the diagonal blocks~$\check H_{ii}$ of~$\check H$,
$\tau^1\check H_{ii,t}+\check H_{ii}\check\eta^1_{ii}=0$,
where the matrix coefficients~$\check\eta^1_{ii}$ are the respective diagonal blocks of~$\check\eta^1$.
Each of these subsystems is of partially coupled structure that is consistent with the flag~\eqref{eq:SLODEKernelFlag}
for the associated eigenvalue of~$\Lambda$.
Therefore, we can find~$\check H$ as the transpose of a fundamental matrix of the system of $n$ first-order ordinary differential equations
$\tau^1\boldsymbol y_t+\check\eta^1\boldsymbol y=\boldsymbol0$.

Let us more thoroughly describe the structure
of the diagonal blocks~$\check\eta^1_{ii}$, $i=1,\dots,r$, of the matrix~$\check\eta^1$,
which is required for understanding how to integrate this system.
Recall that $\sigma(\Lambda)=\{\mu_1,\dots,\mu_r\}$ is the spectrum of the matrix~$\Lambda$.
For each $i\in\{1,\dots,r\}$, by $p_i$ and $n_{i1}$,~\dots,~$n_{ip_i}$
we denote the number of elementary divisors of~$\Lambda$
associated with the eigenvalue~$\mu_i$ and their degrees, respectively.
Hence $n_{i1}+\dots+n_{ip_i}=n_i$, $i=1,\dots,r$, and $n_1+\dots+n_r=n$.
We take a canonical basis for~$\Lambda$ such that
the sequence of the lengths of the basis Jordan chains corresponding to~$\mu_i$
or, equivalently, the sequence of the degrees of related elementary divisors does not ascend.
In other words, there exist $s_i\in\{1,\dots,p_i\}$ and
$q_{i1},\dots,q_{i,s_i-1}\in\mathbb N_0$ with $0=:q_{i0}<q_{i1}<\dots<q_{is_i}:=p_i$
such that
\[
n_{i,q_{i,j-1}+1}=\dots=n_{iq_{ij}}>n_{i,q_{ij}+1},\quad j=1,\dots,s_i, \quad\mbox{where}\quad n_{i,p_i+1}:=0.
\]
Then we reorder the basis elements within the generalized eigenspace~$\mathscr U_i$ in the following way.
Adhering the fixed order of chains, we successively choose
the first vectors in the chains,
the second vectors in the chains of length greater than one,
the third vectors in the chains of length greater than two, etc.
In the modified basis, the submatrix~$\check\eta^1_{ii}$ is block-triangular.
The diagonal blocks of~$\check\eta^1_{ii}$ in the first portion of these blocks
are of sizes $q_{ij}':=q_{ij}-q_{i,j-1}$, $j=1,\dots,s_i$.
Thus, $q_{ij}'$ is the number of basis Jordan chains of the same length~$n_{iq_{ij}}$.
For the $j$th portion of the diagonal blocks of~$\check\eta^1_{ii}$, $j=2,\dots,s_i$,
we repeat the same blocks, just excluding the blocks that correspond to the chains of lengths less than $n_{iq_{ij}}$.

Due to the block structure of the matrix~$\check\eta^1$, the system $\tau^1\boldsymbol y_t+\check\eta^1\boldsymbol y=\boldsymbol0$
splits into $r$ subsystems that are not coupled to each other
and corresponds to the diagonal blocks~$\check\eta^1_{ii}$, $i=1,\dots,r$.
Moreover, each of these subsystems is partially coupled.
To integrate the $i$th subsystem, one should solve $s_i$ homogeneous linear systems
of $q_{i1}'$, \dots, $q_{is_i}'$ first-order ordinary differential equations, respectively,
and then compute $n_i-q_{i1}'$ quadratures for finding particular solutions
of involved inhomogeneous systems, which have the same matrices as the above homogeneous systems.

Finding~$\check H$ completes the construction of the transformation~$\Phi$.
\end{proof}

Rearranging the proof of Theorem~\ref{thm:SLODEIntergrationOfSytemsFromL1WithTwoNontrivSyms},
we can formulate the following step-by-step procedure
for integrating a system~$L_\vartheta$ from the class~$\mathcal L_1$
simultaneously using a special pair of its known Lie-symmetry vector fields,
$Q_\iota=\tau^\iota\p_t+(\eta^{\iota ab}x^b+\chi^{\iota a})\p_{x^a}$, $\iota=1,2$,
where the $t$-components~$\tau^1$ and~$\tau^2$ are linearly independent:

\begin{enumerate}\itemsep=0ex
\item
Replace $Q_\iota$ by their counterparts $Q'_\iota=\tau^\iota\p_t+\eta^{\iota ab}x^b\p_{x^a}$
with the zero values of~$\chi^{\iota a}$.
\item
Linearly recombine $Q'_1$ and~$Q'_2$ to get the commutation relation $[Q'_1,Q'_2]=Q'_1$.
\item
Set $T:=\tau^2/\tau^1$ and $\zeta:=\eta^2-T\eta^1$.
\item
Construct the Jordan normal form~$\Lambda$ of the matrix~$\zeta$
and a generalized modal matrix~$M$ for~$\zeta$ that is associated with~$\Lambda$.
Take $\hat H:=M^{-1}$. Thus, $\hat H\zeta\hat H^{-1}=\Lambda$.
\item
Compute the matrix $\check\eta^1:=(\tau^1\hat H_t+\hat H\eta^1)\hat H^{-1}$, which is block-diagonal.
\item
Find~$\check H$ as the transpose of a (block-diagonal) fundamental matrix of solutions
of the split system of $n$ first-order ordinary differential equations
$\tau^1\boldsymbol y_t+\check\eta^1\boldsymbol y=\boldsymbol0$.
\item
Set $H=\check H\hat H$ and, according to~\eqref{eq:EquivGroupbarLB}--\eqref{eq:EquivGroupbarLC},
compute the matrices~$\tilde A$ and~$\tilde B$, which are necessarily constant.
\item
Solve the system~$L_{\tilde\vartheta}$ with constant matrix coefficients $\tilde\vartheta=(\tilde A,\tilde B)$.
\item
Push forward the found general solution of the system~$L_{\tilde\vartheta}$
to the general solution of the system~$L_\vartheta$
by the transformation $\Phi^{-1}$, $\boldsymbol x(t)=H^{-1}(t)\tilde{\boldsymbol x}\big(T(t)\big)$.
\end{enumerate}

Step 6 is the only step, where one needs to integrate linear systems of ordinary differential equations.
In step 9, it is convenient to compute the inverse~$H^{-1}$ of the matrix~$H$ as~\mbox{$H^{-1}=M\check H^{-1}$}.

If the elementary divisors of the matrix~$\Lambda$ differ to each other,
the structure of the matrix~$\check\eta^1$ is especially simple,
and Theorem~\ref{thm:SLODEIntergrationOfSytemsFromL1WithTwoNontrivSyms} implies
the following assertion.

\begin{corollary}\label{cor:SLODECompleteIntergrationOfSytemsFromL1WithTwoNontrivSyms}
Let a system~$L_\vartheta$ from the class~$\mathcal L_1$ admit a Lie invariance algebra
that is spanned by two known Lie-symmetry vector fields
$Q_\iota=\tau^\iota\p_t+\eta^{\iota ab}x^b\p_{x^a}$, $\iota=1,2$, with linearly independent~$\tau^1$ and~$\tau^2$,
where $[Q_1,Q_2]=Q_1$,
and let the matrix $\zeta:=\eta^2-(\tau^2/\tau^1)\eta^1$ has no coinciding elementary divisors.
Then this system can be completely integrated by algebraic operations and at most $n+p-r$ quadratures,
where
$p$ is the number of elementary divisors of~$\zeta$ and
$r$ is the number of distinct eigenvalues of~$\zeta$.
\end{corollary}

\begin{proof}
The matrix~$\check\eta^1$ is necessarily block-diagonal with $r$ main-diagonal blocks,
where the $i$th block~$\check\eta^1_{ii}$ corresponds to the $i$th eigenvalue~$\mu_i$ of~$\Lambda$
(or, equivalently, of~$\zeta$), $i=1,\dots,r$.
Since the elementary divisors of~$\Lambda$ for the eigenvalue~$\mu_i$ are different,
we can reorder the canonical basis for~$\Lambda$ in such a way
that the submatrix~$\check\eta^1_{ii}$ is triangular
and the sequence of its diagonal elements splits into the subsequences
being in one-to-one correspondence with the elementary divisors of~$\Lambda$ for~$\mu_i$.
Each of these subsequences is constituted by identical entries
and its length coincides with the degree of the associated elementary divisor.
For the system $\tau\boldsymbol y_t+\check\eta^1\boldsymbol y=\boldsymbol0$,
this means that it necessarily splits into $r$ subsystems that are not coupled to each other
and the $i$th subsystem is integrated at most $n_i+p_i-1$ quadratures,
where $n_i$ is the algebraic multiplicity of the eigenvalue~$\mu_i$ of~$\Lambda$,
and $p_i$ is the number of the elementary divisors of~$\Lambda$ for~$\mu_i$.
\end{proof}

\begin{remark}\label{rem:SLODEWeiNormanApproach}
Under certain additional restrictions, the integration of auxiliary linear systems of first-order differential equations
that arise in the assertions of this section
or in the course of mapping systems from the class~$\mathcal L_1$ to ones from the class~$\mathcal L'_1$
also become algorithmic, e.g., within the framework of the Wei--Norman approach~\cite{wei1963a,wei1964a}.
More specifically, let a matrix-valued function~$F$ of~$t$ be represented in the form
$F(t)=\sum_{l=1}^m\varphi^l(t)K_l$, where $\varphi^1$, \dots, $\varphi^m$ are scalar functions of~$t$,
and $K_1$, \dots, $K_m$ are linearly independent constant matrices spanning a Lie algebra~$\mathfrak f$.
Then the solution of the matrix Cauchy problem $H_t=F(t)H$, $H(0)=E$
admits the representation $H(t)=\prod_{l=1}^m\exp(\psi^l(t)K_l)$.
The tuple $(\psi^1,\dots,\psi^m)$ is the solution of the Cauchy problem
$\psi^l_t=\sum_{l'=1}^m\varrho^{ll'}(\psi^1,\dots,\psi^m)\varphi^{l'}(t)$, $\psi^l(0)=0$, $l=1,\dots,m$.
Here the coefficients~$\varrho^{ll'}$ are analytic functions of~$(\psi^1,\dots,\psi^m)$
that are defined only by the coefficients $\varphi^1$, \dots, $\varphi^m$
and the structure constants of~$\mathfrak f$ in the basis $(K_1,\dots,K_m)$.
Thus, the system of ordinary differential equations for~$(\psi^1,\dots,\psi^m)$ is nonlinear in general.
At the same time, if $\mathbb F=\mathbb C$ and the algebra~$\mathfrak f$ is solvable,
then the basis $(K_1,\dots,K_m)$ can be chosen in such a way that $\varrho^{ll'}=0$, $l'>l$,
and $\varrho^{ll'}$ depends only on $\psi^{l''}$ with $l''<l$.
This means that the system for~$\psi^1$, \dots, $\psi^m$ can be integrated by quadratures for such~$\mathfrak f$.
As a result, the solution of the matrix equation $H_t=F(t)H$ associated with the solvable Lie algebra~$\mathfrak f$
reduces to computing exponents of constant matrices.
The case of abelian~$\mathfrak f$ is particularly simple;
then $\psi^l(t)=\int_0^t\varphi^l(t')\,{\rm d}t'$, $l=1,\dots,m$.
\end{remark}

\section{Systems of two equations}\label{sec:SLODEn2}

As an illustrative example of application of the theory developed,
we carry out the complete group classification of normal linear systems of two second-order ordinary differential equations, i.e., $n=2$.
For this value of~$n$, systems~$L'_V$ from the class $\mathcal L'_1$ have the form
\[
\boldsymbol x_{tt} =V(t)\boldsymbol x, \quad V(t) = \begin{pmatrix}
 V^{11}(t) & V^{12}(t) \\
 V^{21}(t) & V^{22}(t)
 \end{pmatrix},
\]
where $\boldsymbol x=(x^1,x^2)$, and the arbitrary-element tuple~$V$ runs through the set of $2\times 2$ matrix-valued functions of~$t$
that are not proportional to the $2\times 2$ identity matrix~$E$ with time-dependent proportionality factors.

Recall that for any system~$L'_V$ from the class~$\mathcal L'_1$,
its maximal Lie invariance algebra~$\mathfrak g_V^{}$ is the semidirect sum
$\mathfrak g_V^{}=\mathfrak g^{\rm ess}_V\lsemioplus\mathfrak g^{\rm lin}_V$,
where
\begin{gather*}
\mathfrak g^{\rm lin}_V:=\big\{\chi^a(t)\p_{x^a}\mid\boldsymbol\chi=(\chi^1,\chi^2)^{\mathsf T}\ \mbox{is a solution of}\ L'_V\big\},\\
\mathfrak g^{\rm ess}_V:=\big\{\tau(t)\p_t+\eta^{ab}(t)x^b\p_{x^a}\mid(\tau,\eta)\ \mbox{is a solution of}\ \eqref{eq:ClassifyingCondL'_V}\big\}
\end{gather*}
are respectively the four-dimensional abelian ideal, which is associated with the linear superposition of solutions,
and its complementary subalgebra called the essential Lie invariance algebra of system~$L'_V$,
cf.\ the end of Section~\ref{sec:SLODEPreliminaryAnalysisOfLieSyms} and the beginning of Section~\ref{sec:DescriptionOfEssLieSymExtensions}.
(In the above formulas and in what follows the indices~$a$ and~$b$ run from~1 to~2.)
This leads to the conclusion that it is more natural to classify the essential Lie invariance algebras of the systems from the class~$\mathcal L'_1$
than the maximal Lie invariance algebras of these systems.

The essential Lie invariance algebra~$\mathfrak g^{\rm ess}_V$ of any system~$L'_V$ from the class $\mathcal L'_1$
contains the subalgebra $\mathfrak s^{\rm vf}:=\{Q_\Gamma\mid\Gamma\in\mathfrak s\}$,
where $Q_\Gamma:=\Gamma^{ab}x^b\p_{x^a}$, and $\mathfrak s$ is a subalgebra of $\mathfrak{sl}(2,\mathbb F)$
whose double centralizer coincides with it,
${\rm C}_{\mathfrak{sl}(2,\mathbb F)}({\rm C}_{\mathfrak{sl}(2,\mathbb F)}(\mathfrak s))=\mathfrak s$.
We choose the basis in $\mathfrak{sl}(2,\mathbb F)$ that consists of the matrices
\begin{gather*}
S_1=\begin{pmatrix}
 0 & 1 \\
 0 & 0 \end{pmatrix}, \quad
S_2= \begin{pmatrix}
 1 & \hphantom{-}0 \\
 0 & -1
 \end{pmatrix}, \quad
S_3=\begin{pmatrix}
 \hphantom{-}0 & 0 \\
 -1 & 0
 \end{pmatrix}.
 \end{gather*}
Thus, $[S_1,S_2]=-2S_1$, $[S_2,S_3]=-2S_3$, $[S_1,S_3]=-S_2$.
In the case $\mathbb F =\mathbb C$,
a complete list of ${\rm SL}(2,\mathbb F)$-inequivalent subalgebras~$\mathfrak s$ of $\mathfrak{sl}(2,\mathbb F)$
is exhausted by the subalgebras
$\{0 \}$, $\langle S_1 \rangle$, $\langle S_2 \rangle$, $\langle S_1,S_2\rangle$ and $\mathfrak{sl}(2,\mathbb F)$ itself,
whereas for $\mathbb F =\mathbb R$, the list should be extended by the one more subalgebra $\langle S_1+S_3 \rangle$.
The double centralizers ${\rm C}_{\mathfrak{sl}(2,\mathbb F)}({\rm C}_{\mathfrak{sl}(2,\mathbb F)}(\mathfrak s))$ of these subalgebras
are respectively $\{0 \}$, $\langle S_1 \rangle$, $\langle S_2 \rangle$, $\mathfrak{sl}(2,\mathbb F)$, $\mathfrak{sl}(2,\mathbb F)$ and $\langle S_1+S_3 \rangle$.
Since ${\rm C}_{\mathfrak{sl}(2,\mathbb F)}({\rm C}_{\mathfrak{sl}(2,\mathbb F)}(\mathfrak s))\ne \mathfrak s$ if and only if
$\mathfrak s = \langle S_1,S_2\rangle$, this is the only proper subalgebra of $\mathfrak{sl}(2,\mathbb F)$
that is not appropriate for using in the course of group classification of the class~$\mathcal L'_1$ with $n=2$.
The subalgebra~$\mathfrak s=\mathfrak{sl}(2,\mathbb F)$ is also not relevant for this classification in view of the fact
that any system possessing the corresponding algebra~$\mathfrak s^{\rm vf}$ belongs to the class~$\mathcal L''_0$.

\begin{theorem}\label{thm:SLODEGroupClassification(n=2)}
For $\mathbb F =\mathbb C$, a complete list of \smash{$G^\sim_{\mathcal L'}$}-inequivalent essential Lie-symmetry extensions
in the class~$\mathcal L'_1$ with $n=2$ is exhausted by the following cases:
{\rm
\begin{enumerate}\itemsep=0.5ex\setcounter{enumi}{-1}
\item\label{case:SLODEn20}
General case $V(t)$: \
$\mathfrak g^{\rm ess}_V=\langle I\rangle$;
\item\label{case:SLODEn2(k=0)1}
$V=v(t)S_1$: \
$\mathfrak g^{\rm ess}_V=\langle I,\,x^2\p_{x^1}\rangle$;
\item\label{case:SLODEn2(k=0)2}
$V=v(t)S_2$: \
$\mathfrak g^{\rm ess}_V=\langle I,\,x^1\p_{x^1}-x^2\p_{x^2}\rangle$;
\item\label{case:SLODEn2(k=1)1}
$V=\varepsilon E+(\beta_1-2\beta_2t+\beta_3t^2)S_1+(\beta_2-\beta_3 t)S_2+\beta_3S_3$, $(\beta_2,\beta_3)\ne(0,0)$: \
$\mathfrak g^{\rm ess}_V=\langle I,\,\p_t+x^2\p_{x^1}\rangle$;
\item\label{case:SLODEn2(k=1)2}
$V=\varepsilon E+\beta_1{\rm e}^{2t}S_1+\beta_2S_2+\beta_3{\rm e}^{-2t}S_3$, \ $(\beta_1\beta_2,\beta_2\beta_3,\beta_3\beta_1)\ne(0,0,0)$:\\
$\mathfrak g^{\rm ess}_V=\langle I,\,\p_t+x^1\p_{x^1}-x^2\p_{x^2}\rangle$;
\item\label{case:SLODEn2(k=1)3}
$V=\varepsilon E+{\rm e}^{2\gamma t}S_1$, \ $4\varepsilon\ne\gamma^2$: \
$\mathfrak g^{\rm ess}_V=\langle I,\,x^2\p_{x^1},\,\p_t+\gamma(x^1\p_{x^1}-x^2\p_{x^2})\rangle$;
\item\label{case:SLODEn2(k=1)4}
$V=\varepsilon E+S_2$: \
$\mathfrak g^{\rm ess}_V=\langle I,\,x^1\p_{x^1}-x^2\p_{x^2},\,\p_t\rangle$;
\item\label{case:SLODEn2(k=2)}
$V=S_1$: \
$\mathfrak g^{\rm ess}_V=\langle I,\,x^2\p_{x^1},\,\p_t,\,t\p_t+2x^1\p_{x^1}\rangle$.
\end{enumerate}}
If $\mathbb F=\mathbb R$, then this list is supplemented with three more cases,
{\rm
\begin{enumerate}\itemsep=0.5ex\renewcommand{\theenumi}{\arabic{enumi}$^{\mathbb R}$}
\item\label{case:SLODEn2(k=0)1R}
$V=v(t)(S_1+S_3)$: \
$\mathfrak g^{\rm ess}_V=\langle I,\,x^1\p_{x^2}-x^2\p_{x^1}\rangle$;
\setcounter{enumi}{3}
\item\label{case:SLODEn2(k=1)2R}
$V=\varepsilon E+\mu(S_1+S_3)+\nu\cos(2t)(S_1-S_3)+\nu\sin(2t)S_2$, \ $\nu\ne0$: \
$\mathfrak g^{\rm ess}_V=\langle I,\,\p_t+x^2\p_{x^1}-x^1\p_{x^2}\rangle$;
\setcounter{enumi}{5}
\item\label{case:SLODEn2(k=1)4R}
$V=\varepsilon E+S_1+S_3$: \
$\mathfrak g^{\rm ess}_V=\langle I,\,x^1\p_{x^2}-x^2\p_{x^1},\,\p_t\rangle$.
\end{enumerate}}

\noindent
Here $\varepsilon,\gamma,\mu,\nu,\beta_1,\beta_2,\beta_3\in\mathbb F$, $I:=x^1\p_{x^1}+x^2\p_{x^2}$,
and $v$ runs through the set of functions of~$t$ with
nonzero Wronskian of $(t^2v)_t$, $tv_t$, $v_t$ and~$v$.
Modulo the \smash{$G^\sim_{\mathcal L'}$}-equivalence,
$V$ is a nonzero traceless matrix-valued function of~$t$ in Case~\ref{case:SLODEn20},
$\beta_2=0$ if $\beta_3\ne0$ or $\beta_1=0$ if $\beta_3=0$ and $\beta_2\ne0$ in Case~\ref{case:SLODEn2(k=1)1},
one of the nonzero $\beta_1$ or $\beta_3$ is equal to~$1$ in Case~\ref{case:SLODEn2(k=1)2},
$\gamma\in\{0,1\}$ in Case~\ref{case:SLODEn2(k=1)3},
and $\nu>0$ in  Case~\ref{case:SLODEn2(k=1)2R}.
\end{theorem}

\begin{proof}
As discussed in the beginning of Section~\ref{sec:DescriptionOfEssLieSymExtensions},
the classifying condition~\eqref{eq:ClassifyingCondL'_V} implies
that $I\in\mathfrak g^{\rm ess}_V$ for any system~$L'_V$ from the class~$\mathcal L'_1$,
and $\mathfrak g^{\rm ess}_V=\langle I\rangle$ for general systems in this class,
which leads to Case~\ref{case:SLODEn20} of the theorem.

Following the consideration of Section~\ref{sec:DescriptionOfEssLieSymExtensions},
looking for essential Lie-symmetry extensions within the class~$\mathcal L'_1$,
we separately study the cases $k=0$, $k=1$ and~$k=2$.

\medskip\par\noindent$\boldsymbol{k=0.}$
Up to the \smash{$G^\sim_{\mathcal L'}$}-equivalence, it is convenient
to assume from the very beginning that $\mathop{\rm tr}V=0$, i.e., $L'_V\in\mathcal L''_1$.
Here the subalgebra $\mathfrak s=\{0\}$ of $\mathfrak{sl}(2,\mathbb F)$ is not appropriate
since it corresponds to the general case with no Lie-symmetry extension.
The subalgebras $\langle S_1 \rangle$, $\langle S_2 \rangle$ and, if $\mathbb F=\mathbb R$, $\langle S_1+S_3 \rangle$
lead to Cases~\ref{case:SLODEn2(k=0)1}, \ref{case:SLODEn2(k=0)2} and~\ref{case:SLODEn2(k=0)1R} of the theorem,  respectively,
where $V=v(t)\Gamma$ and $\mathfrak g^{\rm ess}_V=\langle I, Q_\Gamma\rangle$
with matrix~$\Gamma$ being the basis element of the corresponding subalgebra~$\mathfrak s$.
In view of Corollary~\ref{cor:MAIofL''_V},
we have further Lie-symmetry extensions if and only if $\tau v_t=(\kappa-2\tau_t)v$
for some constant~$\kappa$ and some function~$\tau$ of~$t$ with $\tau_{ttt}=0$,
i.e., the functions $(t^2v)_t$, $tv_t$, $v_t$ and~$v$ are linearly dependent.
For the absence of extensions, the corresponding Wronskian should be nonvanishing.

\medskip\par\noindent$\boldsymbol{k=1.}$
Then $\mathfrak g^{\rm ess}_V=\langle I\rangle\oplus\big(\langle P\rangle\lsemioplus\mathfrak s^{\rm vf}\big)$,
where $P:=\p_t+\Upsilon^{ab}x^b\p_{x^a}$.
The matrix-valued parameter function~$V$ is of the form~\eqref{eq:k=1RepresentationForV} with $W\ne0$.
We separately consider each of the appropriate elements
of the complete list of ${\rm SL}(2,\mathbb F)$-inequivalent subalgebras of~$\mathfrak{sl}(2,\mathbb F)$,
$\{0 \}$, $\langle S_1 \rangle$, $\langle S_2 \rangle$ and, for $\mathbb F=\mathbb R$, $\langle S_1+S_3 \rangle$,
as a candidate for~$\mathfrak s$.

Setting $\mathfrak s=\{0\}$ leads to no strict preliminary restrictions for~$\Upsilon$ and~$W$.
The only obvious restriction that the matrix~$\Upsilon$ is also nonzero
since otherwise ${\rm C}_{\mathfrak{sl}(2,\mathbb F)}\big(\{K_l,l\in\mathbb N_0\}\big)=\langle W\rangle\ne\{0\}=\mathfrak s$.
Therefore, the consideration of this case reduces to the classification of pairs of nonzero $2\times2$ matrices up to matrix similarity.
We fix each of possible $2\times2$ Jordan normal forms as a value for the parameter matrix~$\Upsilon$,
$\Upsilon=S_1$, $\Upsilon=\gamma S_2$ and, over $\mathbb R$, $\Upsilon=\gamma(S_1+S_3)$, where $\gamma\ne0$
and thus $\gamma=1$ modulo scalings of~$t$.
For each of these~$\Upsilon$, we have \smash{${\rm C}_{\mathfrak{sl}(2,\mathbb F)}(\{\Upsilon\})=\langle\Upsilon\rangle$},
and hence any of the invertible matrices~$M$ that commute with the fixed form of~$\Upsilon$
is proportional to ${\rm e}^{\beta\Upsilon}$ for some $\beta\in\mathbb F$.
According to Lemma~\ref{lem:SLODEEquivOfSystemsSimilarToConstCoeffOnes},
we can transform the matrix~$W$ as $\tilde W={\rm e}^{\beta\Upsilon}W{\rm e}^{-\beta\Upsilon}$,
still preserving the matrix~$\Upsilon$.
Taking the general form for~$W$,
$W=:K_0=\beta_1S_1+\beta_2S_2+\beta_3S_3$ with $(\beta_1,\beta_2,\beta_3)\ne(0,0,0)$,
we consider possible reductions for the matrix~$W$ by the above transformations.

If $\Upsilon = S_1$, then $K_1=-2\beta_2S_1-\beta_3S_2$, $K_2=2 \beta_3S_1$, $K_l=0$, $l\geqslant 3$,
and the centralizer ${\rm C}_{\mathfrak{sl}(2,\mathbb F)}\big(\{K_0,K_1,K_2\}\big)$ has
the required value $\{0\}$ if and only if $(\beta_2,\beta_3)\ne(0,0)$.
Up to the \smash{$G^\sim_{\mathcal L'}$}-equivalence, we can set $\beta_2=0$ if $\beta_3\ne0$ or $\beta_1=0$ if $\beta_3=0$ and $\beta_2\ne0$.
This results in Case~\ref{case:SLODEn2(k=1)1}.

If $\Upsilon=S_2$, then $K_l=[\Upsilon,K_{l-1}]=2^l\beta_1 S_1+(-2)^l\beta_3S_3$, $l\in\mathbb N$, and hence
the centralizer ${\rm C}_{\mathfrak{sl}(2,\mathbb F)}\big(\{K_l,l\in\mathbb N_0\}\big)$ coincides with $\mathfrak s=\{0\}$
if and only if at least two of $\beta_1$, $\beta_2$ and $\beta_3$ are nonzero.
We can set one of the nonzero $\beta_1$ or $\beta_3$ to be equal to~1 up to the shifts of~$t$,
thus obtaining Case~\ref{case:SLODEn2(k=1)2}.

In the case $\mathbb F=\mathbb R$ and $\Upsilon=S_1+S_3$, we can set $\beta_2=0$,
i.e., $W=:K_0=\beta_1S_1+\beta_3S_3$, $K_{2l+1}=(\beta_1-\beta_3)(-4)^lS_2$, $K_{2l+2}=2(\beta_1-\beta_3)(-4)^l(S_3-S_1)$, $l\in\mathbb N_0$,
and thus
${\rm C}_{\mathfrak{sl}(2,\mathbb F)}\big(\{K_l,l\in\mathbb N_0\}\big)=\{0\}=\mathfrak s$
if and only if $\beta_1-\beta_3\ne0$.
Up to the shifts of~$t$ proportional to $\pi/2$, we can assume $\beta_1-\beta_3>0$.
Denoting $\mu:=(\beta_1+\beta_3)/2$ and $\nu:=(\beta_1-\beta_3)/2$ leads to Case~\ref{case:SLODEn2(k=1)2R}.

For the subalgebra $\mathfrak s=\langle S_1 \rangle$,
we choose $\langle S_2\rangle$ as its complementary subspace~$\mathsf s$
in ${\rm N}_{\mathfrak{sl}(n,\mathbb F)}(\mathfrak s)=\langle S_1,S_2 \rangle$
and $\Upsilon\in\langle S_2\rangle$, i.e., $\Upsilon=\gamma S_2$ with $\gamma\in \mathbb F$.
Moreover, $W\in {\rm C}_{\mathfrak{sl}(n,\mathbb F)}(\mathfrak s)=\langle S_1\rangle$ and $W\ne0$,
i.e., $W=\beta S_1$ with $\beta\ne0$, and thus $K_l=\beta(2\gamma)^lS_1$, $l\in \mathbb N_0$.
As result, we have Case~\ref{case:SLODEn2(k=1)3},
where $4\varepsilon\ne\gamma^2$ as well (otherwise $k=2$; see the consideration below).
We can set $\gamma\in\{0,1\}$ and $\beta=1$ due to the scalings and shifts of $t$, respectively.

If $\mathfrak s=\langle S_2 \rangle$, then ${\rm N}_{\mathfrak{sl}(2,\mathbb F)}(\mathfrak s)=\mathfrak s$,
$\mathsf s=\{0\}$, and thus the only appropriate choice for the matrix~$\Upsilon$ is $\Upsilon=0$.
Further, $K_0:=W=\beta S_2$ with $\beta\ne0$ since $W\in{\rm C}_{\mathfrak{sl}(2,\mathbb F)}(\mathfrak s)=\langle S_2 \rangle$ and $W\ne0$,
$K_l=0$ for $l\geqslant 1$, and ${\rm C}_{\mathfrak{sl}(2,\mathbb F)}(\{K_0\})=\mathfrak s$.
We can set $\beta=1$ due to scalings of $t$ and the permutation of $x^1$ and $x^2$,
which gives Case~\ref{case:SLODEn2(k=1)4}.

For $\mathbb F=\mathbb R$ and the subalgebra $\mathfrak s=\langle S_1+S_3 \rangle$,
the normalizer ${\rm N}_{\mathfrak{sl}(2,\mathbb F)}(\mathfrak s)$ again coincides with~$\mathfrak s$.
Therefore, $\mathsf s=\{0\}$, $\Upsilon =0$, $W =:K_0=\beta( S_1+S_3)$ with $\beta\ne0$, and $K_l=0$, $l\geqslant1$.
After setting again $\beta=1$ due to scalings of $t$ and the permutation of $x^1$ and $x^2$,
we have Case~\ref{case:SLODEn2(k=1)4R}.

\medskip\par\noindent$\boldsymbol{k=2.}$
Since $W\ne0$, the matrix~$\Lambda$ can have only one chain of two eigenvalues whose difference is equal to two.
Moreover, only the difference of eigenvalues is essential, and the matrix~$\Lambda$ can be assumed to be diagonalizable,
see the respective case in Section~\ref{sec:DescriptionOfEssLieSymExtensions}.
Therefore, from the very beginning, we can take the matrix~$\Lambda$ in the form $\Lambda=\mathop{\rm diag}(2, 0)$.
Then $\Upsilon=0$, $W=:K_0=S_1$ modulo the ${\rm SL}(2,\mathbb F)$-equivalence, $K_l=0$, $l\in\mathbb N$, and
${\rm C}_{\mathfrak{sl}(2,\mathbb F)}(\{K_0\})=\langle S_1\rangle=\mathfrak s$,
which gives Case~\ref{case:SLODEn2(k=2)}.
The improper $t$-shift-invariant version of this case is
$V=\frac14E+{\rm e}^{2t}S_1$ with $\mathfrak g^{\rm ess}_V=\langle I,\,x^2\p_{x^1},\,{\rm e}^{-t}\p_t,\,\p_t+2x^1\p_{x^1}\rangle$.
\end{proof}

\begin{corollary}
Let $n=2$.

(i) $\dim\mathfrak g_V^{}\in\{5,6,7,8\}$ for any system~$L'_V$ from the class~$\mathcal L'_1$.

(ii) $\dim\mathfrak g_\theta^{}\in\{5,6,7,8\}$ for any system~$\bar L_\theta$ from the class~$\bar{\mathcal L}_1$.

(iii) $\dim\mathfrak g_\theta^{}\in\{5,6,7,8,15\}$ for any system~$\bar L_\theta$ from the class~$\bar{\mathcal L}$.

(iv) Any system~$\bar L_\theta\in\bar{\mathcal L}$ with $\dim\mathfrak g_\theta^{}>8$
is \smash{$G^\sim_{\bar{\mathcal L}}$}-equivalent to the elementary system $\boldsymbol x_{tt}=\boldsymbol0$.

(v) Any system~$\bar L_\theta\in\bar{\mathcal L}$ with $\dim\mathfrak g_\theta^{}\geqslant7$
is \smash{$G^\sim_{\bar{\mathcal L}}$}-equivalent to a homogeneous system with constant matrix coefficients.
\end{corollary}

Theorem~\ref{thm:SLODEGroupClassification(n=2)} enhances all the results
on Lie symmetries of normal linear systems of two second-order ordinary differential equations
that have been presented in the literature; cf.~\cite{gorr1988a,mkhi2015a,moyo2013a,wafo2000b}.
We carry out the classifications of such symmetries for both the basic fields~$\mathbb C$ and~$\mathbb R$,
accurately distinguishing them.
Cases~\ref{case:SLODEn2(k=0)1R}, \ref{case:SLODEn2(k=1)2R} and~\ref{case:SLODEn2(k=1)4R}
of Theorem~\ref{thm:SLODEGroupClassification(n=2)}, which are specific for $\mathbb F=\mathbb R$,
are equivalent to Cases~\ref{case:SLODEn2(k=0)1}, \ref{case:SLODEn2(k=1)2} and~\ref{case:SLODEn2(k=1)4}
over the complex field, respectively.
Moreover, considering the class~$\mathcal L'_1$ instead of its subclass~$\mathcal L''_1$
or, equivalently, allowing the arbitrary-element matrix~$V$ not to be traceless
significantly simplifies both the classification procedure and the constructed classification list
via reducing both the number of classification cases and their complexity.
See Remark~\ref{rem:SLODEConvenientClassesForGroupClassification} for a more detailed discussion
of this optimization in the case of general~$n$.

Following this remark, we can convert the classification list from Theorem~\ref{thm:SLODEGroupClassification(n=2)}
to one for the class~$\mathcal L''_1$ with $n=2$ by routine computations,
reducing cases with non-traceless~$V$ to their counterparts with traceless~$V$.
It is necessary just to transform all the values of~$V$
in Cases~\ref{case:SLODEn2(k=1)1}--\ref{case:SLODEn2(k=1)4}, \ref{case:SLODEn2(k=1)2R} and~\ref{case:SLODEn2(k=1)4R}
with $\varepsilon\ne0$ to traceless matrices by elements of~\smash{$G^\sim_{\mathcal L'}$}.

In view of the discussion in the beginning of Section~\ref{sec:SLODEPreliminaryAnalysisOfLieSyms},
it is obvious that the list from Theorem~\ref{thm:SLODEGroupClassification(n=2)} can also be interpreted
as a solution of the group classification problems for the wider classes~$\mathcal L_1$ and~$\bar{\mathcal L}_1$.
At the same time, we can modify this classification list within~$\mathcal L_1$ and~$\bar{\mathcal L}_1$ in the way
discussed in Remark~\ref{rem:SLODEConstCoeffsRepresentatives}, which leads to the following assertion.

\begin{corollary}\label{cor:SLODEGroupClassification(n=2)L}
For $\mathbb F =\mathbb C$, a complete list of \smash{$G^\sim_{\mathcal L}$}-inequivalent essential Lie-symmetry extensions
in the class~$\mathcal L_1$ with $n=2$ is exhausted by the following cases:
{\rm
\begin{enumerate}\itemsep=0.5ex\setcounter{enumi}{-1}
\item\label{case:SLODEn2L0}
General case, $A=0$, \ $B=V(t)$: \
$\mathfrak g^{\rm ess}_\vartheta=\langle I\rangle$;
\item\label{case:SLODEn2L(k=0)1}
$A=0$, \ $B=v(t)S_1$: \
$\mathfrak g^{\rm ess}_\vartheta=\langle I,\,x^2\p_{x^1}\rangle$;
\item\label{case:SLODEn2L(k=0)2}
$A=0$, \ $B=v(t)S_2$: \
$\mathfrak g^{\rm ess}_\vartheta=\langle I,\,x^1\p_{x^1}-x^2\p_{x^2}\rangle$;
\item\label{case:SLODEn2L(k=1)1}
$A=-2S_1$, \ $B=\varepsilon E+\beta_1S_1+\beta_2S_2+\beta_3S_3$, \ $(\beta_2,\beta_3)\ne(0,0)$: \
$\mathfrak g^{\rm ess}_\vartheta=\langle I,\,\p_t\rangle$;
\item\label{case:SLODEn2L(k=1)2}
$A=-2S_2$, \ $B=\varepsilon E+\beta_1S_1+\beta_2S_2+\beta_3S_3$, \ $(\beta_1\beta_2,\beta_2\beta_3,\beta_3\beta_1)\ne(0,0,0)$: \
$\mathfrak g^{\rm ess}_\vartheta=\langle I,\,\p_t\rangle$;
\item\label{case:SLODEn2L(k=1)3}
$A=-2\gamma S_2$, \ $B=(\varepsilon-\gamma^2)E+S_1$, \ $4\varepsilon\ne\gamma^2$: \
$\mathfrak g^{\rm ess}_\vartheta=\langle I,\,{\rm e}^{-2\gamma t}x^2\p_{x^1},\,\p_t\rangle$;
\item\label{case:SLODEn2L(k=1)4}
$A=0$, \ $B=\varepsilon E+S_2$: \
$\mathfrak g^{\rm ess}_\vartheta=\langle I,\,x^1\p_{x^1}-x^2\p_{x^2},\,\p_t\rangle$;
\item\label{case:SLODEn2L(k=2)}
$A=0$, \ $B=S_1$: \
$\mathfrak g^{\rm ess}_\vartheta=\langle I,\,x^2\p_{x^1},\,\p_t,\,t\p_t+2x^1\p_{x^1}\rangle$.
\end{enumerate}}
If $\mathbb F=\mathbb R$, then this list is supplemented with three more cases,
{\rm
\begin{enumerate}\itemsep=0.5ex\renewcommand{\theenumi}{\arabic{enumi}$^{\mathbb R}$}
\item\label{case:SLODEn2L(k=0)1R}
$A=0$, \ $B=v(t)(S_1+S_3)$: \
$\mathfrak g^{\rm ess}_\vartheta=\langle I,\,x^1\p_{x^2}-x^2\p_{x^1}\rangle$;
\setcounter{enumi}{3}
\item\label{case:SLODEn2L(k=1)2R}
$A=-2(S_1+S_3)$, \ $B=\varepsilon E+\beta_1S_1+\beta_3S_3$, \ $\beta_1\ne\beta_3$: \
$\mathfrak g^{\rm ess}_\vartheta=\langle I,\,\p_t\rangle$;
\setcounter{enumi}{5}
\item\label{case:SLODEn2L(k=1)4R}
$A=0$, \ $B=\varepsilon E+S_1+S_3$: \
$\mathfrak g^{\rm ess}_\vartheta=\langle I,\,x^1\p_{x^2}-x^2\p_{x^1},\,\p_t\rangle$.
\end{enumerate}}

\noindent
Here $\varepsilon,\gamma,\beta_1,\beta_2,\beta_3\in\mathbb F$, $I:=x^1\p_{x^1}+x^2\p_{x^2}$,
and $v$ runs through the set of functions of~$t$ with
nonzero Wronskian of $(t^2v)_t$, $tv_t$, $v_t$ and~$v$.
Modulo the \smash{$G^\sim_{\mathcal L}$}-equivalence,
$V$ is a nonzero traceless matrix-valued function of~$t$ in Case~\ref{case:SLODEn2L0},
$\beta_2=0$ if $\beta_3\ne0$ or $\beta_1=0$ if $\beta_3=0$ and $\beta_2\ne0$ in Case~\ref{case:SLODEn2L(k=1)1},
one of the nonzero $\beta_1$ or $\beta_3$ is equal to~$1$ in Case~\ref{case:SLODEn2L(k=1)2},
$\gamma\in\{0,1\}$ in Case~\ref{case:SLODEn2L(k=1)3},
and $\beta_1>\beta_3$ in  Case~\ref{case:SLODEn2L(k=1)2R}.
\end{corollary}

Consider the classes~$\bar{\mathcal L}$, $\mathcal L$, $\mathcal L'$ and~$\mathcal L''$ with $n=2$.
Theorem~\ref{thm:SLODEGroupClassification(n=2)} presents a complete list of
inequivalent regular essential Lie-symmetry extensions within the first three of these classes;
see Definition~\ref{def:RegularAndSingularEssLieSymExtensions}.
The modification of this list via converting all the presented values of~$V$ to the equivalent traceless values
according to Remark~\ref{rem:SLODEConvenientClassesForGroupClassification}
gives a complete list of such extensions within the class~$\mathcal L''$ with $n=2$.
For the classes~$\bar{\mathcal L}$ and $\mathcal L$ with $n=2$,
the list from Corollary~\ref{cor:SLODEGroupClassification(n=2)L} does the same.
To obtain the solutions of the complete group classification problems for the above classes,
it suffices to supplement a relevant complete list of inequivalent regular essential Lie-symmetry extensions
with the single singular classification case canonically given by the elementary system~$\boldsymbol x_{tt}=\boldsymbol0$.
The maximal Lie invariance algebra~$\mathfrak g_0$ of this system is well known
and is isomorphic to $\mathfrak{sl}(4,\mathbb F)$, see~\eqref{eq:SLODEMIAOfElementarySystem}.
(The notion of essential Lie invariance algebra is not relevant for this system at all.)

\section{On generalization to arbitrary order}\label{sec:OnGeneralizationToArbitraryEqOrder} 

For a fixed value of $(r,n)$ with $r\in\{2,3,\dots\}$ and $n\in\mathbb N$, we denote by~$\bar{\mathcal L}_{r,n}$
the class of normal linear systems of $n$ ordinary differential equations of the same order~$r$,
\begin{gather}\label{eq:LSOfrthOrderODEs}
\frac{{\rm d}^r\boldsymbol x}{{\rm d}t^r}
=A_{r-1}(t)\frac{{\rm d}^{r-1}\boldsymbol x}{{\rm d}t^{r-1}}+\dots
+A_1(t)\frac{{\rm d}\boldsymbol x}{{\rm d}t}+A_0(t)\boldsymbol x+\boldsymbol f(t),
\end{gather}
where the tuple $\theta=(A_0,\dots,A_{r-1},\boldsymbol f)$ of arbitrary elements
consists of arbitrary (sufficiently smooth) $n\times n$ matrix-valued functions~$A_0$, \dots, $A_{r-1}$ of~$t$
and an arbitrary (sufficiently smooth) vector-valued function~$\boldsymbol f$ of~$t$.
The value $r=2$ is singular for the family of the classes~$\bar{\mathcal L}_{r,n}$,
which is related to the singularity of the elementary system
${\rm d}^r\boldsymbol x/{\rm d}t^r=\boldsymbol0$ with $r=2$ for an arbitrary $n\in\mathbb N$.
The solution of group classification problems for the classes~$\bar{\mathcal L}_{r,1}$
is well known \cite{krau1991a,maho1990a}.
The extended group analysis of these classes,
including the construction of the associated equivalence groupoids and equivalence groups
and the complete group classification using the algebraic method, was carried out in~\cite{boyk2015a}.
In the present paper, we have realized a similar study
for the classes $\bar{\mathcal L}_{2,n}$ with $n\geqslant2$.
The analysis of the above results shows that they can be extended to an arbitrary value of~$(r,n)$.
In this context, let us list problems to be considered and expected results.

The class~$\bar{\mathcal L}_{r,n}$ with fixed $r\geqslant3$ and $n\geqslant2$ is normalized.
Its equivalence group consists of the point transformations in the space
\smash{$\mathbb F_t\times\mathbb F^n_{\boldsymbol x}\times\mathbb F^{rn^2+n}_\theta$}
whose $(t,\boldsymbol x)$-components are of the form~\eqref{eq:EquivGroupbarLA}
and  uniquely define the corresponding $\theta$-components.
Wide families of admissible transformations from the action groupoid of the equivalence group can be used for
successively imposing the gauges $\boldsymbol f=\boldsymbol0$, $A^{r-1}=0$ and $\mathop{\rm tr}A^{r-2}=0$.
After each of the first two gauges, it is convenient to reparameterize the singled out subclasses
via excluding~$\boldsymbol f$ and then~$A^{r-1}$ from the tuple of arbitrary elements for the corresponding classes.
This results in the chain of classes
\[\bar{\mathcal L}_{r,n}\hookleftarrow\mathcal L_{r,n}\hookleftarrow\mathcal L'_{r,n}\supset\mathcal L''_{r,n}.\]
In particular, the elements of the class~$\mathcal L''_{r,n}$ satisfy all the three gauges,
i.e., they are of the Laguerre--Forsyth form in the terminology of~\cite[p.~266]{seas1993a}.
The classes~$\mathcal L_{r,n}$, $\mathcal L'_{r,n}$ and~$\mathcal L''_{r,n}$
are uniformly semi-normalized with respect to the linear superposition of solutions
whereas the ``inhomogeneous'' counterparts $\bar{\mathcal L}'_{r,n}$ and~$\bar{\mathcal L}''_{r,n}$
of $\mathcal L'_{r,n}$ and~$\mathcal L''_{r,n}$ are still normalized.
The $(t,\boldsymbol x)$-components of equivalence transformations
within the classes~$\mathcal L_{r,n}$, $\mathcal L'_{r,n}$ and~$\mathcal L''_{r,n}$
are of the form~\eqref{eq:EquivGroupbarLA} with $\boldsymbol h=\boldsymbol0$,
\eqref{eq:EquivGroupL'A} and~\eqref{eq:EquivGroupL''A}, respectively,
and the components for the corresponding arbitrary elements can be easily computed
whereas the $(t,\boldsymbol x)$-components are known.\looseness=-1

For any system of the form~\eqref{eq:LSOfrthOrderODEs},
its essential Lie invariance algebra is well defined as a complement subalgebra,
in its maximal Lie invariance algebra,
of the $rn$-dimensional abelian ideal associated with the linear superposition of solutions.
The group classifications of the above classes with fixed $(r,n)$ are equivalent to each other,
and it suffices to classify only the essential Lie invariance algebras of systems from these classes.
Although the significant equivalence group or, equivalently,
the significant equivalence algebra of the class~$\mathcal L''_{r,n}$
is finite-dimensional, in fact the class~$\mathcal L'_{r,n}$ is the most convenient
for carrying out the group classification and presenting the classification results.
The tuple of arbitrary elements of the class~$\mathcal L'_{r,n}$ is $\vartheta'=(A_0,\dots,A_{r-2})$.
Denote by $\mathfrak g_{\vartheta'}^{}$, $\mathfrak g^{\rm ess}_{\vartheta'}$ and $\mathfrak g^{\rm lin}_{\vartheta'}$
the maximal Lie invariance algebra of a system $L'_{\vartheta'}\in\mathcal L'_{r,n}$,
its essential Lie invariance algebra and the abelian ideal of $\mathfrak g_{\vartheta'}^{}$
associated with the linear superposition of solutions, respectively.
The algebra~$\mathfrak g_{\vartheta'}^{}$ consists of the vector fields of the form
\begin{equation*}
Q=\tau\p_t+\left(\tfrac12(r-1)\tau_tx^a+\Gamma^{ab}x^b+\chi^a\right)\p_{x^a},
\end{equation*}
where the vector-valued function $\boldsymbol\chi=(\chi^1,\dots,\chi^n)^{\mathsf T}$ of~$t$ is an arbitrary solution of the system~$L'_{\vartheta'}$,
whereas $\tau$ is an arbitrary function of~$t$
and $\Gamma=(\Gamma^{ab})$ is an arbitrary constant $n\times n$ matrix that satisfy the classifying condition
\begin{gather}\label{eq:ClassifyingCondL'_vartheta'}
\begin{split}
\tau\frac{{\rm d}A_s}{{\rm d}t}={}&[\Gamma,A_s]-(r-s)\tau_tA_s-\sum_{p=s+1}^r 
\left(\frac{r-1}2-\frac s{p-s+1}\right)\binom ps\tau_{(p-s+1)}A_p
\,,\\
&s=0,\dots,r-2,\quad A^{r-1}:=0, \quad A^r:=-E.
\end{split}
\end{gather}
We have
\smash{$\mathfrak g_{\vartheta'}^{}=\mathfrak g^{\rm ess}_{\vartheta'}\lsemioplus\mathfrak g^{\rm lin}_{\vartheta'}$},
\smash{$\mathfrak g^{\rm ess}_{\vartheta'}=\mathfrak g_{\vartheta'}^{}\cap\varpi_*\mathfrak g^\sim_{\mathcal L'_{r,n}}$},
\smash{$k=k_{\vartheta'}^{}:=\dim\pi_*\mathfrak g_{\vartheta'}^{}=\dim\pi_*\mathfrak g^{\rm ess}_{\vartheta'}\leqslant3$},
and, moreover, modulo the \smash{$G^\sim_{\mathcal L'_{r,n}}$}-equivalence
\[
\pi_*\mathfrak g_{\vartheta'}^{}=\pi_*\mathfrak g^{\rm ess}_{\vartheta'}\in
\big\{\{0\},\langle\p_t\rangle,\langle\p_t,t\p_t\rangle,\langle\p_t,t\p_t,t^2\p_t\rangle\big\}.
\]
Here
$\varpi$ is the natural projection
from \smash{$\mathbb F_t\times\mathbb F^n_{\boldsymbol x}\times\mathbb F^{(r-1)n^2}_{\vartheta'}$}
onto $\mathbb F_t\times\mathbb F^n_{\boldsymbol x}$,
and $\pi$ is the natural projection from $\mathbb F_t\times\mathbb F^n_{\boldsymbol x}$ onto $\mathbb F_t$.

In the course of describing Lie symmetries of systems from the class~$\mathcal L'_{r,n}$,
each of the potential values of~$k\in\{0,1,2,3\}$ should be considered separately, cf.\ Section~\ref{sec:DescriptionOfEssLieSymExtensions}.
Similarly to the other systems of the form~\eqref{eq:LSOfrthOrderODEs} with $r\geqslant3$,
the point-symmetry transformations of the elementary system ${\rm d}^r\boldsymbol x/{\rm d}t^r=\boldsymbol0$
with $r\geqslant3$ are fiber-preserving and affine with respect to~$\boldsymbol x$.
This is why the orbits of the elementary system ${\rm d}^r\boldsymbol x/{\rm d}t^r=\boldsymbol0$
with respect to the corresponding equivalence groups are not so singular
in the classes~$\bar{\mathcal L}_{r,n}$, $\mathcal L_{r,n}$,
$\bar{\mathcal L}'_{r,n}$, $\mathcal L'_{r,n}$,
$\bar{\mathcal L}''_{r,n}$ and~$\mathcal L''_{r,n}$ for $r\geqslant3$ as for $r=2$.
Hence the systems from the orbit of the elementary system in the class~$\mathcal L'_{r,n}$
can be considered within the above framework as those with $k=3$,
i.e., the value $k=3$ is possible if $r\geqslant3$.
Moreover, we conjecture that a system $L'_{\vartheta'}\in\mathcal L'_{r,n}$ with $r\geqslant3$
is similar to the elementary system ${\rm d}^r\boldsymbol x/{\rm d}t^r=\boldsymbol0$
if and only if $k_{\vartheta'}^{}=3$.
For any system~$L'_{\vartheta'}$ from the class~$\mathcal L'_{r,n}$,
its essential Lie invariance algebra~$\mathfrak g^{\rm ess}_{\vartheta'}$ can be represented as
$
\mathfrak g^{\rm ess}_{\vartheta'}=\mathfrak i\oplus(\mathfrak t_{\vartheta'}\lsemioplus\mathfrak s^{\rm vf}_{\vartheta'})
$,
where
\begin{itemize}\itemsep=0.ex
\item
$\mathfrak i:=\langle x^a\p_{x^a}\rangle$ is an ideal of~$\mathfrak g^{\rm ess}_{\vartheta'}$,
which is common for all systems from the class~$\mathcal L'_{r,n}$,
\item
$\mathfrak s^{\rm vf}_{\vartheta'}:=\{\Gamma^{ab}x^b\p_{x^a}\mid\Gamma\in\mathfrak{sl}(n,\mathbb F)\colon[\Gamma,A_0]=\dots=[\Gamma,A_{r-2}]=0\}$
is an ideal of~$\mathfrak g^{\rm ess}_{\vartheta'}$, and
\item
$\mathfrak t_{\vartheta'}:=\big\langle\tau^\iota\p_t+\big(\tfrac12(r-1)\tau^\iota_tx^a+\Lambda^{ab}_\iota x^b\big)\p_{x^a},\,\iota=1,\dots,k_{\vartheta'}\big\rangle$
is a subalgebra of~$\mathfrak g^{\rm ess}_{\vartheta'}$ with $\dim\mathfrak t_{\vartheta'}=k_{\vartheta'}\in\{0,1,2,3\}$,
the components $\tau^\iota=\tau^\iota(t)$ are linearly independent, and
each of $\tau^\iota$ satisfies~\eqref{eq:ClassifyingCondL'_vartheta'} with $\Gamma=\Lambda_\iota\in\mathfrak{sl}(n,\mathbb F)$.
\end{itemize}
The complete group classification of the class~$\mathcal L'_{r,n}$ for arbitrary values of $(r,n)$
looks as a wild problem since its solution in particular includes the classification, up to the ${\rm SL}(n,\mathbb F)$-equivalence,
of the subalgebras~$\mathfrak s$ of~$\mathfrak{sl}(n,\mathbb F)$
with $\mathfrak s={\rm C}_{\mathfrak{sl}(n,\mathbb F)}({\rm C}_{\mathfrak{sl}(n,\mathbb F)}(\mathfrak s))$.
Although the collection of such subalgebras of~$\mathfrak{sl}(n,\mathbb F)$ seems
essentially smaller than the entire set of subalgebras of~$\mathfrak{sl}(n,\mathbb F)$,
for now we do know an algorithm for the direct classification of subalgebras~$\mathfrak s$
with $\mathfrak s={\rm C}_{\mathfrak{sl}(n,\mathbb F)}({\rm C}_{\mathfrak{sl}(n,\mathbb F)}(\mathfrak s))$
without classifying all the subalgebras of~$\mathfrak{sl}(n,\mathbb F)$.
Even if a direct classification is possible, it may still be a wild problem.
Nevertheless, this group classification is possible for low values of~$r$ and~$n$.

For any system from the class~$\bar{\mathcal L}_{r,n}$ (resp.\ $\mathcal L_{r,n}$, $\mathcal L'_{r,n}$ or $\mathcal L''_{r,n}$),
the dimension of its maximal Lie invariance algebra is greater than or equal to $rn+1$ and less than or equal to $n^2+rn+3$.
The lower bound is greatest and is attained for a general system of the class.
The upper bound is least, and it is attained only for the systems from the orbit
of the elementary system ${\rm d}^r\boldsymbol x/{\rm d}t^r=\boldsymbol0$ with respect to the corresponding equivalence group.
The last claim is a direct consequence of Theorem~\ref{thm:SODEsMaxDim}.
A basis of the maximal Lie invariance algebra of the elementary system consists of the vector fields
\begin{gather*}
t^{s-1}\p_{x^a},\,_{s=1,\dots,r},\quad
x^b\p_{x^a},\quad
\p_t,\quad t\p_t+\tfrac12(r-1)x_c\p_{x_c},\quad t^2\p_t+(r-1)tx_c\p_{x_c}.
\end{gather*}
(The second summand in the penultimate vector fields is inessential
since it can be canceled via linearly combining with the element $x_c\p_{x_c}$ of this algebra;
we preserve it for the last three vector fields to span a Lie algebra, which is isomorphic to~$\mathfrak{sl}(2,\mathbb F)$.)
The submaximum dimension for $\{\mathfrak g_{\theta}^{}\mid\bar L_{\theta}\in\bar{\mathcal L}_{r,n}\}$
is equal to $n^2+rn+1$ \cite{kess2022a} and it is attained, e.g., by the maximal Lie invariance algebra
\begin{gather*}
\big\langle
\varphi^s(t)\p_{x^a},\,_{s=1,\dots,r},\ x^b\p_{x^a},\ \p_t
\big\rangle
\end{gather*}
of any uncoupled system
${\rm d}^r\boldsymbol x/{\rm d}t^r=c_{r-3}{\rm d}^{r-3}\boldsymbol x/{\rm d}t^{r-3}+\dots+c_1{\rm d}\boldsymbol x/{\rm d}t+c_0\boldsymbol x$
of $n$ copies of the same constant-coefficient linear ordinary differential equation
with nonzero coefficient tuple $(c_0,\dots,c_{r-3})$.
Here and in what follows $\{\varphi^s,\,_{s=1,\dots,r}\}$ is a fundamental set of solutions of the equation under consideration.
The next (``subsubmaximum'') dimension in the decreasing order is equal to $n^2+rn$,
and it realizes, e.g., for the maximal Lie invariance algebra
$\big\langle\varphi^s(t)\p_{x^a},\,_{s=1,\dots,r},\ x^b\p_{x^a}\big\rangle$
of any general uncoupled system of $n$ copies of the same variable-coefficient linear ordinary differential equation.
For $n\geqslant3$, the next less dimension may be the dimension
the maximal Lie invariance algebra
\begin{gather*}
\big\langle
t^s\p_{x^a}\!+\delta_{a2}(s\!+\!1)^{-1}\cdots(s\!+\!r)^{-1}t^{s+r}\p_{x^1},\,_{s=0,\dots,r-1},\,
x^b\p_{x^c},\,_{b\ne1,\,c\ne2},\, x^d\p_{x^d},\,
\p_t,\, t\p_t\!+rx^1\p_{x^1}
\big\rangle
\end{gather*}
of the system ${\rm d}^rx^1/{\rm d}t^r=x^2$, ${\rm d}^rx^2/{\rm d}t^r=0$, \dots, ${\rm d}^rx^n/{\rm d}t^r=0$,
which is equal to $n^2+(r-2)n+4$ and coincides with $n^2+rn$ for $n=2$.

An interesting question is how these submaximum dimensions are related to the analogous values
for the entire class~$\mathcal E_{r,n}$ of normal systems of $n$ ordinary differential equations of the same order~$r$.
The related discussion in~\cite{olve1994b} and~\cite[pp.~202--206]{olve1995A} on the class~$\mathcal E_{r,1}$
of single ordinary differential equations of order $r\geqslant3$
that are solvable with respect to the $r$th order derivative
shows that the answer on the above question is not simple even for $n=1$.
It is well known that the maximum dimension of the maximal Lie invariance algebras of equations
from~$\mathcal E_{r,1}$ (resp.\ $\mathcal L_{r,1}$) with $r\geqslant3$ is equal to~$r+4$ and
attained only for the equations that are similar to the elementary equation ${\rm d}^rx/{\rm d}t^r=0$
with respect to point transformations of~$(t,x)$.
Denote by $\mathop{\rm smd}(\mathcal E_{r,n})$ (resp.\ $\mathop{\rm smd}(\mathcal L_{r,n})$)
the submaximum dimension of the maximal Lie invariance algebras of equations
from the class~$\mathcal E_{r,n}$ (resp.\ $\mathcal L_{r,n}$).
In this notation,
\[
\mathop{\rm smd}(\mathcal E_{3,1})=6>\mathop{\rm smd}(\mathcal L_{3,1})=5,\quad
\mathop{\rm smd}(\mathcal E_{5,1})=8>\mathop{\rm smd}(\mathcal L_{5,1})=7,
\]
whereas
$\mathop{\rm smd}(\mathcal E_{r,1})=\mathop{\rm smd}(\mathcal L_{r,1})=r+2$
for all the other values of~$r$, i.e., for $r=4$ or ${r\geqslant6}$.
Nevertheless, there are nonlinearizable equations from~$\mathcal E_{r,1}$
whose maximal Lie invariance algebras are of dimension~$r+2$.
The equality $\mathop{\rm smd}(\mathcal E_{2,n})=\mathop{\rm smd}(\mathcal L_{2,n})+1$
follows from Theorems~\ref{thm:SODEsSubMaxDim} and~\ref{thm:MaxDimOfMIAinL}.
At the same time, Theorem~1.1 and Table~8 from~\cite{kess2022a} imply that
$\mathop{\rm smd}(\mathcal E_{r,n})=\mathop{\rm smd}(\mathcal L_{r,n})$ for $r\geqslant3$.
The maximum known dimension for the maximal Lie invariance algebras
of nonlinearizable normal systems of ordinary differential equations of the same order~$r\geqslant3$ is
$n^2+(r-1)n+3$ if $r=3$ and $n^2+(r-1)n+2$ if $r\geqslant4$ \cite[Table~10]{kess2022a},
which is attained for the system
\[
\frac{{\rm d}^r\boldsymbol x}{{\rm d}t^r}=\frac r{r-1}\frac{{\rm d}^{r-1}x^1}{{\rm d}t^{r-1}}
\left(\frac{{\rm d}^{r-2}x^1}{{\rm d}t^{r-2}}\right)^{-1}\frac{{\rm d}^{r-1}\boldsymbol x}{{\rm d}t^{r-1}}.
\]
The system ${\rm d}^rx^a/{\rm d}t^r=\delta_{a1}({\rm d}^{r-1}x^2/{\rm d}t^{r-1})^2$,
which is also related to a parabolic geometry \cite[Example~4.6]{cap2017a},
realizes the next known dimension of such kind,
$n^2+(r-2)n+5$ if $r=3$ and $n^2+(r-2)n+4$ if $r\geqslant4$ \cite[Table~10]{kess2022a}.

\section{Conclusion}\label{sec:Conclusion}

Although normal linear systems of second-order ordinary differential equations seemed easy
to be studied within the framework of group analysis of differential equations,
to describe their Lie symmetries,
we needed to apply various advanced methods and modern techniques.
The algebraic method of group classification, which is basic for the present paper,
was supplemented with
gauging the arbitrary class elements by wide families of admissible transformations,
splitting the class under consideration into subclasses with better transformations properties,
finding integer values that characterize cases of Lie-symmetry extensions
and are invariant under the action of the corresponding equivalence group, etc.

In the context of the above characterization,
an important tool was also Lie's theorem on realizations of finite-dimensional Lie algebras on the line.
Previously, this theorem had successfully been applied for solving group classification problems
for a number of classes of differential equations, including
multidimensional nonlinear Schr\"odinger equations with modular nonlinearity in \cite[Section~7]{popo2010a} and \cite[Lemma~42]{kuru2020a},
linear ordinary differential equations of arbitrary order $r\geqslant3$ \cite[Section~3]{boyk2015a},
(1+1)-dimensional linear evolution equations of arbitrary order $r\geqslant3$ \cite[Lemma~14]{bihlo2017a},
general Burgers--Korteweg--de Vries equations \cite[Lemma~18]{opan2017a} and
(1+1)-dimensional linear Schr\"odinger equations \cite[Lemma~2]{kuru2018a}.
The usage of Lie's theorem is particularly efficient when classifying Lie symmetries
of (1+1)-dimensional generalized Klein--Gordon equations in the light-cone (characteristic) coordinates
since then Lie's theorem is relevant for both coordinates;
see Corollary~12 and Section~4 in~\cite{boyk2020a}.

Within the classical Lie approach,
the first step in solving the group classification problem
for a class~$\mathcal K$ of (systems of) differential equations
is to construct the equivalence group of this class.
Moreover, usually one merely finds the infinitesimal counterpart of the equivalence group,
which is the equivalence algebra of the class~$\mathcal K$, instead of the entire equivalence group,
thus neglecting discrete equivalence transformations.
At the same time, the description of the equivalence groupoid of the class~$\mathcal K$ gives
much more information about transformational properties of~$\mathcal K$
than the description of the equivalence group and allows one to properly select the optimal way
for classifying Lie symmetries of systems from the class~$\mathcal K$.

In contrast to the classical Lie approach,
the initial step of our consideration is to gauge the arbitrary elements of the class~$\bar{\mathcal L}$
by a known wide family of admissible transformations from the action groupoid of the equivalence group of~$\bar{\mathcal L}$
before exhaustively describing both this group and the equivalence groupoid of~$\bar{\mathcal L}$.
We use the obvious fact that
any transformation of the form~\eqref{eq:EquivGroupbarL} is an equivalence transformation of the class~$\bar{\mathcal L}$.
This fact can be directly checked and allows us to successively impose the gauges
$\boldsymbol f=\boldsymbol0$, $A=0$ and $\mathop{\rm tr}B=0$.
After each of the first two gauges,
we exclude~$\boldsymbol f$ and then~$A$ from the tuple of arbitrary elements of the singled out subclasses,
thus reparameterizing them to classes that are more convenient for studying.
As a result of gauging and reparameterizing, we construct
the class chain $\bar{\mathcal L}\hookleftarrow\mathcal L\hookleftarrow\mathcal L'\supset\mathcal L''$.

It turned out that the most suitable for computing admissible transformations is the class~$\mathcal L'$,
whose arbitrary-element tuple is the single matrix-valued function~$V$.
We have singled out the subclass~$\mathcal L'_0$ of~$\mathcal L'$ by the constraint
that the arbitrary-element matrix~$V$ is proportional to the identity matrix~$E$ with time-dependent proportionality factor.
This constraint can be written down as a system of algebraic equations with respect to the entries of~$V$.
Non-differential nature of the constraint leads to certain complications
when computing and presenting equivalence transformations within the subclass~$\mathcal L'_0$,
see Section~\ref{sec:AlgebraicAuxiliaryEqsAndClassReparameterization}.
Later, we have shown that the subclass~$\mathcal L'_0$
coincides with the $G^\sim_{\mathcal L'}$-orbit of the elementary system $\boldsymbol x_{tt}=\boldsymbol0$,
and thus it is the singular part of the class~$\mathcal L'$
whereas its complement~$\mathcal L'_1$ is the regular one.
We have constructed the equivalence groupoids and the equivalence groups of the class~$\mathcal L'$
and its subclasses~$\mathcal L'_0$ and~$\mathcal L'_1$.
The equivalence group of the subclass~$\mathcal L'_1$ and
the canonical significant equivalence group of the subclass~$\mathcal L'_0$
coincide with the equivalence group~$G^\sim_{\mathcal L'}$ of the entire class~$\mathcal L'$.
The partition $\mathcal L'=\mathcal L'_0\sqcup\mathcal L'_1$ of the class~$\mathcal L'$
induces the partition
\[\mathcal G^\sim_{\mathcal L'}=\mathcal G^\sim_{\mathcal L'_0}\sqcup\mathcal G^\sim_{\mathcal L'_1}\]
of its equivalence groupoid~$\mathcal G^\sim_{\mathcal L'}$.
As the regular subclass of~$\mathcal L'$, the subclass~$\mathcal L'_1$
is uniformly semi-normalized with respect to linear superposition of solutions.
The equivalence groupoid~\smash{$\mathcal G^\sim_{\mathcal L'_0}$} of~$\mathcal L'_0$ is the Frobenius product
of the vertex group~\smash{$\mathcal G_0$} of the elementary system~$L'_0$
and the restriction of the action groupoid of~$G^\sim_{\mathcal L'}$ on~$\mathcal L'_0$,
which is a particular case of semi-normalization in the usual sense.
The results obtained for the class~$\mathcal L'$ have been extended
to the classes~$\bar{\mathcal L}$ and~$\mathcal L$
via pulling back by the mappings of these classes on~$\mathcal L'$
that are generated by the above gaugings and reparameterizations.
Analogous results for the class~$\mathcal L''$ have been obtained
by modifying the corresponding proofs for~$\mathcal L'$
via taking into account the constraint $\mathop{\rm tr}V=0$,
which singles out~$\mathcal L''$ as a subclass of~$\mathcal L'$.
Since the constraints on~$V$ that are associated with~$\mathcal L''$, $\mathcal L''_0$ and~$\mathcal L''_1$
as subclasses of~$\mathcal L'$ are algebraic, i.e., of non-differential nature,
the consideration of essential equivalence groups instead of the entire equivalence groups
is relevant for these subclasses as for~$\mathcal L'_0$.
The equivalence algebra of an appropriate kind for each of the above classes is easily computed as the set of vector fields
(on the underlying space run by the collection of the independent and dependent variables and the arbitrary elements)
that are infinitesimal generators of one-parameter subgroups of the associated equivalence group.

In view of the partitions of the corresponding equivalence groupoids,
the solution of the group classification problem for each class
from the chain $\bar{\mathcal L}\hookleftarrow\mathcal L\hookleftarrow\mathcal L'\supset\mathcal L''$
splits into two independent parts for its singular and regular subclasses.
The group classification problem for the singular subclass is trivial,
and its solution is given by the elementary system with the maximal Lie invariance algebra~$\mathfrak g_0$
presented in~\eqref{eq:SLODEMIAOfElementarySystem}.
For systems from the regular subclass, the notion of essential Lie invariance algebra is well defined,
and it suffices to classify such algebras instead of the corresponding maximal Lie invariance algebras.
The class~$\mathcal L'_1$ is optimal for both carrying out the classification
and presenting the final classification results.
Due to uniform semi-normalization of all the regular subclasses and the consistency of their equivalence groups,
the classification results for~$\mathcal L'_1$ can be converted into
those for the classes~$\bar{\mathcal L}_1$, $\mathcal L_1$ and~$\mathcal L''_1$,
and the essential Lie invariance algebras of systems from each of the above classes are regular.

Using the equivariance of the natural projection~$\pi$ from $\mathbb F_t\times\mathbb F^n_{\boldsymbol x}$ onto $\mathbb F_t$
under the action of~$G^\sim_{\mathcal L'}$,
we have introduced the $G^\sim_{\mathcal L'}$-invariant nonnegative integer characteristic~$k_V$
for the maximal and the essential Lie invariance algebras~$\mathfrak g_V^{}$ and~$\mathfrak g^{\rm ess}_V$
of systems~$L'_V$ from the class~$\mathcal L'_1$,
${k=k_V:=\dim\pi_*\mathfrak g_V^{}=\dim\pi_*\mathfrak g^{\rm ess}_V}$.
The upper bound for~$k$ that is equal to three can be found via mapping~$\mathcal L'_1$ onto~$\mathcal L''_1$
by gauging~\cite{opan2022a} the arbitrary elements of~$\mathcal L'_1$ with a wide family of admissible transformations
from the action groupoid of the equivalence group~\smash{$G^\sim_{\mathcal L'_1}$} of~$\mathcal L'_1$,
see Corollary~\ref{cor:MAIofL''_V}.
Since for any system~$L'_V$, $\pi_*\mathfrak g^{\rm ess}_V$
is a finite-dimensional Lie algebra of vector fields on the $t$-line,
Lie's theorem on such algebras has efficiently been applied here
for listing the possible $t$-projections of the essential Lie invariance algebras of systems from the class~$\mathcal L'_1$
up to the $G^\sim_{\mathcal L'}$-equivalence.
We have separately studied each of the cases $k=0$, $k=1$ and $k=2$,
describing the structure of the corresponding essential Lie invariance algebras
and developing techniques for their classification.
Based on the results obtained, we have proved that $k$ cannot be equal to three.
The complete group classification for a fixed~$n$ requires
classifying subalgebras of the algebra $\mathfrak{sl}(n,\mathbb F)$.
Complete lists of ${\rm SL}(n,\mathbb F)$-inequivalent subalgebras of~$\mathfrak{sl}(n,\mathbb F)$
were constructed only for $n\in\{2,3\}$.
Using the well-known list for $n=2$, we have carried out the complete group classification
of normal linear systems of second-order ordinary differential equations
with two dependent variables.
Even this specific problem was not properly and exhaustively studied in the literature,
in spite of a number of papers devoted to its solution.
The revisited version of Winternitz's list~\cite{wint2004a} of subalgebras of~$\mathfrak{sl}(3,\mathbb F)$
from~\cite{chap2024a} forms a basis for the analogous classification in the case of three dependent variables.
Nevertheless, this classification is still nontrivial and cumbersome.
It requires a separate consideration and will be a subject of our future studies.
The classification of subalgebras~$\mathfrak s$ of~$\mathfrak{sl}(n,\mathbb F)$,
even satisfying the condition $\mathfrak s={\rm C}_{\mathfrak{sl}(n,\mathbb F)}({\rm C}_{\mathfrak{sl}(n,\mathbb F)}(\mathfrak s))$,
for a general value of~$n$ seems to be a wild problem,
and thus it looks hopeless to completely classify Lie symmetries
of normal linear systems of second-order ordinary differential equations
with an arbitrary number of dependent variables.

For each class in the chain $\bar{\mathcal L}\hookleftarrow\mathcal L\hookleftarrow\mathcal L'\supset\mathcal L''$,
the estimates for the dimensions of the maximal Lie invariance algebras of systems
from its regular subclass in Theorem~\ref{thm:MaxDimOfMIAinL1} can be combined with
the description of its singular subclass in Section~\ref{sec:SLODEEquivGroupoidsAndEquivGroups}
and with the discussion related to the equation~\eqref{eq:SLODEMIAOfElementarySystem}
into an assertion on the dimensions of the maximal Lie invariance algebras of systems
from the entire class.

\begin{theorem}\label{thm:MaxDimOfMIAinL}
For any system from the class~$\bar{\mathcal L}$ (resp.\ $\mathcal L$, $\mathcal L'$ or $\mathcal L''$),
the dimension of its maximal Lie invariance algebra is greater than or equal to $2n+1$ and less than or equal to $(n+2)^2-1$.
The lower bound is greatest and is attained for a general system of the class.
The upper bound is least, and it is attained only for the systems from the orbit
of the elementary system $\boldsymbol x_{tt}=\boldsymbol0$ with respect to the corresponding equivalence group.
The submaximum dimension $n^2+4$, which  is maximum for the systems from the complement to the above orbit,
is attained only for the systems that are equivalent to the system~\eqref{eq:CanonicalSystemOfMaxDimOfMIAinL1}
with respect to the corresponding equivalence group.
\end{theorem}

We have developed symmetry-based approaches for order reduction and integration
of normal linear systems of $n$ second-order ordinary differential equations.
These approaches essentially involve
the explicit description of the admissible point transformations between such systems.
The second required ingredient for integrating the systems from the $G^\sim_{\bar{\mathcal L}}$-orbit~$\bar{\mathcal L}_0$
of the elementary system $\boldsymbol x_{tt}=\boldsymbol0$,
which is the constraint for~$\theta$ singling out~$\bar{\mathcal L}_0$ as a subclass of~$\bar{\mathcal L}$,
is well known.
The necessary prerequisite for applying symmetry-based approaches to order reduction and integration
of a system~$\bar L_\theta$ from the complement subclass~$\bar{\mathcal L}_1$ to~$\bar{\mathcal L}_0$ in~$\bar{\mathcal L}$
is that the system~$\bar L_\theta$ admits Lie-symmetry vector fields with nonzero $t$-components
and, moreover, these vector fields are known.
Under additional conditions for such vector fields,
the corresponding systems can be completely integrated merely using algebraic operations and quadratures.

Although the class~$\bar{\mathcal L}=:\bar{\mathcal L}_{2,n}$ is singular
among the classes~$\bar{\mathcal L}_{r,n}$ of normal systems of $n$ ordinary differential equations of the same order~$r\in\{2,3,\dots\}$
due to the singularity of the elementary system ${\rm d}^r\boldsymbol x/{\rm d}t^r=\boldsymbol0$ with $r=2$,
the results of the present paper can be extended to an arbitrary value of~$r$.
It is also insistent to carry out the group classification of the class~$\bar{\mathcal L}_{r,n}$ at least for $(r,n)=(3,2)$
in the same manner as it has been done for $(r,n)=(2,2)$ in the present paper.
However, classifying Lie symmetries of systems of ordinary differential equations of mixed order
is quite a different matter, cf.\ \cite{doub2014a}.

\appendix

\section{Algebraic auxiliary equations and class reparameterization}\label{sec:AlgebraicAuxiliaryEqsAndClassReparameterization}

\newcommand{\EqOrd}{r}
\newcommand{\DerOrder}[1]{{\mbox{\tiny$($\scriptsize$#1$\tiny$)$}}}

If the definition of a class of differential equations~$\mathcal K$ includes
algebraic auxiliary equations with respect to the class arbitrary elements,
where ``algebraic'' means opposite to ``being truly differential'',
then both the usual and generalized equivalence groups of~$\mathcal K$
have certain specific features.
Here we consider these features only for the usual equivalence group~$G^\sim_{\mathcal K}$ of~$\mathcal K$
using a notation different from the notation in other sections of the paper.
Note that classes of the above kind are often arise as subclasses of a superclass
in course of considering hierarchies of subclasses with normalization properties and/or
gauging~\cite{opan2022a} the arbitrary elements of the superclass by wide families of admissible transformations
from the action groupoid of its equivalence group.
For required notions and results on equivalence groups of classes of differential equations,
see~\cite{bihl2012b,opan2017a,popo2010a,vane2020b}.

Consider a parametric form $\mathsf K_\theta$: $K\big(x,u_{\DerOrder\EqOrd},\theta_{\DerOrder q}(x,u_{\DerOrder\EqOrd})\big)=0$
of systems of differential equations
with $n$ independent variables $x=(x_1,\dots,x_n)$ and $m$ dependent variables $u=(u^1,\dots,u^m)$.
The notation~$u_{\DerOrder\EqOrd}$ is used for the tuple of derivatives of~$u$ with respect to $x$ up to order~$\EqOrd$,
by convention including $u$ as the derivatives of order zero.
The collection of equation left-hand sides~$K$ is a tuple of differential functions
in~$u$ and in the tuple of parametric functions
$\theta=\big(\theta^1(x,u_{\DerOrder\EqOrd}),\dots,\theta^k(x,u_{\DerOrder\EqOrd})\big)$,
which themselves are differential functions in~$u$ and are called the arbitrary elements.
Thus, the tuple~$\theta_{\DerOrder q}$ is constituted by the partial derivatives of~$\theta$
up to order $q$ with respect to the arguments~$x$ and~$u_{\DerOrder\EqOrd}$ of~$\theta$.
By~$\mathcal S_1$, $\mathcal S_2$ and~$\mathcal S_3$ we respectively denote the solution sets of auxiliary systems
$\mathsf S_1$, $\mathsf S_2$ and~$\mathsf S_3$
\begin{itemize}\itemsep=0ex
\item
of differential equations $S_1^{i_1}(x,u_{\DerOrder\EqOrd},\theta_{\DerOrder{\hat q}})=0$,
\item
of algebraic equations $S_2^{i_2}(x,u_{\DerOrder\EqOrd},\theta)=0$, and
\item
of differential and algebraic inequalities $S_3^{i_3}(x,u_{\DerOrder\EqOrd},\theta_{\DerOrder{\check q}})\ne0$ ($>0,<0,\dots$)
\end{itemize}
in~$\theta$.
Here, both~$x$ and~$u_{\DerOrder\EqOrd}$ play the role of the independent variables,
the index $i_j$ runs through a finite index set~$I_j$, $j=1,2,3$, $|I_2|\leqslant\#\theta=k$, and
the function tuple $S_2:=(S_2^{i_2},i_2\in I_2)$ is of maximal rank with respect to~$\theta$.
The class $\mathcal K:=\{\mathsf K_\theta\mid\theta\in\mathcal S_1\cap\mathcal S_2\cap\mathcal S_3\}$
is a subclass of the class $\bar{\mathcal K}:=\{\mathsf K_\theta\mid\theta\in\mathcal S_1\cap\mathcal S_3\}$,
which is singled out from~$\bar{\mathcal K}$ by the system of algebraic constraints~$S_2^{i_2}=0$, $i_2\in I_2$.

Taking into account the condition of maximal rank for~$S_2$,
we split the arbitrary-element tuple~$\theta$ into ``parametric'' and ``leading'' parts,
$\theta'$ and~$\theta''$ with $\#\theta''=|I_2|$ and $|\p S_2/\p \theta''|\ne0$,
and (locally) solve the system $S_2=0$ with respect to~$\theta''$,
$\theta''=\mathit\Psi(x,u_{\DerOrder\EqOrd},\theta')$.
This means that without loss of generality we can assume $S_2=\theta''-\mathit\Psi(x,u_{\DerOrder\EqOrd},\theta')$.
Denote by~$\mathsf\Psi$ the substitution for $\theta''$ in view of the system~$\mathsf S_2$,
$\mathsf\Psi\colon\theta''=\mathit\Psi(x,u_{\DerOrder\EqOrd},\theta')$.

Another possibility for simplifying the system~$\mathsf S_2$
is to ``straighten'' via reparameterizing the classes~$\bar{\mathcal K}$ and~$\mathcal K$.
The new arbitrary-element tuple~$\hat\theta$ is defined by
$\hat\theta'=\theta'$, $\hat\theta''=S_2(x,u_{\DerOrder\EqOrd},\theta)$.
The reparameterized classes~$\mathcal K^{\rm s}$ and~$\bar{\mathcal K}^{\rm s}$
are respectively similar to the classes~$\bar{\mathcal K}$ and~$\mathcal K$
with respect to the point transformation
$\tilde x=x$, $\tilde u_{\DerOrder\EqOrd}=u_{\DerOrder\EqOrd}$, $\hat\theta'=\theta'$, $\hat\theta''=S_2(x,u_{\DerOrder\EqOrd},\theta)$
in the space with coordinates $(x,u_{\DerOrder\EqOrd},\theta)$.
The system $\mathsf S_2^{\rm s}$ singling out~$\mathcal K^{\rm s}$ from~$\bar{\mathcal K}^{\rm s}$ as a subclass
is \smash{$\hat\theta''=0$}.
Similar classes have similar transformational properties.
In particular, their equivalence groups (resp.\ their equivalence groupoids) are similar.
Therefore, up to the class similarity, we can take $\theta''=0$ as the system~$\mathsf S_2$.

It is obvious that any point transformation of the form
\begin{gather}\label{eq:GenFormOfInessEquivTrans}
\tilde x=x,\quad \tilde u_{\DerOrder\EqOrd}=u_{\DerOrder\EqOrd},\quad \tilde\theta=\theta+F(x,u_{\DerOrder\EqOrd},\theta),
\end{gather}
where the components of the function $k$-tuple~$F$ are arbitrary differential functions of~$\theta$
vanishing on the set~$\mathcal S_2$ of solutions of the auxiliary system~$\mathsf S_2$,
$F|_{\mathcal S_2}^{}\equiv0$, and $\big|\p(\theta+F)/\p\theta\big|\ne0$,
is formally a (usual) equivalence transformation of the class~$\mathcal K$.
At the same time, the action of any transformation of the form~\eqref{eq:GenFormOfInessEquivTrans}
on the systems from the class $\mathcal K$ coincides with
that of the identity transformation $(\tilde x,\tilde u_{\DerOrder\EqOrd},\tilde\theta)=(x,u_{\DerOrder\EqOrd},\theta)$.
This is why we call the transformations of the form~\eqref{eq:GenFormOfInessEquivTrans}
the insignificant equivalence transformations of the class $\mathcal K$,
and the group constituted by these transformations the \emph{insignificant equivalence group} of the class~$\mathcal K$,
which is denoted by~$G^{{\rm i}\sim}_{\mathcal K}$.
We can show that the group~$G^{{\rm i}\sim}_{\mathcal K}$
is a normal subgroup of the gauge equivalence group~$G^{{\rm g}\sim}_{\mathcal K}$ of~$\mathcal K$
and, moreover, of the entire equivalence group~$G^\sim_{\mathcal K}$ of~$\mathcal K$,
\smash{$G^{{\rm i}\sim}_{\mathcal K}\triangleleft G^{{\rm g}\sim}_{\mathcal K}\triangleleft G^\sim_{\mathcal K}$}
and \smash{$G^{{\rm i}\sim}_{\mathcal K}\triangleleft G^\sim_{\mathcal K}$}.

There necessarily exists a subgroup of~$G^\sim_{\mathcal K}$
formed via a one-to-one selection of representatives from the cosets of~$G^{{\rm i}\sim}_{\mathcal K}$.
We call such a subgroup a \emph{significant equivalence group} of the class $\mathcal K$
and denote it by~$G^{{\rm s\vphantom{g}}\sim}_{\mathcal K}$.
In other words, we have the following assertion.

\begin{theorem}\label{thm:OnSignificantEquivGroups}
The group~$G^\sim_{\mathcal K}$ splits over~$G^{{\rm i}\sim}_{\mathcal K}$, i.e.,
it is the semidirect product of a subgroup~$G^{{\rm s\vphantom{g}}\sim}_{\mathcal K}$ acting on~$G^{{\rm i}\sim}_{\mathcal K}$,
$G^\sim_{\mathcal K}=G^{{\rm s\vphantom{g}}\sim}_{\mathcal K}\ltimes G^{{\rm i}\sim}_{\mathcal K}$.
\end{theorem}

\begin{proof}
Up to the class similarity, we assume the system~$\mathsf S_2$ to be $\theta''=0$.
An arbitrary transformation $\mathscr T\in G^\sim_{\mathcal K}$,
\[
\mathscr T\colon\ \
\tilde x=X(x,u),\ \
\tilde u_{\DerOrder\EqOrd}=U^{\DerOrder{\EqOrd}}(x,u_{\DerOrder\EqOrd}),\ \
\tilde\theta'=\Theta'(x,u_{\DerOrder\EqOrd},\theta),\ \
\tilde\theta''=\Theta''(x,u_{\DerOrder\EqOrd},\theta),
\]
preserves the system~$\mathsf S_2$.
This implies that $\Theta''|_{\theta''=0}=0$ identically with respect to~$\theta'$,
and thus $(\p\Theta''\!/\p\theta')|_{\theta''=0}^{}\!=\p(\Theta''|_{\theta''=0}^{})/\p\theta'\!=0$.
Hence  $\big|(\p\Theta'\!/\p\theta')|_{\theta''=0}^{}\big|\ne0$
since \mbox{$\big|\p(\Theta',\Theta'')/\p(\theta',\theta'')\big|\ne0$}.
This means that
\[
\mathsf\Psi_*\mathscr T\colon\ \
\tilde x=X(x,u),\ \
\tilde u_{\DerOrder\EqOrd}=U^{\DerOrder{\EqOrd}}(x,u_{\DerOrder\EqOrd}),\ \
\tilde\theta'=\Theta'(x,u_{\DerOrder\EqOrd},\theta)|_{\theta''=0}^{},\ \
\tilde\theta''=\theta''
\]
is a well-defined point transformation in the space with coordinates $(x,u_{\DerOrder\EqOrd},\theta)$.
One can show that $\mathsf\Psi_*(\mathscr T_1\circ\mathscr T_2)=(\mathsf\Psi_*\mathscr T_1)\circ(\mathsf\Psi_*\mathscr T_2)$
for any $\mathscr T_1,\mathscr T_2\in G^\sim_{\mathcal K}$.
In addition, for any $\mathscr T\in G^\sim_{\mathcal K}$ one have
$(\Theta')^{-1}\circ(\Theta'|_{\theta''=0})=\theta'$ on~$\mathcal S_2$,
where $(\Theta')^{-1}$ is the inverse of~$\Theta'$ with respect to~$\theta'$,
i.e., \mbox{$\mathscr T^{-1}\circ\mathsf\Psi_*\mathscr T\in G^{{\rm i}\sim}_{\mathcal K}$}.
Therefore, $\mathsf\Psi_*\mathscr T$ belongs to the coset $\mathscr T\circ G^{{\rm i}\sim}_{\mathcal K}$
of $G^{{\rm i}\sim}_{\mathcal K}$ in~$G^\sim_{\mathcal K}$.
Varying~$\mathscr T$ within~$G^\sim_{\mathcal K}$,
we obtain the map $\mathsf\Psi_*\colon G^\sim_{\mathcal K}\to G^\sim_{\mathcal K}$ with $\mathscr T\mapsto\mathsf\Psi_*\mathscr T$,
which a group homomorphism.
The transformation~$\mathsf\Psi_*\mathscr T$ is the identity transformation, ${\rm id}$,
in the space with coordinates $(x,u_{\DerOrder\EqOrd},\theta)$
if and only if $\mathscr T\in G^{{\rm i}\sim}_{\mathcal K}$, i.e., $\ker\mathsf\Psi_*=G^{{\rm i}\sim}_{\mathcal K}$.
The image $\mathsf\Psi_*G^\sim_{\mathcal K}$ of~$G^\sim_{\mathcal K}$ under the homomorphism~$\mathsf\Psi_*$
is a subgroup of~$G^\sim_{\mathcal K}$ and trivially intersects~$G^{{\rm i}\sim}_{\mathcal K}$,
$(\mathsf\Psi_*G^\sim_{\mathcal K})\cap G^{{\rm i}\sim}_{\mathcal K}=\{{\rm id}\}$.
Hence
$G^\sim_{\mathcal K}=(\mathsf\Psi_*G^\sim_{\mathcal K})\ltimes G^{{\rm i}\sim}_{\mathcal K}$,
i.e., $G^{{\rm s\vphantom{g}}\sim}_{\mathcal K}:=\mathsf\Psi_*G^\sim_{\mathcal K}$
is a significant equivalence group of the class~$\mathcal K$.
\end{proof}

\begin{remark}\label{rem:OnMappingForSignificantEquivGroups}
If the system~$\mathsf S_2$ is of the more general form $\theta''=\mathit\Psi(x,u,\theta')$,
then the map $\mathsf\Psi_*\colon G^\sim_{\mathcal K}\to G^\sim_{\mathcal K}$ is defined by
$G^\sim_{\mathcal K}\ni\mathscr T\mapsto\mathsf\Psi_*\mathscr T$,
where the components of the transformation~$\mathsf\Psi_*\mathscr T$~are
\begin{subequations}\label{eq:SignificantEquivGroupPsiTrans}
\begin{gather}\label{eq:SignificantEquivGroupPsiTransA}
\tilde x=X(x,u),\ \
\tilde u_{\DerOrder\EqOrd}=U^{\DerOrder{\EqOrd}}(x,u_{\DerOrder\EqOrd}),\ \
\tilde\theta'=\hat\Theta'(x,u_{\DerOrder\EqOrd},\theta'),
\\ \label{eq:SignificantEquivGroupPsiTransB}
\tilde\theta''=\theta''-\mathit\Psi(x,u,\theta')+\mathit\Psi\big(X(x,u),U^{\DerOrder{\EqOrd}}(x,u_{\DerOrder\EqOrd}),\hat\Theta'(x,u_{\DerOrder\EqOrd},\theta')\big),
\end{gather}
where $\hat\Theta'(x,u_{\DerOrder\EqOrd},\theta'):=\Theta'(x,u_{\DerOrder\EqOrd},\theta)|_{\theta''=\mathit\Psi(x,u,\theta')}^{}$.
\end{subequations}
\end{remark}

It is clear that the significant equivalence group of the class~$\mathcal K$
obtained in the proof of Theorem~\ref{thm:OnSignificantEquivGroups} is in general not unique
since the construction of the homomorphism~$\mathsf\Psi_*$ depends on
splitting the arbitrary-element tuple~$\theta$ into ``parametric'' and ``leading'' parts.
If the intersection $G^\sim_{\mathcal K}\cap G^\sim_{\bar{\mathcal K}}$ is a significant equivalence group of the class $\mathcal K$,
then we call it the \emph{canonical significant equivalence group} of the class~$\mathcal K$.

To factor out the insignificant equivalence group~$G^{{\rm i}\sim}_{\mathcal K}$
from the entire equivalence group~$G^\sim_{\mathcal K}$ of~$\mathcal K$, we reparameterize the class~$\mathcal K$
via making the substitution $\mathsf\Psi\colon\theta''=\mathit\Psi(x,u,\theta')$
for~$\theta''$ into~$\mathsf K_\theta$, $\mathsf S_1$ and~$\mathsf S_3$.
Thus, the required reparameterization~$\mathcal K_{\mathsf\Psi}$ of the class~$\mathcal K$ consists of
the systems~$\mathsf K_\theta|_{\mathsf\Psi}^{}$,
where the tuple~$\theta'$ is assumed as the arbitrary-element tuple
and runs through the joint solution set of
the auxiliary system of differential equations $\mathsf S_1|_{\mathsf\Psi}^{}$ and
the auxiliary system of differential and algebraic inequalities $\mathsf S_3|_{\mathsf\Psi}^{}$.
It is not correct to say that the class~$\mathcal K_{\mathsf\Psi}$ is a subclass of the class~$\bar{\mathcal K}$.
Instead of this, we write $\mathcal K_{\mathsf\Psi}\hookrightarrow\bar{\mathcal K}$,
denoting that the class~$\mathcal K_{\mathsf\Psi}$ is associated, via the reparameterization in the above way,
with the subclass~$\mathcal K$ of the class~$\bar{\mathcal K}$.

One can prove that $G^\sim_{\mathcal K_{\mathsf\Psi}}=\mathsf\Pi_*(\mathsf\Psi_*G^\sim_{\mathcal K})$,
where $\mathsf\Pi$ is the natural projection from the space with coordinates $(x,u_{\DerOrder\EqOrd},\theta)$
onto the space with coordinates $(x,u_{\DerOrder\EqOrd},\theta')$.
Moreover, the projection~$\mathsf\Pi$ establishes an isomorphism
between the groups~$G^\sim_{\mathcal K_{\mathsf\Psi}}$ and $\mathsf\Psi_*G^\sim_{\mathcal K}$.
The inverse~$\mathsf\Pi_*^{-1}$ of~$\mathsf\Pi_*$ on~\smash{$G^\sim_{\mathcal K_{\mathsf\Psi}}$} is defined
in the following way.
For an arbitrary transformation~\smash{$\mathscr T'\in G^\sim_{\mathcal K_{\mathsf\Psi}}$},
whose components are of the form~\eqref{eq:SignificantEquivGroupPsiTransA},
the image $\mathsf\Pi_*^{-1}\mathscr T'$ is the transformation
whose components are of the form~\eqref{eq:SignificantEquivGroupPsiTrans}
with the same functions $X$, $U^{\DerOrder{\EqOrd}}$ and~$\hat\Theta'$,
which definitely belongs to~$\mathsf\Psi_*G^\sim_{\mathcal K}$.
In this context, the equation~\eqref{eq:SignificantEquivGroupPsiTransB} can be interpreted
as the prolongation of the transformation~$\mathscr T'$ on~$\theta''$.

\subsection*{Acknowledgments}

The authors thank Boris Doubrov, Boris Kruglikov, Serhii Koval, Michael Kunzinger, Dmytro Popovych and Galyna Popovych
for helpful discussions.
The authors also sincerely thank the three reviewers for their helpful suggestions and comments,
which led to essentially improving the presentation of results.
This research was supported in part
by a grant from the Simons Foundation (1290607, V.M.B. and O.V.L.),
by the Austrian Science Fund (FWF), projects P25064 and P28770 (R.O.P.),
by the Ministry of Education, Youth and Sports of the Czech Republic (M\v SMT \v CR) under RVO funding for I\v C47813059 (R.O.P.),
by the NAS of Ukraine under the projects 0116U003059 (V.M.B., O.V.L. and R.O.P.) and 0121U110543 (O.V.L.),
as well as the project of the National Research Foundation of Ukraine 2020.02/0089, state registration number 0120U104004 (O.V.L.).
The authors express deepest thanks to the Armed Forces of Ukraine and the civil Ukrainian people
for their bravery and courage in defense of peace and freedom in Europe and in the entire world from russism.

\end{document}